\documentclass[12pt,reqno]{amsart}
\usepackage[numbers, square]{natbib}
\usepackage{mathptmx}
\usepackage{amsmath}
\usepackage{amscd}
\usepackage{amssymb}
\usepackage{amsthm}
\usepackage{enumerate}
\usepackage{xspace}
\usepackage[all,tips]{xy}
\usepackage[dvips]{graphicx}
\usepackage{verbatim}
\usepackage{syntonly}
\usepackage{hyperref}
\usepackage{amsmath, amsthm, graphics, amssymb,fullpage,color, epsfig,url}
\usepackage{indentfirst}
\usepackage{color}
\usepackage{soul, amsaddr}
\usepackage{esint}

\theoremstyle{plain}
\newtheorem{thm}{Theorem}
\newtheorem{lem}{Lemma}

\newtheorem{defn}{Definition}
\newtheorem{cor}{Corollary}

\newtheorem{clm}{Claim}

\theoremstyle{definition}
\newtheorem*{rem}{Remark}

\newenvironment{pf}
{\begin{proof}} {\end{proof}}

\newcommand{\disp}{\displaystyle}

\DeclareMathOperator{\supp}{supp}
\DeclareMathOperator{\di}{div}

\DeclareMathOperator{\tr}{tr}
\DeclareMathOperator{\loc}{loc}

\newcommand{\eps}{\varepsilon}
\newcommand{\vp}{\varphi}

\newcommand{\al}{\alpha}
\newcommand{\be}{\beta}
\newcommand{\ga}{\gamma}
\newcommand{\de}{\delta}

\newcommand{\Ga}{\Gamma}
\newcommand{\te}{\theta}
\newcommand{\la}{\lambda}
\newcommand{\La}{\Lambda}
\newcommand{\om}{\omega}
\newcommand{\Om}{\Omega}
\newcommand{\si}{\sigma}
\newcommand{\Si}{\Sigma}

\newcommand{\ol}{\overline}

\newcommand{\nid}{\noindent}

\newcommand{\iny}{\infty}
\newcommand{\del}{ \partial}
\newcommand{\su}{\subset}
\newcommand{\LP}{\Delta}
\newcommand{\gr}{\nabla}

\newcommand{\norm}[1]{\left\vert \left\vert #1\right\vert\right\vert}

\newcommand{\abs}[1]{\left\vert#1\right\vert}
\newcommand{\set}[1]{\left\{#1\right\}}
\newcommand{\brac}[1]{\left[#1\right]}
\newcommand{\pr}[1]{\left( #1 \right) }
\newcommand{\pb}[1]{\left( #1 \right] }
\newcommand{\brp}[1]{\left[#1\right)}

\newcommand{\N}{\ensuremath{\mathbb{N}}}

\newcommand{\R}{\ensuremath{\mathbb{R}}}

\newcommand{\Keywords}[1]{\par\noindent 
{\small{\bf Keywords\/}: #1}}
\newcommand{\MSC}[1]{\par\noindent 
{\small{\bf Mathematics Subject Classification\/}: #1}}

\date{}
\begin{document}

\title{Landis' conjecture for general second order elliptic equations \\
with singular lower order terms in the plane}

\author[Davey]{Blair Davey}
\address{Department of Mathematics, City College of New York CUNY, New York, NY 10031, USA}
\email{bdavey@ccny.cuny.edu}
\thanks{Davey is supported in part by the Simons Foundation Grant number 430198.}
\author[Wang]{Jenn-Nan Wang}
\address{Institute of Applied Mathematical Sciences, NCTS, National Taiwan University, \\
Taipei 106, Taiwan}
\email{jnwang@math.ntu.edu.tw}
\thanks{Wang is supported in part by MOST 102-2115-M-002-009-MY3.}

\begin{abstract}
In this article, we study the order of vanishing and a quantitative form of Landis' conjecture in the plane for solutions to second-order elliptic equations with variable coefficients and singular lower order terms. 
Precisely, we let $A$ be real-valued, bounded and elliptic, but not necessary symmetric or continuous, and we assume that $V$ and $W_i$ are real-valued and belong to $L^p$ and $L^{q_i}$, respectively. 
We prove that if $u$ is a real-valued, bounded and normalized solution to an equation of the form $-\di \pr{A \gr u + W_1 u} + W_2 \cdot \gr u + V u = 0$ in $B_d$, then under suitable conditions on the lower order terms, for any $r$ sufficiently small, the following order of vanishing estimate holds
$$\norm{u}_{L^\iny\pr{B_r}} \ge r^{C M},$$
where $M$ depends on the Lebesgue norms of the lower order terms.
In a number of settings, a scaling argument gives rise to a quantitative form of Landis' conjecture,
\[
\inf_{\abs{z_0} = R} \norm{u}_{L^\iny\pr{B_1\pr{z_0}}} \ge \exp\pr{- C R^\be \log R},
\]
where $\be$ depends on $p$, $q_1$, and $q_2$. 
The integrability assumptions that we impose on $V$ and $W_i$ are nearly optimal in view of a scaling argument. 
We use the theory of elliptic boundary value problems to establish the existence of positive multipliers associated to the elliptic equation.
Then the proofs rely on transforming the equations to Beltrami systems and applying a generalization of Hadamard's three-circle theorem. \\

\Keywords{Landis' conjecture; quantitative unique continuation; order of vanishing; Beltrami system} \\

\MSC{35B60, 35J10} 
\end{abstract}

\maketitle

\section{Introduction}

In this paper, we study quantitative versions of Landis' conjecture for real-valued solutions to second-order, uniformly elliptic equations with singular lower order terms.
Over an open, connected $\Om \su \R^2$, define the second-order divergence-form operator
$$L :=  -\di\pr{A \gr },$$
where we assume that $A = \pr{a_{ij}}_{i, j = 1}^2$ is real-valued, measurable, and is not necessarily symmetric.
We also assume that $A$ is uniformly elliptic and bounded, i.e., there exist $\la \in \pb{0, 1}$, $\La > 0$ so that for every $z \in \Om$,
\begin{align}
&a_{ij}\pr{z} \xi^i \xi^j \ge \la  \abs{\xi}^2 \; \text{ for all } \xi \in \R^2,\label{ellip} \\
&\abs{a_{ij}\pr{z}} \le \La.  \label{ABd}
\end{align}
For real-valued $W_1, W_2, V$ belonging to appropriate Lebesgue spaces, 
we study the unique continuation properties of real-valued solutions to the following second-order elliptic equation in the plane:
\begin{equation}
- \di\pr{A \gr u + W_1 u} + W_2 \cdot \gr u + V u = 0.
\label{epde}
\end{equation}

We use the notation $B_r\pr{z_0}$ to denote the ball of radius $r$ centered at $z_0 \in \R^2$.
Often, we abbreviate this notation and simply write $B_r$ when the centre is understood from the context.
For the order of vanishing estimates, we consider solutions to \eqref{epde} in $B_d$, where $d$ is a constant to be specified below (see \eqref{drho}).
The constants $b, \tilde b$ will also be specified later on (see \eqref{bsigma} and \eqref{btsigma}), once we have introduced quasi-balls.
Quasi-balls are sets associated with the levels sets of fundamental solutions, and are therefore appropriate generalizations of standard balls to the variable coefficient setting.

Our first collection of theorems describe the order of vanishing for solutions to equations of the form \eqref{epde}.
First we consider very general equations with smallness and non-negativity conditions on the singular lower order terms.

\begin{thm}
\label{OofV}
Assume that  conditions \eqref{ellip} and \eqref{ABd} are satisfied, and that for some $q_1, q_2 \in (2, \iny]$, $p \in (1, \iny]$, $\norm{W_1}_{L^{q_1}\pr{B_d}} \le K$, $\norm{W_2}_{L^{q_2}\pr{B_d}} \le \min\set{ \frac \la {2 c_{q_2}}, \frac 1 {3 C_{q_2}}}$, and $\norm{V}_{L^p\pr{B_d}}\le \frac 1 {3 C_p}$, where $K \ge 1$ and $c_{q_2}, C_{q_2}, C_p \ge 1$ are specific constants.
Assume further that
\begin{align}
&\int W_1 \cdot \gr \phi \ge 0 \quad \text{for every} \; \phi \in W^{1,q_1'}\pr{\Om} \; \text{ such that} \; \phi \ge 0,
\label{pos1} \\
&\int W_2 \cdot \gr \phi + V \phi \ge 0 \quad \text{for every} \; \phi \in W^{1,q_2'}\pr{\Om} \cap L^{p'}\pr{\Om} \; \text{ such that} \; \phi \ge 0,
\label{pos2}
\end{align}
where $q_1',q_2',p'$ denote the conjugate exponent of $q_1,q_2,p$, respectively. Let $u$ be a real-valued solution to \eqref{epde} in $B_d$ that satisfies 
\begin{align}
& \norm{u}_{L^\iny\pr{B_d}} \le \exp\pr{C_0 K}
\label{localBd} \\
& \norm{u}_{L^\iny\pr{B_b}} \ge 1. 
\label{localNorm}
\end{align}
Then for any $r$ sufficiently small and any $\eps > 0$, 
\begin{equation}
\norm{u}_{L^\iny\pr{B_r}} \ge r^{C K^{1+\eps}},
\label{localEst}
\end{equation}
where $C$ depends on $\la$, $\La$, $q_1$, $q_2$, $p$, $C_0$, and $\eps$.
\end{thm}

In the case where $W_2, V \equiv 0$, the previous result holds with $\eps = 0$ in the absence of condition \eqref{pos1}.

\begin{thm}
\label{OofV1}
Assume that  conditions \eqref{ellip} and \eqref{ABd} are satisfied, and that for some $q \in [2, \iny]$, $\norm{W}_{L^{q}\pr{B_d}} \le K$.
Let $u$ be a real-valued solution to 
\begin{equation}
- \di\pr{A \gr u + W u}  = 0
\label{epde2}
\end{equation}
in $B_d$ that satisfies \eqref{localBd}. 
\begin{enumerate}
\item[{\rm (a)}] If $q > 2$ and $u$ satisfies \eqref{localNorm}, then for any $r$ sufficiently small, \eqref{localEst} holds with $\eps = 0$ and $C$ depending on $\la$, $\La$, $q$, and $C_0$.
\item[{\rm (b)}] If $q = 2$, then the strong unique continuation property holds.
\end{enumerate}
\end{thm}

When $W_1, V \equiv 0$ and we have additional assumptions on the coefficient matrix, we get another order of vanishing estimate.
Again, we do not require any smallness or non-negativity on the lower order terms.

\begin{thm}
\label{OofV2}
Let $A$ be symmetric and uniformly elliptic with Lipschitz continuous coefficients. 
That is, \eqref{ellip} and \eqref{ABd} hold, $a_{12} = a_{21}$, and as a consequence of Rademacher's theorem, there exists $\mu > 0$ such that
\begin{align}
\norm{\gr a_{ij}}_{L^\iny} \le \mu \quad \text{ for each } \;\; i, j = 1, 2.
\label{gradDec}
\end{align}
Assume that for some $q \in [2, \iny]$, $\norm{W}_{L^{q}\pr{B_d}} \le K$.
Let $u$ be a real-valued solution to 
\begin{equation}
- \di\pr{A \gr u} + W \cdot \gr u = 0.
\label{epde3}
\end{equation}
in $B_d$.
\begin{enumerate}
\item[{\rm (a)}] If $q > 2$ and $u$ satisfies \eqref{localBd} and \eqref{localNorm}, then for any $r$ sufficiently small, \eqref{localEst} holds with $\eps = 0$ and $C$ depending on $\la$, $\La$, $\mu$, $q$, and $C_0$.
\item[{\rm (b)}] If $q = 2$ and $u$ satisfies $\norm{u}_{L^\iny\pr{B_d}} \le M$, $\norm{\gr u}_{L^2\pr{B_{\tilde b}}} \ge 1$, then for any $r$ sufficiently small
\begin{equation}
\norm{u}_{L^\iny\pr{B_r}} \ge r^{C\pr{\log M + K^2}},
\label{localEst2}
\end{equation}
where $C$ depends on $\la$, $\La$, and $\mu$.
\end{enumerate}
\end{thm}

Finally, we consider a very general form of the equation without any smallness assumptions on the three lower order terms.
In this setting, we impose a comprehensive sign condition on the lower order terms.

\begin{thm}
\label{OofV3}
Let $A$ be a symmetric matrix for which \eqref{ellip}, \eqref{ABd}, and \eqref{gradDec} hold.
Assume that $\norm{W_1}_{L^{q}\pr{B_d}} \le K$ for some $q \in [2,\iny]$, $\norm{W_2}_{L^{\iny}\pr{B_d}} \le K$, and $\norm{V}_{L^{\iny}\pr{B_d}} \le K^2$.
Assume further that $W_1$ is weakly curl-free and that $V - W_1 \cdot W_2 \ge 0$ a.e.
Let $u$ be a real-valued solution to \eqref{epde} in $B_d$.
\begin{enumerate}
\item[{\rm (a)}] If $A = I$, $q = \iny$ and $u$ satisfies \eqref{localBd} with $d$ replaced by $9/5$ and \eqref{localNorm} with $b$ replaced by $1$, then for any $r$ sufficiently small, \eqref{localEst} holds with $\eps = 0$ and $C$ depending on $C_0$.
\item[{\rm (b)}] If $q = 2$, $W_2, V \equiv 0$, and $\norm{u}_{L^\iny\pr{B_d}} \le M$, $\norm{\gr u}_{L^2\pr{B_{\tilde b}}} \ge 1$, then for any $r$ sufficiently small, \eqref{localEst2} holds with $C$ depending on $\la$, $\La$, and $\mu$.
\end{enumerate}
\end{thm}

As usual, the order of vanishing estimates are used in combination with a scaling argument to prove the following quantitative unique continuation at infinity estimates.
Since the smallness conditions required in Theorem \ref{OofV} do not hold up under scaling, we do not have a corresponding Landis theorem for that type of equation.
However, the three other settings described by Theorems \ref{OofV1}, \ref{OofV2}, and \ref{OofV3} lead to Landis-type results.

\begin{thm}
\label{LandisThm}
Assume that conditions \eqref{ellip} and \eqref{ABd} hold, and that for some $q \in (2, \iny]$, $\norm{W}_{L^{q}\pr{\R^2}} \le \al$.
Let $u$ be a real-valued solution to \eqref{epde2} in $\R^2$ for which
\begin{align}
& \abs{u\pr{z}} \le \exp\pr{C_1 \abs{z}^{1 - \frac 2 {q}}}
\label{uBd} \\
& \abs{u\pr{0}} \ge 1.
\label{normed}
\end{align}
Then for any $R$ sufficiently large, we have
\begin{equation}
\inf_{\abs{z_0} = R} \norm{u}_{L^\iny\pr{B_1\pr{z_0}}} \ge \exp\pr{- C R^{1 - \frac 2 {q}} \log R},
\label{globalEst}
\end{equation}
where $C$ depends on $\la$, $\La$, $q$, $\al$, and $C_1$.
\end{thm}

\begin{thm}
\label{LandisThm2}
Let $A$ be a symmetric matrix for which \eqref{ellip}, \eqref{ABd}, and \eqref{gradDec} hold.
Assume that $\norm{W}_{L^{q}\pr{\R^2}} \le \al$ for some $q \in [2, \iny)$.
Let $u$ be a real-valued solution to \eqref{epde3} in $\R^2$.
\begin{enumerate}
\item[{\rm (a)}] If $q > 2$ and \eqref{uBd} and \eqref{normed} hold, then for any $R$ sufficiently large, \eqref{globalEst} holds with $C$ depending on $\la$, $\La$, $\mu$, $q$, $\al$, and $C_1$.
\item[{\rm (b)}] If $q = 2$, there exists $m > 0$ so that $\abs{u\pr{z}} \le \abs{z}^m$ when $\abs{z} \ge 1$, and $\norm{\gr}_{L^2\pr{B_{1/2}}} \ge 1$, then
\begin{equation}
\inf_{\abs{z_0} = R} \norm{u}_{L^\iny\pr{B_1\pr{z_0}}} \ge \exp\pr{- C \pr{\log R}^2},
\label{globalEst2}
\end{equation}
where $C$ depends on $\la$, $\La$, $\mu$, $\al$, and $m$.
\end{enumerate}
\end{thm}

\begin{thm}
\label{LandisThm3}
Let $A$ be a symmetric matrix for which \eqref{ellip}, \eqref{ABd}, and \eqref{gradDec} hold.
Assume that $\norm{W_1}_{L^{q}\pr{\R^2}} \le \al_1$ for some $q \in \brac{2, \iny}$, $\norm{W_2}_{L^{\iny}\pr{\R^2}} \le \al_2$, $\norm{V}_{L^{\iny}\pr{\R^2}} \le \al_0$, where $W_1$ is weakly curl-free and $V - W_1 \cdot W_2 \ge 0$ a.e.
Let $u$ be a real-valued solution to \eqref{epde} in $\R^2$.
\begin{itemize}
\item[(a)] If $A = I$, $q = \iny$ and \eqref{uBd} and \eqref{normed} hold, then for any $R$ sufficiently large, \eqref{globalEst} holds with $C$ depending on $\al_1$, $\al_2$, $\al_0$, and $C_1$.
\item[(b)] If $q = 2$, $W_2, V \equiv 0$, there exists $m > 0$ so that $\abs{u\pr{z}} \le \abs{z}^m$ when $\abs{z} \ge 1$, and $\norm{\gr}_{L^2\pr{B_{1/2}}} \ge 1$, then for any $R$ sufficiently large, \eqref{globalEst2} holds with $C$ depending on $\la$, $\La$, $\mu$, $\al_1$ and $m$.
\end{itemize}
\end{thm}

We point out now that the estimates in Theorems \ref{LandisThm} -- \ref{LandisThm3} are (almost) sharp in an exterior domain.
Consider $u\pr{r} = \exp\pr{- r^\al}$ for some $\al \in \pr{0, 1}$ to be determined.
A computation gives
\begin{align*}
\gr u\pr{r} &= - \al r^{\al -1} \pr{\frac x r, \frac y r} u\pr{r} \\
\LP u\pr{r} &= \al^2 r^{2\pr{\al - 1}}\pr{1 - r^{-\al}} u\pr{r}.
\end{align*}
If we define
\begin{align*}
W &= - \al r^{\al-1} \pr{\frac x r, \frac y r}
\end{align*}
then
$$\di \pr{\gr u + W u} = 0.$$
If $q = \iny$, set $\al = 1$ and note that $W \in L^\iny\pr{\R^2 \setminus B_1}$.
Otherwise, if $q \in \pr{2, \iny}$, then for any $\de \in \pr{0, q - 2}$, let $\al = 1 - \frac {2 + 2\de} {q_1}$ and we see that
\begin{align*}
\norm{W}^{q}_{L^{q}\pr{\R^2 \setminus B_1}}
&\le C \int_1^\iny r^{-\pr{2 + 2\de}} r \, dr
= \frac C {2\de} 
< \iny.
\end{align*}
It follows that Theorem \ref{LandisThm} is almost sharp in $\R^2 \setminus B_1$ with an arbitrarily small error.

With
\begin{align*}
W &= - \al r^{\al-1} \pr{1 - r^{-\al}} \pr{\frac x r, \frac y r},
\end{align*}
we have
$$- \di \pr{\gr u} + W \cdot \gr u = 0.$$
Defining $q$ as before, we see that Theorem \ref{LandisThm2} is almost sharp in $\R^2 \setminus B_1$ with an arbitrarily small error.

Finally, if we set $u = \exp\pr{-r}$ and define
\begin{align*}
V &= \frac 1 3 \pr{1 - r^{-1}} \\
W_1 &= \frac 1 3 \pr{\frac x r, \frac y r}  \\
W_2 &= - \frac 1 3 \pr{1 - r^{-1}} \pr{\frac x r, \frac y r},
\end{align*}
then $-\di\pr{\gr u + W_1 u} + W_2 \cdot \gr u + V u = 0$, $W_1, W_2, V \in L^\iny$ and $V - W_1 \cdot W_2 \ge 0$.  
Moreover, $W_1$ is curl-free.
Therefore, Theorem \ref{LandisThm3} is sharp in an exterior domain.

Since we are working with real-valued solutions and equations in the plane, the best approach to proving these theorems is to the use the relationship between our solutions and the solutions to first order equations in the complex plane, Beltrami systems.
Therefore, we will closely follow the proof ideas that were first developed in \cite{KSW15}, with further generalizations in \cite{DKW17}.
Since we are no longer working with bounded lower order terms, but rather with singular potentials, we also borrow some of the ideas that were presented in \cite{KW15} where the authors considered drift equations with singular potentials.

Our results generalize those previously established in \cite{KSW15}, \cite{KW15}, and \cite{DKW17} in a few ways.
First, in a couple of the settings that we consider, the leading operator is no longer assumed to be Lipschitz continuous and symmetric as in \cite{DKW17}.
(In \cite{KSW15} and \cite{KW15}, the leading operator is the Laplacian.)
In those cases, we only assume that $A$ is bounded and uniformly elliptic.
Second, we consider when all of the lower order terms are unbounded.
In \cite{KSW15} and \cite{DKW17}, the two lower order terms, $V$ and $W$, are assumed to be bounded; whereas in \cite{KW15}, one lower order term, $W$, can be unbounded, but $V$ has to be zero.
One of our settings deals with equations that have three singular lower order terms.
Third, we consider some very general elliptic equations that can have two non-trivial first order terms.
In \cite{KSW15}, \cite{KW15}, and \cite{DKW17}, it is always assumed that either $W_1 \equiv 0$ or $W_2 \equiv 0$.

When the lower order terms are not assumed to be bounded, our approach to the construction of the positive multipliers is completely new in this article.  
We use the existence of solutions to Dirichlet boundary value problems in combination with the maximum principle to argue that positive multipliers with appropriate pointwise bounds exist.

We remark that our current methods do not apply to the scale-invariant case of $V \in L^1\pr{\R^2}$.
This is not surprising since the counterexample of Kenig and Nadirashvili in \cite{KN00} implies that weak unique continuation can fail for the operator $\LP + V$ with $V \in L^1$. 

A similar problem was investigated by the first-named author and Zhu in \cite{DZ17} and \cite{DZ217}.
In these papers, the authors studied the quantitative unique continuation properties of solutions to equations of the form \eqref{epde} under the assumption that $L = -\LP$, $W_1 \equiv 0$, and the other lower order terms belong to some admissible Lebesgue spaces.
Since the proof techniques are based on certain $L^p - L^q$ Carleman estimates, the results apply to complex-valued solutions and equations in any dimension $n \ge 2$.
Consequently, the estimates derived in \cite{DZ17} and \cite{DZ217} are not as sharp as those that we prove in the current paper.
For a broader survey of related works, we refer the reader to \cite{KSW15} and the references therein.

The organization of this article is as follows.
In Section \ref{S2}, we discuss quasi-balls.
Quasi-balls are a natural generalization of standard balls and they are associated to a uniformly elliptic divergence-form operator.
Section \ref{S3} deals with the positive multipliers.
In particular, we construct a positive multiplier, prove that it has appropriate pointwise bounds, satisfies generalized Caccioppoli-type inequalities, and then show that its logarithm also has good bounds in some $L^t$ spaces.
The Beltrami operators are introduced in Section \ref{S4}.
Much of this section resembles work that was previously done in \cite{DKW17}, and we therefore omit some of the proofs.
In Section \ref{S5}, we use the tools that have been developed to prove Theorem \ref{OofV}.
Sections \ref{S6}, \ref{S7}, and \ref{S8} treat the proofs of Theorems \ref{OofV1}, \ref{OofV2}, and \ref{OofV3}, respectively.
The proofs of Theorems \ref{LandisThm} -- \ref{LandisThm3} are presented in Section \ref{S9}.

In addition to the main content of this paper, we rely on some theory regarding elliptic boundary value problems, and this content has been relegated to the appendices.
In Appendix \ref{AppA}, we prove a maximum principle.
Appendix \ref{AppB} presents a collection of results regarding the Green's functions for general elliptic operators in open, bounded, connected subsets of $\R^2$.
This work is based on the constructions that appear in \cite{GW82}, \cite{HK07}, and \cite{DHM16}.
We include this section for completeness since the specific representation that we sought was not available in the literature.
The results of the appendices are used in Section \ref{S2} where we argue that the positive multipliers satisfy appropriate pointwise bounds. \\

\nid {\bf Acknowledgement.}
Part of this research was carried out while the first author was visiting the National Center for Theoretical Sciences (NCTS) at National Taiwan University.
The first author wishes to the thank the NCTS for their financial support and their kind hospitality during her visits to Taiwan.

\section{Quasi-balls}
\label{S2}

Since we are working with variable-coefficient operators instead of the Laplacian, we will at times need to work with sets that are not classical balls.
Therefore, we introduce the notion of quasi-balls.

Throughout this section, assume that $L := - \di\pr{A \gr}$ is a second-order divergence form operator acting on $\R^2$ that satisfies the ellipticity and boundedness conditions described by \eqref{ellip} and \eqref{ABd}.
Let $\mathbf{L}\pr{\la, \La}$ denote the set of all such operators.
We start by discussing the fundamental solutions of $L$.
These results are based on the Appendix of \cite{KN85}.

\begin{defn}
A function $G$ is called a fundamental solution for $L$ with pole at the origin if
\begin{itemize}
\item $G \in H^{1,2}_{loc}\pr{\R^2 \setminus \set{0}}$, $G \in H^{1,p}_{loc}\pr{\R^2}$ for all $p < 2$, and for every $\vp \in C^\iny_0\pr{\R^2}$
$$\int a_{ij}\pr{z} D_i G\pr{z} \,  D_j \vp\pr{z} dz = \vp\pr{0}.$$
\item $\abs{G\pr{z} } \le C \log \abs{z}$, for some $C > 0$, {$|z|\ge C$}.
\end{itemize}
\end{defn}

\begin{lem}[Theorem A-2, \cite{KN85}]
There exists a unique fundamental solution $G$ for $L$, with pole at the origin and with the property that $\disp \lim_{\abs{z} \to \iny} G\pr{z} - g\pr{z} = 0$, where $g$ is a solution to $L g = 0$ in $\abs{z} > 1$ with $g = 0$ on $\abs{z} = 1$.
Moreover, there are constants $C_1, C_2, C_3, C_4, R_1 < 1 < R_2$, that depend on $\la$ and $\La$, such that
\begin{align*}
&C_1 \log\pr{\frac{1}{\abs{z}}} \le - G\pr{z} \le C_2 \log \pr{\frac{1}{\abs{z}}} \;\; \text{ for } \abs{z} < R_1 \\
& C_3 \log\abs{z} \le G\pr{z} \le C_4 \log \abs{z} \;\; \text{ for } \abs{z} > R_2.
\end{align*}
\label{fundSolBds}
\end{lem}

The level sets of $G$ will be important to us.

\begin{defn}
Define a function $\ell: \R^2 \to \pr{0, \iny}$ as follows: $\ell\pr{z} = s$ iff $G\pr{z} = \ln s$.
Then set 
\begin{align*}
Z_s &= \set{ z \in \R^2 : G\pr{z} = \ln s} 
= \set{ z \in \R^2 : \ell\pr{z} = s} .
\end{align*}
We refer to these level sets of $G$ as {\bf quasi-circles.}
That is, $Z_s$ is the quasi-circle of radius $s$.
We also define (closed) {\bf quasi-balls} as
\begin{align*}
Q_s &= \set{ z \in \R^2 : \ell\pr{z} \le s} .
\end{align*}
Open {\bf quasi-balls} are defined analogously.
We may also use the notation $Q_s^L$ and $Z_s^L$ to remind ourselves of the underlying operator.
\end{defn}

The following lemma follows from the bounds given in Lemma \ref{fundSolBds}.
The details of the proof may be found in \cite{DKW17}.

\begin{lem}
There are constants $c_1, c_2, c_3, c_4, c_5, c_6, S_1 < 1 < S_2$, that depend on $\la$ and $\La$, such that if $z \in Z_s$, then
\begin{align*}
& s^{c_1} \le \abs{z} \le s^{c_2}  \;\; \text{ for } s \le S_1 \\
& c_5 s^{c_1} \le \abs{z} \le c_6 s^{c_4} \;\; \text{ for } S_1 < s < S_2 \\
& s^{c_3} \le \abs{z} \le s^{c_4} \;\; \text{ for } s \ge S_2.
\end{align*}
\label{ZsBounds}
\end{lem}

Thus, the quasi-circle $Z_s$ is contained in an annulus whose inner and outer radii depend on $s$, $\la$, and $\La$.
For future reference, it will be helpful to have a notation for the bounds on these inner and outer radii.

\begin{defn}
Define
\begin{align*}
& \si\pr{s; \la, \La} = \sup\set{ r > 0 : B_r \su \bigcap_{L \in \mathbf{L}\pr{\la, \La}} Q_s^L } \\
& \rho\pr{s; \la, \La} = \inf \set{r > 0 : \bigcup_{L \in \mathbf{L}\pr{\la, \La}} Q_s^L \su B_r }
\end{align*}
\end{defn}

\begin{rem}
These functions are defined so that for any operator $L$ in $\mathbf{L}\pr{\la, \La}$, $B_{\si\pr{s; \la, \La}} \su Q^L_s \su B_{\rho\pr{s;\la, \La}}$.
\end{rem}

The quasi-balls and quasi-circles just defined above are centered at the origin since $G$ is a fundamental solution with a pole at the origin.
As a reminder, we may sometimes use the notation $Z_s\pr{0}$ and $Q_s\pr{0}$.
If we follow the same process for any point $z_0 \in \R^2$, we may discuss the fundamental solutions with pole at $z_0$, and we may similarly define the quasi-circles and quasi-balls associated to these functions.
We denote the quasi-circle and quasi-ball of radius $s$ centred at $z_0$ by $Z_s\pr{z_0}$ and $Q_s\pr{z_0}$, respectively.
Although $Q_s\pr{z_0}$ is not necessarily a translation of $Q_s\pr{0}$ for $z_0 \ne 0$, both sets are contained in annuli that are translations.

\section{Positive multipliers}
\label{S3}

In \cite{KSW15} and \cite{DKW17}, the first step in the proofs of the order of vanishing estimates is to establish that a positive multiplier associated to the operator (or its adjoint) exists and has suitable bounds.
Unlike the settings in those papers, since our lower order terms are unbounded, we cannot simply construct positive super- and subsolutions in $B_d$, then argue that a positive solution exists.
Therefore, our approach here is more involved.
Instead, we use solutions to the Dirichlet boundary value problem for constant boundary data and rely on the maximum principle and Green's function representations to give us desirable bounds.

From now on, we set
\begin{equation}
\label{drho}
d = \rho\pr{7/5} + 2/5,
\end{equation}
where $\rho\pr{s} = \rho\pr{s; \la, \La}$ is as defined in the previous section. 
Throughout this section, assume that $A$ satisfies \eqref{ellip} and \eqref{ABd}, while $W_1 \in L^{q_1}\pr{B_d}$, $W_2 \in L^{q_2}\pr{B_d}$, and $V \in L^{p}\pr{B_d}$ for some $q_1, q_2 \in (2, \iny]$, $p \in (1, \iny]$. 
Note that $A^T$ also satisfies \eqref{ellip} and \eqref{ABd}. 
Associated to an operator of the type $\mathcal{L} := - \di\pr{A \gr + W_1} + W_2 \cdot \gr + V$ is the bilinear form $\mathcal{B} : W^{1,2}_0\pr{\Om} \times W^{1,2}_0\pr{\Om} \to \R$ given by
\begin{equation}
\mathcal B\brac{u, v} = \int_{\Om} A \gr u \cdot \gr v + W_1 u \cdot \gr v + W_2\cdot \gr u  \, v + V \, u \, v.
\label{biForm}
\end{equation}
In every case, we take $\Om = B_d$.

Since we need the existence of solutions to various elliptic equations, the following lemma serves as a useful tool.

\begin{lem}
\label{solvableLem}
Let $g \in C^1\pr{\overline{B_d}}$.
Assume that the bilinear form given by \eqref{biForm} is bounded and coercive in $W^{1,2}_0\pr{B_d}$.
That is, there exist constants $c$ and $C$ so that for any $u, v \in W^{1,2}_0\pr{B_d}$
\begin{align*}
&|\mathcal B\brac{u, v}| \le C \norm{u}_{W^{1,2}\pr{B_d}} \norm{v}_{W^{1,2}\pr{B_d}} \\
&\mathcal B\brac{v, v} \ge c \norm{v}_{W^{1,2}\pr{B_d}}^2.
\end{align*}
Then there exists a weak solution $\phi \in W^{1,2}\pr{B_d}$ to 
\begin{equation}
\left\{\begin{array}{rl} - \di \pr{A \gr \phi + W_1 \phi} + W_2 \cdot \gr \phi + V \phi = 0 & \text{ in }\; B_d \\ \phi = g & \text{ on }\; \del B_d \end{array} \right..
\label{BVP1}
\end{equation}
\end{lem}

\begin{proof}
To establish that a solution to \eqref{BVP1} exists, we prove that there exists a $\psi \in W^{1,2}_0\pr{B_d}$ for which
\begin{equation}
- \di \pr{A \gr \psi + W_1 \psi} + W_2 \cdot \gr \psi+ V \psi  = -\di G + f \;\text{ in }\; B_d,
\label{BVP*}
\end{equation}
where $G \in L^{q}\pr{B_d}$ and $f \in L^p\pr{B_d}$ for some $q \in \pb{2, \iny}$ and $p \in \pb{1, \iny}$.
With $\psi = \phi - g$, we have $G = -A \gr g - W_1 g$ and $f = - W_2 \cdot \gr g - V g$, and this gives the claimed result since $g \in C^1\pr{\overline{B_d}}$ and $B_d$ is bounded.

To show that \eqref{BVP*} is solvable, we need to show that for any $v \in W^{1,2}_0\pr{B_d}$, there exists a $\psi \in W^{1,2}_0\pr{B_d}$ for which
\begin{equation}
\mathcal B\brac{\psi, v} = \int_{B_d} G \cdot \gr v + f \, v.
\label{weakForm}
\end{equation}
For any $v \in W^{1,2}_0\pr{B_d}$, consider the linear functional
\begin{equation}
v \mapsto \int_{B_d} G \cdot \gr v + f \, v.
\label{functional}
\end{equation}
If $p \ge 2$, then
\begin{align*}
|\int_{B_d} G \cdot \gr v + f \, v|
&\le \norm{G}_{L^{q}\pr{B_d}} \norm{\gr v}_{L^2\pr{B_d}} \abs{B_d}^{\frac 1 2 - \frac 1 {q}} + \norm{f}_{L^p\pr{B_d}} \norm{v}_{L^{2}\pr{B_d}} \abs{B_d}^{\frac 1 2 - \frac 1 p} .
\end{align*}
On the other hand, if $p < 2$, then $p' > 2$ and
\begin{align*}
|\int_{B_d} G \cdot \gr v + f \, v|
&\le \norm{G}_{L^{q}\pr{B_d}} \norm{\gr v}_{L^2\pr{B_d}} \abs{B_d}^{\frac 1 2 - \frac 1 {q}} + \norm{f}_{L^p\pr{B_d}} \norm{v}_{L^{p'}\pr{B_d}} \\
&\le \pr{\norm{G}_{L^{q}\pr{B_d}} \abs{B_d}^{\frac 1 2 - \frac 1 {q}} 
+ C_p \norm{f}_{L^p\pr{B_d}}} \norm{\gr v}_{L^{2}\pr{B_d}},
\end{align*}
where the last line follows from an application of the Sobolev inequality with $2^* = p' \in \pr{2, \iny}$. 
Hereafter, we use the notation $2^*$ to denote the Sobolev exponent of $2$, and it will be chosen in $\pr{2, \iny}$.
In either case, 
$$|\int_{B_d} G \cdot \gr v + f v| \le C \norm{v}_{W^{1,2}\pr{B_d}},$$
so the functional defined by \eqref{functional} is bounded on $W^{1,2}_0\pr{B_d}$.

By assumption, $\mathcal B\brac{\cdot, \cdot}$ is a bounded, coercive form on $W^{1,2}_0\pr{B_d}$.
Therefore, we may apply the Lax-Milgram theorem to conclude that there exists a unique $\psi \in W^{1,2}_0\pr{B_d}$ that satisfies \eqref{weakForm}.
Consequently, \eqref{BVP*} has a unique solution, and therefore, \eqref{BVP1} is solvable.
\end{proof}

Using the lemma above, we now prove that a general positive multiplier exists.
With an appropriate choice of boundary data, we show that this positive multiplier has the required pointwise bounds from above and below.

\begin{lem}
\label{phiLem}
Assume that conditions \eqref{ellip}, \eqref{ABd}, \eqref{pos1}, and \eqref{pos2} hold and that $\norm{W_1}_{L^{q_1}\pr{B_d}} \le K$, $\norm{W_2}_{L^{q_2}\pr{B_d}} \le \min\set{\frac \la {2 c_{q_2}}, \frac{1}{3C_{q_2}}}$, and $\norm{V}_{L^p\pr{B_d}} \le \frac{1}{3C_p}$, where $c_{q_2}$, $C_{q_2}$ and $C_p$ will be specified below.
Then there exists a weak solution $\phi \in W^{1,2}\pr{B_d}$ to 
\begin{equation}
- \di \pr{A^T \gr \phi + W_2 \phi} + W_1 \cdot \gr \phi + V \phi = 0 \quad \text{ in } B_d
\label{epde4}
\end{equation}
with the property that
\begin{equation}
\frac 1 3 \le \phi\pr{z} \le 1 \quad \text{ for a.e. } z \in B_d.
\label{phiBound}
\end{equation}
\end{lem}

\begin{proof}[Proof of Lemma \ref{phiLem}]
Let $\phi \in W^{1,2}\pr{B_d}$ be the weak solution to the following Dirichlet boundary value problem with constant boundary data
\begin{equation}
\left\{ \begin{array}{rl}- \di \pr{A^T \gr \phi + W_2 \phi} + W_1 \cdot \gr \phi + V \phi  = 0 & \text{ in } B_d \\ \phi = 1  & \text{ on } \del B_d . \end{array}\right.
\label{BVP2}
\end{equation}
To establish that a solution to \eqref{BVP2} exists, we need to check that the associated bilinear form is bounded above and below, then we may apply Lemma \ref{solvableLem} (with $A$ being replaced by $A^T$ and the roles of $W_1$ and $W_2$ interchanged).
For any $u, v \in W^{1,2}_0\pr{B_d}$, \eqref{ABd} and H\"older's inequality imply that
\begin{align*}
|\mathcal B^* \brac{u, v}| 
&= \abs{\int_{B_d} A^T \gr u \cdot \gr v + W_2 \, u \cdot \gr v + W_1 \cdot \gr u \, v + V \, u \, v} \\
&\le \La \norm{\gr u}_{L^2\pr{B_d}} \norm{\gr v}_{L^2\pr{B_d}} 
+ \norm{W_1}_{L^{q_1}\pr{B_d}} \norm{\gr u}_{L^2\pr{B_d}} \norm{v}_{L^{\frac{2q_1}{q_1-2}}\pr{B_d}} \\
&+ \norm{W_2}_{L^{q_2}\pr{B_d}} \norm{\gr v}_{L^2\pr{B_d}} \norm{u}_{L^{\frac{2q_2}{q_2-2}}\pr{B_d}}  
+ \norm{V}_{L^p\pr{B_d}} \norm{u}_{L^{\frac{2p}{p-1}}\pr{B_d}} \norm{v}_{L^{\frac{2p}{p-1}}\pr{B_d}} \\
&\le \pr{\La + c_{q_1} \norm{W_1}_{L^{q_1}\pr{B_d}} + c_{q_2} \norm{W_2}_{L^{q_2}\pr{B_d}} + c_p \norm{V}_{L^p\pr{B_d}}} \norm{\gr u}_{L^2\pr{B_d}} \norm{\gr v}_{L^2\pr{B_d}},
\end{align*}
where we have used the Sobolev inequality three times with $2^* = \frac{2q_i}{q_i -2} \in \pr{2, \iny}$ for $i = 1, 2$ and $2^* = \frac{2p}{p-1} \in \pr{2, \iny}$ to reach the last line.
Therefore,
$$\abs{\mathcal B^* \brac{u,v}} \le C \norm{u}_{W^{1,2}\pr{B_d}} \norm{v}_{W^{1,2}\pr{B_d}}.$$
For any $v \in W^{1,2}_0\pr{B_d}$, we see that
\begin{align*}
\mathcal B^* \brac{v, v}
&= \int_{B_d} A^T \gr v \cdot \gr v + W_2 \, v \cdot \gr v + W_1 \cdot \gr v \, v + V \, \abs{v}^2 \\
&= \int_{B_d} A \gr v \cdot \gr v 
+ \frac 1 2 \int_{B_d}  W_1 \cdot \gr\pr{ v^2} 
+ \int_{B_d} W_2 \cdot \gr\pr{ v^2} + V \, \abs{v}^2 
- \int_{B_d}  W_2 \cdot \gr v \, v \\
&\ge \la  \norm{\gr v}_{L^2\pr{B_d}}^2
- \int_{B_d}  W_2 \cdot \gr v \, v ,
\end{align*}
where we have used conditions \eqref{ellip}, \eqref{pos1}, and \eqref{pos2}.
As shown above, 
$$\int_{B_d}  W_2 \cdot \gr v \, v \le c_{q_2} \norm{W_2}_{L^{q_2}\pr{B_d}} \norm{\gr v}_{L^2\pr{B_d}}^2 \le \frac \la 2 \norm{\gr v}_{L^2\pr{B_d}}^2,$$
where the second inequality follows from the assumption that $\norm{W_2}_{L^{q_2}\pr{B_d}} \le \frac{\la}{2 c_{q_2}}$.
As a result, $\mathcal B^*\brac{v, v} \ge \frac \la 2 \int_{B_d} \abs{\gr v}^2$.
The Poincar\'e inequality immediately implies that for any $v \in W^{1,2}_0\pr{B_d}$, $\mathcal B^*\brac{v, v} \ge c \norm{v}_{W^{1,2}\pr{B_d}}^2$.
In conclusion, $\mathcal B\brac{\cdot, \cdot}$ is a bounded, coercive form on $W^{1,2}_0\pr{B_d}$.
It follows from Lemma \ref{solvableLem} that \eqref{BVP2} is solvable.

It remains to show that $\phi$ satisfies the stated pointwise bounds a.e.
To this end, set $w\pr{z} = \phi\pr{z} - 1 \in W^{1,2}_0\pr{B_d}$ and note that
\begin{equation}
\left\{ \begin{array}{rll}
- \di \pr{A^T \gr w } + W_1 \cdot \gr w 
&= \di\pr{W_2 \phi} -V \phi & \text{ in } B_d \\ 
w &= 0  & \text{ on } \del B_d \end{array}\right..
\label{BVP4}
\end{equation}
Let $\Ga\pr{z, \zeta}$ denote the Green's function for the operator $ -\di\pr{A \gr+W_1}$ in $B_d$, and let $\Gamma^\ast(z,\zeta)$ denote the Green's function for the adjoint operator $-\di\pr{A^T\nabla}+W_1\cdot\nabla$ in $B_d$. 
Note that $\Gamma^\ast(\zeta,z)=\Gamma(z,\zeta)$. 
The conditions \eqref{ellip} and \eqref{pos1} ensure that such a Green's function exists (see Appendix \ref{AppB}). 
According to Theorem \ref{t3.2} in Appendix \ref{AppB} (see also Definition \ref{d3.1}), \eqref{BVP4} implies that
$$w\pr{z} = -\int_{B_d} \brac{D_\zeta\Ga\pr{z,\zeta} \cdot W_2\pr{\zeta} + \Ga\pr{z,\zeta} V\pr{\zeta}} \phi\pr{\zeta} d\zeta.$$
Therefore, 
\begin{align*}
\phi\pr{z} - 1 &= - \int_{B_d} \brac{D_\zeta\Ga\pr{z,\zeta} \cdot W_2\pr{\zeta} + \Ga\pr{z,\zeta} V\pr{\zeta}} \phi\pr{\zeta} d\zeta
\end{align*}
so that
\begin{align*}
\sup_{z \in B_d} \abs{\phi\pr{z} - 1}
&\le \sup_{z \in B_d}\brac{\int_{B_d} \abs{D_\zeta \Ga\pr{z,\zeta}} \abs{W_2\pr{\zeta}} d\zeta + \int_{B_d} \abs{\Ga\pr{z,\zeta}} \abs{ V\pr{\zeta}} d\zeta } \norm{\phi}_{L^\iny\pr{B_d}} \\
&\le \sup_{z \in B_d}\brac{\int_{B_d} \abs{D_\zeta \Ga\pr{z,\zeta}} \abs{W_2\pr{\zeta}} d\zeta + \int_{B_d} \abs{\Ga\pr{z,\zeta}} \abs{ V\pr{\zeta}} d\zeta },
\end{align*}
where we have used the maximum principle from Theorem \ref{maxPrinc} to conclude that $\norm{\phi}_{L^\iny\pr{B_d}} \le 1$.
By H\"older's inequality, we have
\begin{align*}
\int_{B_d} \abs{\gr_\zeta \Ga\pr{z,\zeta}} \abs{W_2\pr{\zeta}}d\zeta
&\le \pr{ \int_{B_d} \abs{D\Ga\pr{z,\zeta}}^{q_2^\prime} d\zeta}^{\frac 1 {q_2^\prime}}\pr{ \int_{B_d} \abs{ W_2\pr{\zeta}}^{q_2} d\zeta}^{\frac 1 {q_2}} \\
&\le \norm{W_2}_{L^{q_2}\pr{B_d}} \norm{D\Ga\pr{z,\cdot}}_{L^{q_2^\prime}\pr{B_{2d}\pr{z}\cap B_d}}
\end{align*}
and
\begin{align*}
\int_{B_d} \abs{\Ga\pr{z,\zeta}} \abs{ V\pr{\zeta}} d\zeta
\le  \norm{V}_{L^{p}\pr{B_d}}  \norm{\Ga\pr{z,\cdot}}_{L^{p^\prime}\pr{B_{2d}\pr{z} \cap B_d}}.
\end{align*}
Since $q_2 > 2$, then $q_2^\prime \in \brp{1, 2}$ and \eqref{eqB.18} implies that $\disp \sup_{z \in B_d}\norm{D\Ga\pr{z,\cdot}}_{L^{q_2^\prime}\pr{B_{2d}\pr{z}\cap B_d}} = C_{q_2}$.
Similarly, since $p > 1$, then $p^\prime \in \brp{1, \iny}$, so it follows from \eqref{eqB.17} in Theorem \ref{t3.2} that $\disp \sup_{z \in B_d}\norm{\Ga\pr{z,\cdot}}_{L^{p^\prime}\pr{B_{2d}\pr{z}\cap B_d}} = C_{p}$.
Combining the observations above with the assumed bounds on $\norm{W_2}_{L^{q_2}\pr{B_d}}$ and $\norm{V}_{L^{p}\pr{B_d}}$, we see that
\begin{align*}
\norm{\phi - 1}_{L^\iny\pr{B_d}} \le \frac 2 3,
\end{align*}
and conclusion of the lemma follows.
\end{proof}

In addition to the pointwise bounds for $\phi$ that were established in the previous lemmas, we also prove gradient estimates for all solutions to \eqref{epde}. 
A similar argument shows that analogous bounds hold for $\phi$, and hence for all solutions to equations of the form \eqref{epde4}.
The following estimates follow from a standard integration by parts argument (a Caccioppoli estimate) in combination with Theorem 2 from \cite{Mey63}.

\begin{lem}
For any $r > 0$ and $\al > 1$ for which $\al r < d$, let $v$ be a weak solution to \eqref{epde} in $B_{\al r}$.
Assume that $\norm{W_1}_{L^{q_1}\pr{B_{\al r}}} \le K_1$, $\norm{W_2}_{L^{q_2}\pr{B_{\al r}}} \le K_2$, and $\norm{V}_{L^p\pr{B_{\al r}}} \le M$.
Then for any $t \in \brac{2, \tau_0}$,
\begin{align}
\pr{ \int_{B_{r}} \abs{\gr v}^{t}}^{\frac 1 {t}}
&\le C r^{\frac 2 {t} -1}  \pr{1 + r^{2 - \frac {4}{q_1} }K_1^2 + r^{2 - \frac {4}{q_2} }K_2^2 + r^{2 - \frac 2 p } M } \norm{v}_{L^\iny\pr{B_{\al r}}},
\label{Cacc+Est}
\end{align}
where $C$ depends on $\la$, $\La$, $q_1$, $q_2$, $p$, $\al$, and $t$, and
\begin{equation}\label{t0}
\tau_0=\left\{\begin{aligned}
&\min\set{q_1,q_2,\frac{2p}{2-p}}\quad\mbox{if}\quad 1<p<2,\\
&\min\set{q_1,q_2} \quad\mbox{if}\quad p\ge 2.
\end{aligned}\right.
\end{equation}
\label{Cacc+}
\end{lem}

\begin{proof}
We start with a Caccioppoli estimate, i.e. with $t = 2$.
Let $\eta \in C^\iny_c\pr{B_{\al r}}$ be such that $\eta \equiv 1$ in $B_r$ and $\abs{\gr \eta} \le \frac C {\pr{\al -1} r}$.
Take $v \eta^2$ to be the test function. 
Then
\begin{align*}
\int \pr{A \gr v  + W_1 v }\cdot \gr\pr{v \eta^2}  + \pr{W_2 \cdot \gr v +V v} v \eta^2 = 0
\end{align*}
so that with the use of \eqref{ellip}, H\"older and Cauchy inequalities, we have
\begin{align*}
&\la \int \abs{\gr v}^2 \eta^2
\le \int A \gr v \cdot \gr v \, \eta^2 \\
&= - 2 \int A \gr v \cdot \gr \eta \, v \eta 
- \int W_1 \cdot \gr v \, v \eta^2
- 2 \int W_1 \cdot \gr \eta \abs{v}^2 \eta
- \int \pr{W_2 \cdot \gr v +V v} v \eta^2  \\
&\le 2 \La \int \abs{\gr v} \abs{\gr \eta} \abs{v} \eta 
+ 2 \int \abs{W_1} \abs{\gr \eta} \abs{v}^2 \eta
+ \int \pr{\abs{W_1} + \abs{W_2}} \abs{v} \abs{\gr v} \eta^2 
+ \int \abs{V} v^2 \eta^2 \\
&\le \frac{\la}{2} \int \abs{\gr v}^2 \eta^2 
+ \brac{\pr{\frac{4 \La^2}{\la} + \frac \la 2} \int \abs{\gr \eta}^2 
+ \frac{2}{\la} \int \pr{2\abs{W_1}^2 + \abs{W_2}^2} \eta^2 
+ \int \abs{V} \eta^2}  \norm{v}_{L^\iny\pr{B_{\al r}}}^2 \\
&\le \frac{\la}{2} \int \abs{\gr v}^2 \eta^2 
+ C \brac{\pr{\frac{4 \La^2}{\la} + \frac \la 2} \frac{\al +1}{\al -1}+ \frac{4}{\la} K_1^2 \pr{ \al r}^{2 - \frac 4 {q_1}} + \frac{2}{\la} K_2^2 \pr{ \al r}^{2 - \frac 4 {q_2}}  + M \pr{\al r}^{2 - \frac 2 p} }  \norm{v}_{L^\iny\pr{B_{\al r}}}^2.
\end{align*}
After simplifying, we see that
\begin{align*}
\norm{\gr v}_{L^2\pr{B_{r}}}
&\le C \sqrt{1 + K_1^2 r^{2 - \frac 4 {q_1}} + K_2^2 r^{2 - \frac 4 {q_2}}  + M r^{2 - \frac 2 p} }  \norm{v}_{L^\iny\pr{B_{\al r}}}.
\end{align*}

Let $\be = \frac 1 2 \pr{\al + 1}$.
Since $\di \pr{A \gr v} = - \di\pr{W_1 v} + W_2 \cdot \gr v + V v$, then an application of Theorem 2 from \cite{Mey63} shows that for $\tau_0$ given in \eqref{t0},
\begin{align*}
\pr{\int_{B_{r}} \abs{\gr v}^{\tau_0}}^{\frac 1 {\tau_0}} 
&\le C\pr{r^{\frac 2 {\tau_0} -1} \norm{\gr v}_{L^2\pr{B_{\be r}}}
+ r^{\frac 2 {\tau_0}- 1} \norm{v}_{L^\iny\pr{B_{\be r}}}} \\
&+ C \pr{\norm{W_1 v}_{L^{\tau_0}\pr{B_{\be r}}}
+ r^{\frac 2 {\tau_0} - \frac {q_2+2}{q_2} + 1} \norm{W_2 \cdot \gr v}_{L^{\frac{2q_2}{q_2+2}}\pr{B_{\be r}}}
+r^{\frac 2 {\tau_0} - \frac 2 p + 1} \norm{V v}_{L^p\pr{B_{\be r}}}} \\
&\le C r^{\frac 2 {\tau_0} -1} \brac{ \pr{1 + r^{1 - \frac {2}{q_2} }K_2 } \norm{\gr v}_{L^{2}\pr{B_{\be r}}} 
+ \pr{1 + r^{1 - \frac 2 {q_1}} K_1 + r^{2 - \frac 2 p } M } \norm{v}_{L^\iny\pr{B_{\al r}}} } \\
&\le C r^{\frac 2 {\tau_0} -1}  \pr{1 + r^{2 - \frac {4}{q_1} }K_1^2 + r^{2 - \frac {4}{q_2} }K_2^2 + r^{2 - \frac 2 p } M  } \norm{v}_{L^\iny\pr{B_{\al r}}}, 
\end{align*}
where we have used H\"older's inequality, the bounds on $W_1$, $W_2$, and $V$, and the result above for $t = 2$.
H\"older's inequality leads to the conclusion for general $t \in (2, \tau_0)$.
\end{proof}

For the positive function $\phi$ given in Lemma \ref{phiLem}, define $\Phi = \log \phi$.
From Lemma \ref{phiLem}, it is clear that $\abs{\Phi\pr{z}} \le c$ for a.e. $z \in B_d$.
Since $\di \pr{A^T \gr \phi } = W_1 \cdot \gr \phi - \di \pr{W_2 \phi} + V \phi$ weakly in $B_d$, then
\begin{equation}
\di\pr{A^T \gr \Phi} + \pr{W_2 - W_1} \cdot \gr \Phi + A^T \gr \Phi \cdot \gr \Phi  = - \di W_2 + V  \quad \text{ weakly in } B_d.
\label{phiPDE}
\end{equation}
The following estimates for $\gr \Phi$ will be crucial to our proofs.
We begin with an $L^2$-estimate for $\gr \Phi$.

\begin{lem}
Let $\Phi = \log \phi$, where $\phi$ is the positive multiplier given in Lemma \ref{phiLem}.
Then $$\norm{\gr \Phi}_{L^2\pr{B_{\rho\pr{7/5}+1/5}}} \le C K,$$ where $C\pr{\la, \La, q_1, q_2, p}$.
\label{grPhiL2}
\end{lem}

\begin{proof}
Recall that $d = \rho\pr{7/5}+2/5$.
Let $\te \in C^\iny_0\pr{B_d}$ be a cutoff function for which $\te \equiv 1$ in $B_{\rho\pr{7/5}+1/5}$.
Multiply \eqref{phiPDE} by $\te^2$, then integrate by parts to get
\begin{align*}
\la \int \abs{\gr \Phi}^2 \te^2 
&\le \int A^T \gr \Phi \cdot \gr \Phi \te^2 \\
&= \int  V \te^2 + 2 \int W_2 \cdot \te \gr \te + 2 \int A^T \gr \Phi  \te \gr \te - \int \pr{W_2 - W_1} \cdot \gr \Phi \te^2 \\
&\le \int  V \te^2 + 2 \int W_2 \cdot \te\gr \te + \frac \la 4 \int \abs{\gr \Phi}^2 \te^2 + \frac{4 \La^2}{\la} \int  \abs{\gr\te}^2 \\
&+ \frac \la 4 \int \abs{\gr \Phi}^2 \te^2 + \frac{1}{\la} \int \abs{W_2 - W_1}^2 \te^2.
\end{align*}
Rearranging and repeatedly applying H\"older's inequality, we see that
\begin{align*}
\frac \la 2 \int \abs{\gr \Phi}^2 \te 
&\le \frac 1 3 \norm{\te^2}_{L^{p^\prime}\pr{B_d}} + \frac 2 3 \norm{\gr \te}_{L^{q_2^\prime}\pr{B_d}} + \frac{4 \La^2}{\la} \norm{\gr \te}_{L^{2}\pr{B_d}}^2 \\
&+ \frac{2}{\la} \pr{\frac 2 3}^2 \norm{\te^2}_{L^{(q_2/2)'}\pr{B_d}}+ \frac{2}{\la} K^2 \norm{\te^2}_{L^{(q_1/2)'}\pr{B_d}}.
\end{align*}
Since $K \ge 1$, then
$$\int_{B_{\rho\pr{7/5}+1/5}} \abs{\gr \Phi}^2 \le C K^2,$$
where $C$ depends on $\la$, $\La$, $q_1$, $q_2$, and $p$, as required.
\end{proof}

Now we prove that $\gr \Phi$ belongs to $L^{t_0}$ for some $t_0 > 2$.

\begin{lem}
Let $\Phi = \log \phi$, where $\phi$ is the positive multiplier given in Lemma \ref{phiLem}.
Then there exists $t_0 > 2$ such that $\norm{\gr \Phi}_{L^{t_0}\pr{B_{\rho\pr{7/5}}}} \le C K^{\frac 2 \mu\pr{3 - \frac 2 {t_0}}}$, where $\mu = \min\set{2 - \frac 4 {q_1}, 2 - \frac 4 {q_2},2 - \frac 2 p}$ and $C$ depends on $\la$, $\La$, $q_1$, $q_2$, $p$, and $t_0$.
\label{grPhiLt0}
\end{lem}

\begin{proof}
We rescale equation \eqref{phiPDE}.
Set $\vp = \frac{\Phi}{C K}$ for some $C > 0$.
Then \eqref{phiPDE} is equivalent to 
\begin{equation}
\eps \di\pr{A^T \gr\vp + \widetilde W_2} + A^T \gr \vp \cdot \gr \vp = \pr{\widetilde W_1 - \widetilde W_2} \cdot \gr \vp + \widetilde V,
\label{vpPDE}
\end{equation}
where $\eps = \frac{1}{C K}$, $\widetilde W_i = \frac{W_i}{C K}$ for $i = 1, 2$, and $\widetilde V = \frac{V}{C^2 K^2}$.
We'll choose $C$ sufficiently large so that
\begin{align}
\norm{\widetilde W_1}_{L^{q_1}\pr{B_d}} \le 1, \quad 
\norm{\vp}_{L^\iny\pr{B_d}} \le 1, \quad 
\int_{B_{\rho\pr{7/5}+1/5}} \abs{\gr \vp}^2 \le 1,
\label{scaleBds}
\end{align}
where the last bound is possible because of Lemma \ref{grPhiL2}.
Since $\norm{W_2}_{L^{q_2}\pr{B_d}} \le 1$ and $\norm{V}_{L^p\pr{B_d}} \le 1$ by assumption, then it is clear that $\norm{\widetilde V}_{L^p\pr{B_d}} \le 1$ and $\norm{\widetilde W_2}_{L^{q_2}\pr{B_d}} \le 1$ as well.

\begin{clm}
Let $c> 0$ be such that for any $z \in B_{\rho\pr{7/5}}$, $B_{2c/5}\pr{z} \su B_{\rho\pr{7/5}+1/5}$.
For any $z \in B_{\rho\pr{7/5}}$ and $\eps < r < c/5$, we have
$$\int_{B_r\pr{z}} \abs{\gr \vp}^2 \le C r^\mu,$$
where $\mu = \min\set{2 - \frac 4 {q_1}, 2 - \frac 4 {q_2}, 2 - \frac 2 p}$.
\label{clm1}
\end{clm}

\begin{proof}[Proof of Claim \ref{clm1}]
It suffices to take $z = 0$.
Let $\eta \in C^\iny_0\pr{B_{2r}}$ be a cutoff function such that $\eta \equiv 1$ in $B_r$.
By the divergence theorem,
\begin{align}
0 &= \eps \int \di \brac{\pr{A^T \gr \vp + \widetilde W_2} \eta^2} \nonumber \\
&= \eps \int \di \pr{A^T \gr \vp + \widetilde W_2} \eta^2
+2 \eps \int \eta \gr \eta \cdot \pr{A^T \gr \vp}
+2 \eps \int \eta \gr \eta \cdot \widetilde W_2
\label{DivThmRes}
\end{align}
We now estimate each of the three terms.
By \eqref{vpPDE} and \eqref{scaleBds},
\begin{align}
 &\eps \int \di \pr{A^T \gr \vp + \widetilde W_2} \eta^2 \nonumber \\
 &=- \int A^T \gr \vp \cdot \gr \vp\eta^2 
 + \int \widetilde V \eta^2 
+ \int \pr{\widetilde W_1 - \widetilde W_2} \cdot \gr \vp \eta^2 \nonumber \\
 &\le - \la \int \abs{\gr \vp}^2 \eta^2 
 + \norm{\widetilde V}_{L^p\pr{B_d}} \pr{\int_{B_{2r}} 1 }^{1 - \frac 1 p} \nonumber \\
&+ \pr{\int \abs{\gr \vp}^2 \eta^2 }^{\frac 1 2}\brac{ \norm{\widetilde W_1}_{L^{q_1}\pr{B_d}} \pr{\int_{B_{2r}} 1 }^{1 - \frac 1 {q_1} - \frac 1 2}
 + \norm{\widetilde W_2}_{L^{q_2}\pr{B_d}} \pr{\int_{B_{2r}} 1 }^{1 - \frac 1 {q_2} - \frac 1 2}}\nonumber \\
&\le - \frac{\la}{2} \int \abs{\gr \vp}^2 \eta^2 
 + C r^{2 - \frac 2 p}
+ \frac{C}{2 \la} \pr{r^{2 - \frac 4 {q_1}} + r^{2 - \frac 4 {q_2}}}.
\label{IBd}
\end{align}
By Cauchy-Schwarz and Young's inequality,
\begin{align}
\abs{2 \eps \int \eta \gr \eta \cdot \pr{A^T \gr \vp} } 
&\le 2 \eps \La \int \eta \abs{\gr \eta} \abs{\gr \vp}
\le 2 \eps \La \pr{ \int \abs{\gr \vp}^2 }^{1/2} \pr{ \int \eta^2  \abs{\gr \eta}^2 }^{1/2} \nonumber \\
&\le \frac{C \La^2}{\la} \eps^2 + \frac{\la}{200} \int_{B_{2r}} \abs{\gr \vp}^2.
\label{IIBd}
\end{align}
Similarly, by H\"older and Young's inequality,
\begin{align}
\abs{2 \eps \int \eta \gr \eta \cdot \widetilde W_2}
\le 2 \eps \norm{\widetilde W_2}_{L^{q_2}\pr{B_d}} \pr{\int_{B_{2r}\setminus B_r}  \abs{\gr \eta}^{\frac{q_2}{q_2-1}}}^{1 - \frac 1 {q_2}} 
\le 2 C \eps r^{1 - \frac 2 {q_2}}
\le C \eps^2 + C r^{2 - \frac 4 {q_2}}.
\label{IIIBd}
\end{align}
Combining \eqref{DivThmRes}-\eqref{IIIBd} gives
\begin{align}
\int_{B_r} \abs{\gr \vp}^2 
\le C \eps^2 
+C\pr{ r^{2 - \frac 2 p} + r^{2 - \frac 4 {q_1}} + r^{2 - \frac 4 {q_2}} }
+ \frac{1}{100} \int_{B_{2r}} \abs{\gr \vp}^2
\le C r^\mu
+ \frac{1}{100} \int_{B_{2r}} \abs{\gr \vp}^2,
\label{combinedEst}
\end{align}
since $\mu = \min\set{2 - \frac 4 {q_1}, 2 - \frac 4 {q_2}, 2 - \frac 2 p} \le 2$ and $\eps < r < \frac c 5 < 1$.

If $r^{\mu} \ge \frac{1}{100}$, then by the last estimate of \eqref{scaleBds}, the inequality above implies that
$$\int_{B_r} \abs{\gr \vp}^2 \le C r^{\mu}.$$
Otherwise, if $r^{\mu} < \frac{1}{100}$, choose $k \in \N$ so that
$$\frac{c}{5} \le 2^k r \le \frac{2c}{5}.$$
Since $r^{\mu} \ge \pr{\frac c {2^k 5}}^{\mu} \ge \pr{\frac c {2^k 5}}^{2} = \pr{\frac {c^2} {4^k 25}} \ge C \pr{\frac 1 {100}}^k$, then it follows from repeatedly applying \eqref{combinedEst} that
\begin{align*}
\int_{B_r} \abs{\gr \vp}^2
\le C r^{\mu} + C \pr{\frac{1}{100}}^k \int_{B_{2^kr}} \abs{\gr \vp}^2
\le C r^{\mu},
\end{align*}
proving the claim.
\end{proof}

We now use Claim \ref{clm1} to give an $L^{t_0}$ bound for $\gr \vp$ in $B_{\rho\pr{7/5}}$.
Define 
\begin{align*}
\vp_\eps\pr{z} = \frac{1}{\eps} \vp\pr{\eps^{\frac{2}{\mu}} z}, \qquad
A_\eps\pr{z} = A\pr{\eps^{\frac{2}{\mu}} z}, \qquad
L^*_\eps = \di A_\eps^T \gr.
\end{align*}
Then
\begin{align*}
&\gr \vp_\eps\pr{z} = \eps^{\frac 2 \mu -1}\gr \vp \pr{\eps^{\frac 2 \mu} z} \\
&L^*_\eps \vp_\eps\pr{z} = \eps^{\frac 4 \mu -1} \di\pr{ A^T\pr{\eps^{\frac 2 \mu}z} \gr \vp\pr{\eps^{\frac 2 \mu} z}}.
\end{align*}
It follows from \eqref{vpPDE} that
\begin{align*}
L^*_\eps \vp_\eps\pr{z} + A_\eps^T \gr \vp_\eps \cdot \gr \vp_\eps
&= \eps^{\frac 4 \mu -2} \brac{\eps \di\pr{ A^T\pr{\eps^{\frac 2 \mu}z} \gr \vp\pr{\eps^{\frac 2 \mu} z}} 
+ A\pr{\eps^{\frac{2}{\mu}} z} \gr \vp \pr{\eps^{\frac 2 \mu} z} \cdot \gr \vp \pr{\eps^{\frac 2 \mu} z}} \\
&= \widetilde V_\eps\pr{z}
 + \pr{\widetilde W_{1,\eps}\pr{z} - \widetilde W_{2,\eps}\pr{z}} \cdot \gr \vp_\eps\pr{z} 
- \di \widetilde W_{2,\eps}\pr{z},
\end{align*}
where
\begin{align*}
& \widetilde W_{i, \eps}\pr{z} :=  \eps^{\frac 2 \mu -1} \widetilde W_i\pr{\eps^{\frac 2 \mu} z} \quad \text{ for } \; i = 1, 2 \\
& \widetilde V_\eps\pr{z} := \eps^{\frac 4 \mu -2} \widetilde V\pr{\eps^{\frac 2 \mu} z} .
\end{align*}
With $\de = \eps^{\frac 2 \mu} \le \frac{\rho\pr{7/5} + 1/5}{2}$, note that for $i = 1, 2$,
\begin{align*}
\norm{\widetilde W_{i,\eps}}_{L^{q_i}\pr{B_1}}
&= \pr{\int_{B_1} \abs{\widetilde W_{i,\eps}\pr{z}}^{q_i} dz}^{\frac 1 {q_i}}
= \pr{\int_{B_1} \abs{\eps^{\frac 2 \mu -1} \widetilde W_i\pr{\eps^{\frac 2 \mu} z} }^{q_i} dz}^{\frac 1 {q_i}} \\
&=\eps^{\frac 1 \mu \pr{2-\frac 4 {q_i} - \mu} } \pr{\int_{B_1} \abs{ \widetilde W_i\pr{\eps^{\frac 2 \mu} z} }^{q_i} d\pr{\eps^{\frac 2 \mu} z} }^{\frac 1 {q_i}}
\le \norm{\widetilde W_i}_{L^{q_i}\pr{B_{\de}}} 
\le 1
\end{align*}
and
\begin{align*}
\norm{\widetilde V_\eps}_{L^p\pr{B_1}}
&= \pr{\int_{B_1} \abs{\widetilde V_\eps\pr{z}}^p dz}^{\frac 1 p}
= \pr{\int_{B_1} \abs{\eps^{2\pr{\frac 2 \mu -1}} \widetilde V\pr{\eps^{\frac 2 \mu} z} }^p dz}^{\frac 1 p} \\
&=\eps^{\frac 2 \mu \pr{2-\frac 2 p - \mu} } \pr{\int_{B_1} \abs{ \widetilde V\pr{\eps^{\frac 2 \mu} z} }^p d\pr{\eps^{\frac 2 \mu} z} }^{\frac 1 p}
\le \norm{\widetilde V}_{L^p\pr{B_{\de}}} 
\le 1.
\end{align*}
Moreover,
\begin{align*}
\int_{B_2} \abs{\gr \vp_\eps}^2 
\le \eps^{2\pr{\frac 2 \mu -1}} \int_{B_{2}}\abs{\gr \vp\pr{\eps^{\frac 2 \mu} z}}^2 dz
= \frac{1}{\eps^2} \int_{B_{2\de}} \abs{\gr \vp}^2
\le \frac{1}{\eps^2} C \pr{2\eps^{\frac 2 \mu}}^{\mu} = C,
\end{align*}
where we have used Claim \ref{clm1}.
It follows from Theorem 2.3 and Proposition 2.1 in Chapter V of \cite{Gi83} that there exists $t_0 > 2$ such that
\begin{equation}
\norm{\gr \vp_\eps}_{L^{t_0}\pr{B_1}} \le C.
\label{GiaBd}
\end{equation}
Recalling the definition of $\vp_\eps$, we see that
\begin{align*}
C \ge \norm{\gr \vp_\eps}_{L^{t_0}\pr{B_{1}}} 
= \eps^{\frac 2 \mu - 1 - \frac 4 {\mu t_0}} \norm{\gr \vp}_{L^{t_0}\pr{B_{\de}}}
= \frac{\eps^{\frac 2 \mu - 1 - \frac 4 {\mu t_0}}}{C K} \norm{\gr \Phi}_{L^{t_0}\pr{B_{\de}}}.
\end{align*}
As $\eps = \frac{1}{C K}$, then we conclude that
\begin{align*}
 \norm{\gr \Phi}_{L^{t_0}\pr{B_{ \de}}}
\le C K^{\frac 2 \mu\pr{1 - \frac 2 {t_0}}}.
\end{align*}
Since this derivation works for any $z \in B_{\rho\pr{7/5}}$ and we may cover $B_{\rho\pr{7/5}}$ with $N$ balls of radius $\de$, where $N \sim \de^{-2} = \eps^{-4/\mu} \sim K^{4/\mu}$, then the result follows.
\end{proof}

Using an interpolation argument in combination with the estimates just proved, we establish $L^t$ bounds for $\gr \Phi$.

\begin{lem}
Let $\Phi = \log \phi$, where $\phi$ is the positive multiplier given in Lemma \ref{phiLem}.
Let $t_0 > 0$ be the exponent provided in Lemma \ref{grPhiLt0}.
For any $t \in \brac{2, t_0}$, $$\norm{\gr \Phi}_{L^t\pr{B_{\rho\pr{7/5}}}} \le C K^{1 + \brac{\frac 2 \mu\pr{3 - \frac 2 {t_0}} - 1}{\frac{t_0}{t}}\pr{\frac{t - 2}{t_0 - 2}}},$$ where the constant $C$ depends on $\la$, $\La$, $q_1$, $q_2$, $p$, $t_0$, and $t$.
\end{lem}

\begin{proof}
Take $t \in \pr{2,t_0}$ since the endpoint estimates are given in Lemmas \ref{grPhiL2} and \ref{grPhiLt0}.
Choose $\ga \in \pr{0, 1}$ so that $t = 2 \ga + t_0 \pr{1 - \ga}$, i.e. $\ga = \frac{t_0 - t}{t_0 - 2}$.
Then by the H\"older inequality along with Lemmas \ref{grPhiL2} and \ref{grPhiLt0},
\begin{align*}
\norm{\gr \Phi}_{L^t\pr{B_{\rho\pr{7/5}}}}
&= \pr{\int_{B_{\rho\pr{7/5}}} \abs{\gr \Phi}^{2\ga} \abs{\gr \Phi}^{t_0\pr{1 - \ga}} }^{\frac 1 t} 
\le \brac{\pr{\int_{B_{\rho\pr{7/5}}} \abs{\gr \Phi}^{2}}^{\ga} \pr{ \int_{B_{\rho\pr{7/5}}} \abs{\gr \Phi}^{t_0}}^{1 - \ga} }^{\frac 1 t} \\
&= \norm{\gr \Phi}_{L^2\pr{B_{\rho\pr{7/5}}}}^{\frac{2\ga}{t}}  \norm{\gr \Phi}_{L^{t_0}\pr{B_{\rho\pr{7/5}}}}^{\frac{t_0 \pr{1 - \ga}}{t}}
\le \pr{C K}^{\frac{2\ga}{t}}  \brac{C K^{\frac 2 \mu\pr{3 - \frac 2 {t_0}}}}^{\frac{t_0\pr{1 - \ga}}{t}}  \\
&= C K^{\frac {2\ga} t + \frac 2 \mu\pr{3 - \frac 2 {t_0}}{\frac{t_0\pr{1 - \ga}}{t}} }.
\end{align*}
Since $\frac 2 \mu\pr{3 - \frac 2 {t_0}} > 1$, then simplifying the exponent gives
\begin{align*}
\frac {2\ga} t + \brac{1 + \frac 2 \mu\pr{3 - \frac 2 {t_0}} - 1}{\frac{t_0\pr{1 - \ga}}{t}}
&= 1 + \brac{\frac 2 \mu\pr{3 - \frac 2 {t_0}} - 1} \frac{t_0}{t} \pr{\frac{t - 2}{t_0 - 2}},
\end{align*}
as required.
\end{proof}

\begin{cor}
Let $\Phi = \log \phi$, where $\phi$ is the positive multiplier given in Lemma \ref{phiLem}.
Then for any $\eps > 0$, there exists $t > 2$, depending on $\eps$, $q_1$, $q_2$, $p$, and $t_0$, such that 
\begin{equation}
\norm{\gr \Phi}_{L^t\pr{B_{\rho\pr{7/5}}}} \le C K^{1 + \eps},
\label{PhitBd}
\end{equation}
where $C$ depends on $\la$, $\La$, $q_1$, $q_2$, $p$, $t_0$, and $t$.
\label{PhiepsBound}
\end{cor}

\section{The Beltrami operators}
\label{S4}

We define a Beltrami operator that allows us to the reduce the second-order equation to a first order system.
For a complex-valued function $f = u + i v$, define 
\begin{align}
D f &=  \bar\del f + \eta\pr{z} \del f + \nu\pr{z} \ol{\del f} ,
\label{DDef} 
\end{align}
where 
\begin{align}
& \bar \del = \tfrac{1}{2} \pr{\del_x + i \del_y} \nonumber \\
&  \del = \tfrac{1}{2} \pr{\del_x - i \del_y} \nonumber \\
& \eta\pr{z} = \frac{a_{11} - a_{22} }{\det\pr{A + I}}  + i \frac{a_{12} + a_{21}}{\det\pr{A+I}}
\label{etaDef} \\
& \nu\pr{z} = \frac{\det A - 1}{\det\pr{A + I}} + i \frac{a_{21} - a_{12}}{\det\pr{A+I}}
\label{nuDef}.
\end{align}

\begin{lem}
For $\eta, \nu$ defined above, there exists $K < 1$ so that
$$\abs{\eta\pr{z}} + \abs{\nu\pr{z}} \le K.$$
\label{etanuBds}
\end{lem}

The following proof is purely computational and relies on the assumption \eqref{ellip}.

\begin{proof}
We have 
\begin{align*}
\abs{\eta} 
&= \frac{\sqrt{\pr{a_{11} - a_{22}}^2 + \pr{a_{12} + a_{21}}^2 }}{\det\pr{ A + I}}
= \frac{\sqrt{ \tr A^2 - 4 a_{11} a_{22} + \pr{a_{12} + a_{21}}^2 }}{\det A + \tr A + 1} \\
\abs{\nu}
&= \frac{\sqrt{\pr{\det A -1}^2 + \pr{a_{21} - a_{12}}^2 }}{\det\pr{ A + I}}
= \frac{\sqrt{\pr{\det A +1}^2 - 4 a_{11}a_{22} +  \pr{a_{21} + a_{12}}^2 }}{\det A + \tr A + 1}.
\end{align*}
Note that it follows from \eqref{ellip} that
\[
a_{11}a_{22}-\frac 14(a_{12}+a_{21})^2\ge\lambda^2.
\]
Therefore, we see that
\begin{align*}
\abs{\eta\pr{z}} + \abs{\nu\pr{z}}
&\le \frac{\sqrt{ \tr A^2 - 4\lambda^2 } + \sqrt{\pr{\det A +1}^2 - 4\lambda^2 }}{\tr A + \det A + 1} =: K.
\end{align*}
\end{proof}

Let $f = u + iv$.
A computation shows that
\begin{align*}
Df 
&=  \frac{ \pr{a_{11} + \det A} + i {a_{21}} }{\det\pr{A + I}}u_x 
+ \frac{ {a_{12}} + i \pr{a_{22} + \det A}}{ \det\pr{A + I}} u_y 
+ \frac{\pr{a_{11} + 1} + i {a_{12}}}{\det\pr{A + I}} i v_x
+ \frac{ {a_{21}} + i\pr{a_{22}  + 1} }{ \det\pr{A + I}} i v_y.
\end{align*}
This presentation will be useful in subsequent sections.

In addition to the operator $D$, we will also make use of an operator that is related to $D$ through some function $w$.
For a given function $w$, set
$$\eta_w\pr{z} = \left\{\begin{array}{ll} 
\eta\pr{z} + \nu\pr{z} \frac{\ol{\del w}}{\del w} & \text{ for } \del w \ne 0  \\  
\eta\pr{z} + \nu\pr{z} & \text{ otherwise }
\end{array} \right.,$$
where $\eta$ and $\nu$ are as defined in \eqref{etaDef} and \eqref{nuDef}, respectively.
By Lemma \ref{etanuBds}, it follows that $\disp \abs{\eta_w} \le K$.
Define
\begin{align}
D_w f = \ol{\del} f + \eta_w\pr{z} \del f.
\label{DwDef}
\end{align}
If $\eta_w\pr{z} = \al_w\pr{z} + i \be_w\pr{z}$, then
\begin{align}
D_w &= \frac{1}{2} \brac{\del_x + i \del_y + \pr{\al_w + i \be_w} \pr{\del_x - i \del_y } }
\nonumber \\
&= \frac{1 + \al_w + i \be_w}{2} \del_x + \frac{\be_w + i\pr{1 - \al_w}}{2} \del_y
\end{align}
Bertrami operators of this form will be used in the proofs of the main theorems.

At times, the dependence on $w$ will not be important to our arguments, so we define
\begin{equation}
\hat D = \frac{1 + \al + i \be}{2} \del_x + \frac{ \be + i \pr{1 - \al}}{2} \del_y,
\label{hatDDef}
\end{equation}
where $\al, \be$ are assumed to be functions of $z$ such that $\disp \al^2 + \be^2 \le K < 1$.
Associated to $\hat D$ is the symmetric second-order elliptic operator $\hat L = \di \pr{ \hat A \gr}$ with 
\begin{equation}
\hat A 
= \brac{\begin{array}{ll} \frac{\pr{1+\al}^2 + \be^2}{1 - \al^2 - \be^2} & \frac{2 \be}{1 - \al^2 - \be^2} \\ \frac{2\be}{1 - \al^2 - \be^2} &  \frac{\pr{1-\al}^2 + \be^2}{1 - \al^2 - \be^2} \end{array}}
= \brac{\begin{array}{ll} \hat a_{11} & \hat a_{12} \\ \hat a_{12} & \hat a_{22} \end{array}}.
\label{hatADef}
\end{equation}
A computation shows that
\begin{align*}
& \brac{\frac{\pr{1+\al}^2 + \be^2}{1 - \al^2 - \be^2}}\brac{ \frac{\pr{1-\al}^2 + \be^2}{1 - \al^2 - \be^2}} - \frac 1 4\brac{\frac{2 \be}{1 - \al^2 - \be^2} + \frac{2 \be}{1 - \al^2 - \be^2}}^2 \\
&= \frac{1}{ \pr{1 - \al^2 - \be^2}^2}\brac{\pr{1 - \al^2}^2 + 2\be^2\pr{1 + \al^2} + \be^4 - 4\be^2} \\
&= \frac{1 - 2\al^2 + \al^4 - 2\pr{1 - \al^2}\be^2 + \be^4}{ \pr{1 - \al^2 - \be^2}^2}
= 1.
\end{align*}
Therefore, $\hat A$ satisfies the same ellipticity and boundedness given in \eqref{ellip} and \eqref{ABd} with possibly different constants $\lambda$, $\Lambda$. 

\begin{rem}
\label{rem4.1}
{Note that if $D$ is given as in \eqref{DDef} and $Df=0$, then $D_wf=0$ with $w=f$, where $D_w$ is defined in \eqref{DwDef}.}
\end{rem}

\subsection{A Hadamard three-quasi-circle theorem}

Within this subsection, we present the Hadamard three-quasi-circle theorem.
We originally proved this result in \cite[Theorem~4.5]{DKW17}, but include the proof here for completeness.
The related lemmas are all presented, but we refer the reader to \cite{DKW17} for their computational proofs.

The following lemmas show that $\hat D$ relates to $\hat L$ in some of the same ways that $\ol \del$ relates to $\LP$.
These properties will allow us to prove the Hadamard three-quasi-circle theorem.

\begin{lem}{\rm\cite[Lemma~4.2]{DKW17}}
If $\hat D f = 0$, where $f\pr{x,y} = u\pr{x,y} + i v\pr{x,y}$ for real-valued $u$ and $v$, then
\begin{align*}
& \hat L u = 0 = \hat Lv.
\end{align*}
\label{rpLem}
\end{lem}

We find another parallel with the Laplace equation.  
As in the case of $\hat L = \LP$, the logarithm of the norm of ${f}$ is a subsolution to the second-order equation whenever $\hat D f = 0$. 
To see this, it suffices to prove that

\begin{lem}{\rm\cite[Lemma~4.3]{DKW17}}
If $\hat D f = 0$ and $f \ne 0$, then $\hat L\brac{ \log\abs{f\pr{z}}} = 0$.
\label{logLem}
\end{lem}

Using the fundamental solution $\hat G$ for the operator $\hat L$, we can now prove the following Hadamard three-quasi-ball inequality. 
We would like to mention that similar theorems were proved by Alessandrini and Escauriaza in \cite{AE08}, see Propositions 1, 2, using quasi-regular mappings.
\begin{thm}
Let $f$ be a function for which $\hat D f = 0$ in $Q_{s_0}$.
Set
$$M\pr{s} = \max\set{\abs{f\pr{z}} : z \in Z_s }.$$
Then for any $0 < s_1 < s_2 < s_3 < s_0$,
\begin{equation}
\log\pr{\frac{s_3}{s_1}} \log M\pr{s_2} \le \log\pr{\frac{s_3}{s_2}} \log M\pr{s_1} + \log\pr{\frac{s_2}{s_1}} \log M\pr{s_3}.
\end{equation}
\label{Hadamard}
\end{thm}

\begin{proof}
Let $\mathcal{A}_{s_1, s_3} = \set{z : s_1 \le \ell\pr{ z} \le s_3} = \overline{ Q_{s_3} \setminus Q_{s_1}}$, where $\ell$ is associated to $\hat G$, the fundamental solution of $\hat L$.  
By Lemma \ref{ZsBounds}, this set is contained in an annulus with inner and outer radius depending on $s_1$, $s_3$, $\la$, and $\La$.
In particular, it is bounded and does not contain the origin.
Therefore, $\hat G\pr{z}$ is bounded on $\mathcal{A}_{s_1, s_3}$. 
Let $z_0$ be in the interior of $\mathcal{A}_{s_1,s_3}$. 
If $f\pr{z_0}=0$, then $a \hat G\pr{z_0}+\log\abs{f\pr{z_0}}=-\infty$ for any $a\in\R$. 
On the other hand, if $f\pr{z_0} \ne 0$, then Lemma~\ref{logLem} implies that $\hat L \brac{a \hat G\pr{z} + \log\abs{f\pr{z}}} = 0$ for $z$ near $z_0$. 
By the maximum principle, $z_0$ cannot be an extremal point. 
Therefore, $a \hat G\pr{z} + \log\abs{f\pr{z}}$ takes it maximum value on the boundary of $\mathcal{A}_{s_1, s_3}$. We will choose the constant $a\in\R$ so that 
$$\max\set{a \hat G\pr{z} + \log\abs{f\pr{z}} : z \in Z_{s_1}}  = \max\set{a \hat G\pr{z} + \log\abs{f\pr{z}} : z \in  Z_{s_3}},$$
or rather
$$\log \pr{s_1^a M\pr{s_1}}  = \log\pr{s_3^a M\pr{s_3}}.$$
It follows that for any $z \in \mathcal{A}_{s_1, s_3}$,
\begin{align*}
a \hat G\pr{z} + \log\abs{f\pr{z}} 
&\le\log \pr{s_i^a M\pr{s_i}}  \quad {\mbox{for}}\; i=1,3.
\end{align*}
Furthermore, for any $s_2 \in \pr{s_1, s_3}$,
\begin{align*}
\max\set{a \hat G\pr{z} + \log\abs{f\pr{z}}: z \in Z_{s_2}} 
&\le \log \pr{s_i^a M\pr{s_i}} \quad {\mbox{for}}\; i = 1,3,
\end{align*}
or
\begin{align*}
\log \pr{s_2^a M\pr{s_2}} 
\le \log \pr{s_i^a M\pr{s_i}} \quad {\mbox{for}}\; i=1,3.
\end{align*}
Consequently,
$${s_2^a M\pr{s_2}} \le 
{s_i^a M\pr{s_i}} \quad {\mbox{for}}\;i=1,3,$$
so that for any $\tau \in \pr{0,1}$, since ${s_1^a M\pr{s_1}} = {s_3^a M\pr{s_3}}$, then
\begin{align*}
&{s_2^a M\pr{s_2}} \le  \brac{s_1^a M\pr{s_1}}^\tau \brac{s_3^a M\pr{s_3}}^{1-\tau} \\
& \brac{ M\pr{s_2}}^{\log\pr{\frac{s_3}{s_1}}} \le  \brac{\pr{\frac{s_1}{s_2}}^a M\pr{s_1}}^{\tau \log\pr{\frac{s_3}{s_1}}} \brac{\pr{\frac{s_3}{s_2}}^a M\pr{s_3}}^{\pr{1-\tau}\log\pr{\frac{s_3}{s_1}}}.
\end{align*}
We choose $\tau$ so that $\tau \log\pr{\frac{s_3}{s_1}} = \log\pr{\frac{s_3}{s_2}}$. 
Then $\pr{1-\tau} \log\pr{\frac{s_3}{s_1}} = \log\pr{\frac{s_2}{s_1}}$ and
\begin{align*}
&\brac{\pr{\frac{s_1}{s_2}}^a }^{\tau \log\pr{\frac{s_3}{s_1}}} \brac{\pr{\frac{s_3}{s_2}}^a }^{\pr{1-\tau}\log\pr{\frac{s_3}{s_1}}} 
= \exp\brac{a\log\pr{\frac{s_3}{s_2}} \log\pr{\frac{s_1}{s_2}} + a \log\pr{\frac{s_2}{s_1}} \log \pr{\frac{s_3}{s_2}} } = 1.
\end{align*}
Therefore,
\begin{align*}
& { M\pr{s_2}}^{\log\pr{\frac{s_3}{s_1}}} \le  { M\pr{s_1}}^{ \log\pr{\frac{s_3}{s_2}}} {M\pr{s_3}}^{\log\pr{\frac{s_2}{s_1}}}.
\end{align*}
Taking logarithms completes the proof.
\end{proof}

\begin{cor}
\label{3circle}
Let $f$ satisfy $\hat Df=0$  in $Q_{s_0}$. 
Then for $0<s_1<s_2<s_3 < s_0$
\[
\norm{f}_{L^\iny\pr{Q_{s_2}}} \le \pr{\norm{f}_{L^\iny\pr{Q_{s_1}}}}^\te \pr{\norm{f}_{L^\iny\pr{Q_{s_3}}}}^{1 - \te},
\]
where
\[
\te=\frac{\log(s_3/s_2)}{\log(s_3/s_1)}.
\]
\end{cor}
\begin{rem}
{From Remark~\ref{rem4.1}, we know that if $Df=0$, then $D_ff=0$. Hence Corollary~\ref{3circle} applies to such $f$}.
\end{rem}

\subsection{The similarity principle}

This subsection is similar to Section 4.4 of \cite{DKW17}. 
As usual, we include it here for the sake of completeness. 
The approach here is based on the work of Bojarksi, as presented in \cite{Boj09}. Define the operators
$$T{\om}\pr{z} = -\frac{1}{\pi} \int_\Om \frac{\om\pr{\zeta}}{\zeta - z} d\zeta$$
$$S \om \pr{z} = -\frac{1}{\pi} \int_\Om \frac{\om\pr{\zeta}}{\pr{\zeta - z}^2} d\zeta.$$

We use the of the following results, collected from \cite{Boj09}.

\begin{lem}
Suppose that $g \in L^p$ for some $p \ge 2$.  
Then $T g$ exists everywhere as an absolutely convergent integral and $S g$ exists almost everywhere as a Cauchy principal limit.
The following relations hold:
\begin{align*}
& \bar{\del} \pr{T g} = g \\
& \del \pr{T g} = S g \\
& \abs{T g\pr{z}} \le c_p \norm{g}_{L^p}, \text{ if } p > 2  \\
& \norm{S g}_{L^p} \le C_p \norm{g}_{L^p} \\
& \lim_{p \to 2^+} C_p = 1 \\
& C_2 = 1.
\end{align*}
\label{TResults}
\end{lem}

\begin{lem}[see Theorems 4.1, 4.3 \cite{Boj09}]
Let $w$ be a generalized solution (possibly admitting isolated singularities) to
\begin{equation*}
\bar \del w + q_1\pr{z} \del w + q_2\pr{z} \ol{\del w} = A\pr{z} w + B\pr{z} \bar w
\end{equation*}
in a bounded domain $\Omega \su \R^2$.
Assume that $\abs{q_1\pr{z}} + \abs{q_2\pr{z}} \le  \alpha_0 < 1$ in $\Om$, and $A$, $B$ belong to $L^t\pr{\Om}$ for some $t \ge 2$.
Then $w\pr{z}$ is given by
$$w\pr{z}= f\pr{z} e^{\phi\pr{z}},$$
where $f$ is a solution to
$$\bar \del f + q_0\pr{z} \del f = 0$$
and
$$\phi\pr{z} = T{\om}\pr{z}.$$
Here, $q_0$ is defined by
\begin{equation}
\label{q0}
q_0\pr{z} = \left\{\begin{array}{ll} 
q_1\pr{z} + q_2\pr{z} \frac{\ol{\del w}}{\del w} & \text{ for } \del w \ne 0  \\  
q_1\pr{z} + q_2\pr{z} & \text{ otherwise },
\end{array} \right.
\end{equation}
and $\om \in L^t\pr{\Om}$ solves 
\begin{equation}
\om + q_0 S \om = h
\label{intEq}
\end{equation}
with
$$h\pr{z} = \left\{\begin{array}{ll} 
A\pr{z} + B\pr{z} \frac{\bar w}{w} & \text{ for } w\pr{z} \ne 0 \text{ and } w\pr{z} \ne \iny \\  
A\pr{z} + B\pr{z} & \text{ otherwise }.
\end{array}\right.$$
\label{simPrinc}
\end{lem}

The proof ideas are available in \cite{Boj09} and detailed arguments can be found in \cite{DKW17}.
We repeat the details here since we now include the case of $t = 2$.

\begin{proof}
Let $w\pr{z}$ be the generalized solution.
Set
$$h\pr{z} = \left\{\begin{array}{ll} 
A\pr{z} + B\pr{z} \frac{\bar w}{w} & \text{ for } w\pr{z} \ne 0 \text{ and } w\pr{z} \ne \iny \\  
A\pr{z} + B\pr{z} & \text{ otherwise }
\end{array}\right.$$
and
\begin{equation*}
q_0\pr{z} = \left\{\begin{array}{ll} 
q_1\pr{z} + q_2\pr{z} \frac{\ol{\del w}}{\del w} & \text{ for } \del w \ne 0  \\  
q_1\pr{z} + q_2\pr{z} & \text{ otherwise }.
\end{array} \right.
\end{equation*}
We have $\abs{q_0\pr{z}} \le \abs{q_1\pr{z}} + \abs{q_2\pr{z}} \le \alpha_0$. Consider the integral equation
\begin{equation*}
\om + q_0 S \om = h.
\end{equation*}
Let $p \in \brac{2, t}$ be such that $C_p q_0 < 1$.
That is, $\norm{q_0 S}_{L^p\pr{\Om} \to L^p\pr{\Om}} \le C_p q_0$.
Since $\Om$ is bounded, then $h \in L^p\pr{\Omega}$ and using a Neumann series or fixed point argument, we see that this integral equation has a unique solution $\om\pr{z} \in L^p$.
Set $\phi\pr{z} = T \om\pr{z}$, then define $f\pr{z} = w\pr{z} e^{- \phi\pr{z}}$.
A computation shows that $\bar \del f + q_0 \del f = 0$, as required.
\end{proof}

\begin{cor}
Let $w$ be a generalized solution (possibly admitting isolated singularities) to
\begin{equation*}
\bar \del w + q_1\pr{z} \del w + q_2\pr{z} \ol{\del w} = A\pr{z} w + B\pr{z} \bar w
\end{equation*}
in a bounded domain $\Omega \su \R^2$.
Assume that $\abs{q_1\pr{z}} + \abs{q_2\pr{z}} \le \alpha_0 < 1$ in $\Om$, and $A$, $B$ belong to $L^t\pr{\Om}$ for some $t > 2$.
Then $w\pr{z}$ is given by
$$w\pr{z} = f\pr{z} g\pr{z},$$
where $f$ is a solution to
$$\bar \del f + q_0\pr{z} \del f = 0$$
and
$$\exp\brac{-C \pr{ \norm{A}_{L^t\pr{\Omega}} +  \norm{B}_{L^t\pr{\Omega}}}} \le \abs{g\pr{z}} \le \exp\brac{C \pr{ \norm{A}_{L^t\pr{\Omega}} +  \norm{B}_{L^t\pr{\Omega}}}}.$$
\label{simCor}
\end{cor}

\begin{proof}
From the previous lemma, we have that $g\pr{z} = \exp\pr{T \om \pr{z}}$, where $\om$ is the unique solution to \eqref{intEq}.
As above, $\norm{\om}_{L^p} \le C \norm{h}_{L^p} \le C \norm{h}_{L^t}$.
It follows from the third fact in Lemma \ref{TResults} that since we may choose $p > 2$,
$$\abs{T \om \pr{z}} \le C \norm{h}_{L^p} \le C \brac{\norm{A}_{L^t\pr{\Omega}} +  \norm{B}_{L^t\pr{\Omega}}},$$
where $C$ depends on $\Omega$.
The conclusion follows.
\end{proof}

\section{The proof of Theorem \ref{OofV}}
\label{S5}

Here we present the proof of the first order of vanishing estimate, that of Theorem \ref{OofV}.
We follow the approach used previously in \cite{KSW15} and \cite{DKW17} and use the positive multiplier to transform the equation for $u$ into a divergence-free equation.

Let $u$ be a solution to \eqref{epde} in $B_d \su \R^2$.
That is,
$$- \di\pr{A \gr u + W_1 u} + W_2 \cdot \gr u + V u = 0 \quad \text{ in } \; B_d.$$
Conditions \eqref{pos1} and \eqref{pos2} in combination with the bounds on $W_1$, $W_2$, $V$ imply that Lemma \ref{phiLem} is applicable, and therefore there exists a positive function $\phi$ satisfying \eqref{phiBound} that weakly solves
$$- \di\pr{A^T\gr \phi + W_2 \phi} + W_1 \cdot \gr \phi + V \phi = 0  \quad \text{ in } \; B_d.$$ 

Set $b =- A^T \gr \Phi + W_1 - W_2$ and observe that
\begin{align*}
&\di \brac{ \phi \pr{ A\gr u + b u}}  \\
&= \di \brac{ \phi \pr{ A\gr u - u A^T \gr \Phi + u W_1 - u W_2 }}  \\
&=\di \pr{ \phi A\gr u - u A^T \gr \phi + u \phi W_1 - u \phi W_2 } \\
&=\gr \phi \cdot A \gr u + \phi \di\pr{A \gr u} - \gr u \cdot A^T \gr \phi - u \di\pr{A^T \gr \phi}+ \di \pr{ u \phi W_1 - u \phi W_2 } \\
&= \phi \brac{Vu + W_2 \cdot \gr u - \di\pr{W_1 u}} - u \brac{V \phi + W_1 \cdot \gr \phi - \di\pr{W_2 \phi}} + \di \pr{ u \phi W_1 - u \phi W_2 } \\
&=0.
\end{align*}
Therefore, the PDE for $u$ can be transformed into a divergence-free equation.
Let $v$ be the stream function associated to the vector $\phi \pr{A \gr u + b u}$ with $v\pr{0,0} = 0$.
That is, for every $\pr{x,y} \in B_d$, 
\begin{equation}
v\pr{x,y} = \int_0^1 \brac{- \phi \pr{a_{21} u_x + a_{22} u_y + b_2 u}\pr{tx,ty} x + \phi \pr{a_{11} u_x + a_{12} u_y + b_1 u }\pr{tx,ty}y} dt. 
\label{tildev}
\end{equation}
To verify the validity of \eqref{tildev}, we let $P = \pr{P_1, P_2} = \phi \pr{A \gr u + b u}$, then
\begin{equation*}
v\pr{x,y} = -\int_0^1P_2\pr{tx,ty} x \, dt +  \int_0^1 P_1\pr{tx,ty}y \, dt.
\end{equation*}
So we have
\begin{align*}
v_{x}\pr{x,y} 
&=  -\int_0^1 \del_1P_2\pr{tx,ty} t x \,dt - \int_0^1 P_2\pr{tx,ty} dt + \int_0^1\del_1P_1\pr{tx,ty} ty \, dt \\
&=  -\int_0^1 \del_1P_2\pr{tx,ty} t x \,dt - \int_0^1 P_2\pr{tx,ty} dt - \int_0^1\del_2P_2\pr{tx,ty} ty \, dt \\
&= - \int_0^1 \brac{\del_t P_2\pr{tx,ty} t + P_2\pr{tx,ty} } dt
= - \int_0^1 \del_t \brac{P_2\pr{tx,ty} t } dt 
= - P_2\pr{x,y}
\end{align*}
and
\begin{align*}
v_y\pr{x,y} 
&= - \int_0^1\del_2 P_2\pr{tx,ty} t x\,dt + \int_0^1 \del_2 P_1\pr{tx,ty} ty\,dt + \int_0^1 P_1\pr{tx,ty} dt \\
&=  \int_0^1\del_1 P_1\pr{tx,ty} t x\,dt + \int_0^1 \del_2 P_1\pr{tx,ty} ty\,dt + \int_0^1 P_1\pr{tx,ty} dt \\
&= \int_0^1 \brac{ \del_t P_1\pr{tx,ty} t + P_1\pr{tx,ty} } dt
= \int_0^1  \del_t \brac{P_1\pr{tx,ty} t} dt
= P_1\pr{x,y}.
\end{align*}
That is,
\begin{equation}
\left\{\begin{array}{rl}
{v}_y &= \phi \pr{a_{11} u_x + a_{12}  u_y + b_1 u} \\
-{v}_x &= \phi \pr{a_{21} u_x + a_{22} u_y + b_2 u} .
\end{array}  \right.
\label{streamFunc}
\end{equation}

\begin{lem}
For any $r$ and $\kappa > 1$ such that $\kappa r \le d$, there is a constant $C$, depending on $\la$, $\La$, $q_1$, $q_2$, $p$, and $\kappa$, for which 
$$\norm{v}_{L^1\pr{B_r}} \le C r^2 \pr{1 + r^{2 - \frac {4}{q_1} }K^2 } \norm{u}_{L^\iny\pr{B_{\be r}}}.$$
\label{tildevBd}
\end{lem}

\begin{proof}
As above, we use the notation
$$v\pr{x,y} = -\int_0^1P_2\pr{tx,ty} x \, dt +  \int_0^1 P_1\pr{tx,ty}y \, dt.$$
It follows that
\begin{align*}
\norm{ v}_{L^1\pr{B_r}}
&= \int_{B_r}\abs{-\int_0^1P_2\pr{tx,ty} x \, dt +  \int_0^1 P_1\pr{tx,ty}y \, dt} dz \\
&\le \int_{B_r} \int_0^1 \abs{P_2\pr{tx,ty} x } dt dz
+ \int_{B_r} \int_0^1 \abs{P_1\pr{tx,ty} y } dt dz \\
&\le r \int_{B_r} \int_0^1 \abs{P_2\pr{tx,ty}} dt dz
+ r \int_{B_r} \int_0^1 \abs{P_1\pr{tx,ty}} dt dz.
\end{align*}
A computation shows that
\begin{align*}
P_1 &= \pr{a_{11} u_x + a_{12} u_y} \phi - u \pr{a_{11} \phi_x + a_{21} \phi_y} + u \pr{W_{1,1}- W_{2,1}} \phi \\
P_2 &= \pr{a_{21} u_x + a_{22} u_y} \phi - u \pr{a_{12} \phi_x + a_{22} \phi_y} + u \pr{W_{1,2}- W_{2,2}} \phi,
\end{align*}
where we use the notation $W_i = \pr{W_{i,1}, W_{i,2}}$ for $i = 1, 2$.
By Lemma \ref{Cacc+} applied to $u$ and $\phi$, along with the assumption that each $W_i \in L^{q_i}$, it follows that each $P_i \in L^{\tau_0}$.
Interchanging the order of integration, applying H\"older's inequality, then simplifying, we see that
\begin{align*}
\int_{B_r} \int_0^1 \abs{P_1\pr{tx,ty}} dt dz
&= \int_0^1 \frac 1 {t^2} \int_{B_{tr}}  \abs{P_1\pr{x,y}} dz dt
\le \int_0^1 \frac 1 {t^2} \pr{\int_{B_{tr}} \abs{P_1\pr{x,y}}^{\tau_0} dz}^{\frac 1{\tau_0}} \abs{B_{tr}}^{1 - \frac{1}{\tau_0}}dt \\
&= C r^{2 - \frac{2}{\tau_0}} \int_0^1 \norm{P_1}_{L^{\tau_0}\pr{B_{tr}}} t^{ - \frac{2}{\tau_0}} dt
\le C r^{2 - \frac{2}{\tau_0}} \norm{P_1}_{L^{\tau_0}\pr{B_{r}}} \int_0^1 t^{ - \frac{2}{\tau_0}} dt.
\end{align*}
Since $\tau_0 > 2$, then $\disp \int_0^1 t^{ - \frac{2}{\tau_0}} dt$ converges and $\disp \int_{B_r} \int_0^1 \abs{P_1\pr{tx,ty}} dt dz \le C r^{2 - \frac{2}{\tau_0}} \norm{P_1}_{L^{\tau_0}\pr{B_{r}}}.$
A similar estimate holds for $\disp \int_{B_r} \int_0^1 \abs{P_2\pr{tx,ty}}^2 dt dz$, so we conclude that
\begin{align*}
\norm{v}_{L^1\pr{B_r}}
&\le C r^{3 - \frac 2 {\tau_0}} \pr{ \norm{P_1}_{L^{\tau_0}\pr{B_r}} + \norm{P_2}_{L^{\tau_0}\pr{B_r}}}.
\end{align*}
For $i = 1, 2$, with an application of \eqref{ABd} and Lemma \ref{Cacc+}, we see that
\begin{align*}
\norm{P_i}_{L^{\tau_0}\pr{B_r}} 
&\le \La \pr{\norm{\gr u}_{L^{\tau_0}\pr{B_r}} \norm{\phi}_{L^\iny\pr{B_r}} + \norm{u}_{L^\iny\pr{B_r}} \norm{\gr\phi}_{L^{\tau_0}\pr{B_r}}} \\
&+\pr{ \norm{W_{1,i}}_{L^{q_1}\pr{B_r}} \abs{B_r}^{\frac 1 {\tau_0} -\frac{1}{q_1}} + \norm{W_{2,i}}_{L^{q_2}\pr{B_r}} \abs{B_r}^{\frac 1 {\tau_0} -\frac{1}{q_2}}} \norm{u}_{L^\iny\pr{B_r}} \norm{\phi}_{L^\iny\pr{B_r}} \\
&\le C r^{\frac 2 {\tau_0} -1}  \pr{1 + r^{2 - \frac {4}{q_1} }K^2 } \norm{u}_{L^\iny\pr{B_{\kappa r}}}  \norm{\phi}_{L^{\iny}\pr{B_{\kappa r}}} \\
&+ \pr{K r^{\frac 2 {\tau_0} -\frac{2}{q_1}} +  r^{\frac 2 {\tau_0} -\frac{2}{q_2}}} \norm{u}_{L^\iny\pr{B_r}} \norm{\phi}_{L^\iny\pr{B_r}} \\
&\le C r^{\frac 2 {\tau_0} -1} \pr{1 + r^{2 - \frac {4}{q_1} }K^2 } \norm{u}_{L^\iny\pr{B_{\kappa r}}},
\end{align*}
where we have used the pointwise bounds on $\phi$ from Lemma \ref{phiLem}.
It follows that
\begin{align*}
\norm{v}_{L^1\pr{B_r}}
&\le C r^2  \pr{1 + r^{2 - \frac {4}{q_1} }K^2 }   \norm{u}_{L^\iny\pr{B_{\kappa r}}},
\end{align*}
as required.
\end{proof}

With $w = \phi u + i v$ and $D$ as defined in \eqref{DDef}, we see that in $B_d \Supset B_{\rho\pr{7/5}}$,
\begin{align}
D w &=  D \phi u + \phi Du + D (i v) \nonumber \\
&= D\pr{ \log \phi} \phi u 
+ \phi \brac{\frac{ \pr{a_{11} + \det A} + i {a_{21}} }{\det\pr{A + I}}u_x 
+ \frac{ {a_{12}} + i \pr{a_{22} + \det A}}{ \det\pr{A + I}} u_y } \nonumber \\
&+ \frac{\pr{a_{11} + 1} + i {a_{12}}}{\det\pr{A + I}} i v_x
+ \frac{ {a_{21}} + i\pr{a_{22}  + 1} }{ \det\pr{A + I}} i v_y \nonumber \\
&= \phi \brac{D\pr{ \log \phi} +  i  b_1\frac{ {a_{21}} + i\pr{a_{22}  + 1} }{ \det\pr{A + I}} - i b_2 \frac{\pr{a_{11} + 1} + i {a_{12}}}{\det\pr{A + I}} } u \nonumber \\
&+ \phi \brac{\frac{ \pr{a_{11} + \det A} + i {a_{21}} }{\det\pr{A + I}} + \frac{-ia_{21}\pr{a_{11} + 1} + {a_{12}a_{21}}}{\det\pr{A + I}} + \frac{ i{a_{21}}a_{11}  -a_{11}  \pr{a_{22}  + 1} }{ \det\pr{A + I}} } u_x \nonumber \\
&+ \phi \brac{ \frac{ {a_{12}} + i \pr{a_{22} + \det A}}{ \det\pr{A + I}}+ \frac{-ia_{22}\pr{a_{11} + 1} + {a_{12}a_{22}}}{\det\pr{A + I}} + \frac{ ia_{12} {a_{21}} - a_{12}\pr{a_{22}  + 1} }{ \det\pr{A + I}}  } u_y \nonumber \\
&= \pr{\al + \be_1 - \be_2} \pr{w + \bar w},
\label{diffEq}
\end{align}
where, recalling that we set $\Phi = \log \phi$,
\begin{align*}
\al + \be_1 - \be_2 &= \frac 1 2 D \Phi 
+ \frac{ b_2 a_{12}- b_1\pr{a_{22}  + 1} + i b_1 a_{21}- i b_2\pr{a_{11} + 1} }{ 2\det\pr{A + I}}.
\end{align*}
That is, 
\begin{align}
\al &=\frac{ 2a_{11} \pr{1 + a_{22} } - \pr{a_{12}+a_{21}} a_{12} 
+ i \brac{\pr{a_{12} + a_{21}} + a_{11}\pr{a_{12}- a_{21}}}}{4\det\pr{A + I}}\Phi_x  \nonumber \\
&+ \frac{ \pr{a_{12} + a_{21}} - a_{22}\pr{a_{12}- a_{21}} + i \brac{2a_{22}\pr{1 + a_{11}} - \pr{a_{12}+ a_{21}} a_{21}}}{ 4\det\pr{A + I}} \Phi_y 
\label{alDef} \\
\be_j  &= \frac{- W_{j,1} \pr{a_{22}  + 1} + W_{j,2} a_{12} - i W_{j,2} \pr{a_{11} + 1} +  i W_{j,1} a_{21}}{ 2\det\pr{A + I}} \qquad \text{ for } \, j = 1, 2.
\label{beDef}
\end{align}
It follows from the boundedness of $A$ described by \eqref{ABd} in combination with Corollary \ref{PhiepsBound}, that for any $\eps > 0$, there exists $t > 2$ such that
\begin{equation*}
\norm{\al}_{L^t\pr{B_{\rho\pr{7/5}}}} \le C K^{1 + \eps}.
\end{equation*}
The boundedness of $A$ along with the assumptions on $W_1$ and $W_2$ implies that
\begin{align*}
\norm{\be_1}_{L^{q_1}\pr{B_{\rho\pr{7/5}}}} \le C K \\
\norm{\be_2}_{L^{q_2}\pr{B_{\rho\pr{7/5}}}} \le C .
\end{align*}
We now apply the similarity principle given in Lemma \ref{simPrinc} and Corollary \ref{simCor} to conclude that any solution to \eqref{diffEq} in $B_{\rho\pr{7/5}}$ is a function of the form
$$w\pr{z} = f\pr{z} g\pr{z},$$
with
$$D_w f = 0 \;\; \text{ in } \; B_{\rho\pr{7/5}},$$
and for a.e. $z \in B_{\rho\pr{7/5}}$,
\begin{align*}
&\exp\brac{-C \pr{ \norm{\al}_{L^t\pr{B_{\rho\pr{7/5}}}} +  \norm{\be_1}_{L^{q_1}\pr{B_{\rho\pr{7/5}}}}+  \norm{\be_2}_{L^{q_2}\pr{B_{\rho\pr{7/5}}}}}} \\
&\le \abs{g\pr{z}} \le \exp\brac{C \pr{ \norm{\al}_{L^t\pr{B_{\rho\pr{7/5}}}} +  \norm{\be}_{L^q\pr{B_{\rho\pr{7/5}}}}+  \norm{\be_2}_{L^{q_2}\pr{B_{\rho\pr{7/5}}}}}}.
\end{align*}
That is,
\begin{align}
\exp\pr{-C K^{1 + \eps} } \le & \abs{g\pr{z}} \le \exp\pr{C K^{1 + \eps}} \;\; \text{ in } B_{\rho\pr{7/5}},
\label{gBnd}
\end{align}
where we have used the bounds on $\al$ and $\be_i$ from above.
By Corollary \ref{3circle}, the Hadamard three-quasi-circle theorem, applied to the operator $D_w$, 
\begin{align*}
\norm{f}_{L^\iny\pr{Q_{s_1}}} \le \pr{\norm{f}_{L^\iny\pr{Q_{s/4}}}}^\te \pr{\norm{f}_{L^\iny\pr{Q_{s_2}}}}^{1 - \te},
\end{align*}
where $s < s_1 < s_2 < \frac{7}{5}$ and 
$$\te = \frac{\log\pr{s_2/s_1}}{\log\pr{4 s_2/s}}.$$
Let $r/4 = \rho\pr{s/4}$ and $r_2 = \rho\pr{s_2}$ so that $Q_{s/4} \su B_{r/4}$ and $Q_{s_2} \su B_{r_2}$.
We choose $r_3 \in \pr{r_2, \rho\pr{7/5}}$ so that $r_3-r_2 \sim 1$ and $\rho\pr{7/5} - r_3 \sim 1$.
Since $f$ is a solution to an elliptic equation (see Lemma \ref{rpLem}), then standard interior estimates for elliptic equations (see, for example, \cite[Theorem~4.1]{HL11}) imply that
\begin{align}
\norm{f}_{L^\iny\pr{Q_{s_1}}} \le C \pr{r^{-2} \norm{f}_{L^1\pr{B_{r/2}}}}^\te \pr{\norm{f}_{L^1\pr{B_{r_3}}}}^{1 - \te},
\label{threeBall}
\end{align}
where $C$ is an absolute constant.
Substituting $f = w g^{-1}$ into \eqref{threeBall} and applying \eqref{gBnd}, we see that
\begin{align}
\norm{ w}_{L^\iny\pr{Q_{s_1}}} \le \exp\pr{C K^{1 + \eps}} \pr{r^{-2}\norm{w}_{L^1\pr{B_{r/2}}}}^\te \pr{\norm{ w}_{L^1\pr{B_{r_3}}}}^{1 - \te}.
\label{3balls2}
\end{align}
Since $w = \phi u+ i v$, then
\begin{align*}
\abs{\phi u} \le \abs{w} \le \abs{\phi u} +\abs{v}.
\end{align*}
An application of \eqref{phiBound} and Lemma \ref{tildevBd} with $\kappa = 2$ shows that
\begin{align*}
\norm{w}_{L^1\pr{B_{r/2}}}
&\le \norm{\phi u}_{L^1\pr{B_{r/2}}} + \norm{v}_{L^1\pr{B_{r/2}}}
\le C r^2 \pr{1 + r^{2 - \frac {4}{q_1} }K^2 } \norm{u}_{L^\iny\pr{B_{r}}}.
\end{align*}
We similarly conclude that
\begin{align*}
\norm{w}_{L^1\pr{B_{r_3}}}
\le C \pr{1 + K^2 } \norm{u}_{L^\iny\pr{B_{d}}}
\le \exp\pr{C K},
\end{align*}
where we have applied \eqref{localBd} in the second inequality.
Upon setting $s_1 = 1$ in \eqref{3balls2} and using the bounds established above, we have
\begin{align*}
\norm{u}_{L^\iny\pr{Q_{1}}} 
&\le \exp\pr{C K^{1 + \eps}} \norm{ u}_{L^\iny\pr{B_{r}}}^\te .
\end{align*}
Now we define 
\begin{equation}
\label{bsigma}
b = \si\pr{1}
\end{equation}
so that $B_b \su Q_1$.
Since $\norm{u}_{L^\iny\pr{Q_1}} \ge \norm{u}_{L^\iny\pr{B_b}} \ge 1$ by \eqref{localNorm}, after rearranging, we have
\begin{equation*}
\norm{ u}_{L^\iny\pr{B_r}} \ge \exp\pr{-\frac C \te K^{1 + \eps}}.
\end{equation*}
It follows from the definition of $\te$ that
\begin{align*}
\norm{ u}_{L^\iny\pr{B_{r}}} \ge r^{C K^{1 + \eps}},
\end{align*}
and the conclusion of Theorem \ref{OofV} follows.

\section{The proof of Theorem \ref{OofV1}}
\label{S6}

\subsection{The case of $q > 2$}

When $W_2, V \equiv 0$, the proof above carries through with $\phi = 1$.
Since there is no need to construct a positive multiplier using Lemma \ref{phiLem}, the positivity condition on $W_1$ described by \eqref{pos1} is unnecessary.
In this simplified setting, we see that $\al$ given in \eqref{alDef} is equal to zero and then \eqref{gBnd} holds with $\eps = 0$.
The remainder of the proof is unchanged and the estimate \eqref{localEst} therefore holds with $\eps = 0$.

\subsection{The case of $q = 2$}

Here we need to prove that the strong unique continuation property (SUCP) holds for solutions to $-\di\pr{A \gr + Wu} = 0$ when $W \in L^2\pr{B_d}$.
Recall the following definition of SUCP:

\begin{defn}
Suppose $u \in W^{1,2}_{\loc}\pr{B_d}$ is a solution to $-\di\pr{A \gr u + W u} = 0$.
We say that the {\em strong unique continuation property} holds if whenever $u$ vanishes to infinite order at some point $z_0 \in B_d$, i.e. for every $N \in \N$,
$$\abs{u\pr{z}} \le \mathcal{O}\pr{\abs{z - z_0}^N} \quad \text{ as } z \to z_0,$$
this implies that $u \equiv 0$.
\end{defn}

Assume that $z_0 = 0$.
That is, we assume that $u$ vanishes to infinite order at $0$ and we will prove that $u \equiv 0$ in $B_d$.
Let $v$ be the stream function associated to $-\di\pr{A \gr u + W u} = 0$ defined by
\begin{equation*}
v\pr{x,y} = \int_0^1 \brac{- \pr{a_{21} u_x + a_{22} u_y + W_2 u}\pr{tx,ty} x + \pr{a_{11} u_x + a_{12} u_y + W_1 u }\pr{tx,ty}y} dt. 
\end{equation*}
Here we write $W=(W_1,W_2)$. With $w = u + i v$ and $D$ as defined in \eqref{DDef}, we see that in $B_d$,
\begin{align*}
D w &= \be w,
\end{align*}
where
\begin{align*}
\be  &= \left\{\begin{array}{ll}
\frac{- W_{1} \pr{a_{22}  + 1} + W_{2} a_{12} - i W_{2} \pr{a_{11} + 1} +  i W_{1} a_{21}}{ 2\det\pr{A + I}}\pr{1 + \frac{\overline w}{w}} & w \ne 0 \\
0 & \text{otherwise} \end{array}\right..
\end{align*}
It follows from \eqref{ellip}, \eqref{ABd}, and the bound on $W$ that $\norm{\be}_{L^2\pr{B_d}} \le C K$.
The similarity principle given in Lemma \ref{simPrinc} implies that 
$$w = f\pr{z} g\pr{z},$$
with
$$D_{w} f = 0 \;\; \text{ in } \; B_d,$$
and 
$$g\pr{z} = \exp\pr{T \om\pr{z}},$$
where 
$\om \in L^2$ with $\norm{\om}_{L^2\pr{B_d}} \le C \norm{\be}_{L^2\pr{B_d}} \le C K$.
As
\begin{equation*}
h\pr{z} :=T \om (z)=\frac{1}{\pi}\int_{B_d}\frac{\om(\xi)}{\xi-z}d\xi,
\end{equation*}
we have that
$$\norm{h}_{W^{1,2}\pr{B_d}}=\norm{h}_{L^2\pr{B_d}}+ \norm{\gr h}_{L^2\pr{B_d}}\le CK.$$
Since we do not have $T \om \in L^\iny$ for $\om \in L^2$, we rely on the following result from \cite{KW15}.

\begin{lem}[cf. Lemma 3.3 in \cite{KW15}]
\label{lemma0701}
Let $h$ be as defined above.
For $s>0$ and $0<r\le \rho\pr{7/5}$, we have that
\begin{equation}\label{qqest}
\fint_{B_r}\exp(s|h|)\le Cr^{-sCK}\exp(sCK+s^2CK^2).
\end{equation}
\end{lem}

Now we demonstrate how a modification of the ideas in \cite{KW15} leads to the proof of our theorem.
Following the arguments in the proof of Theorem \ref{OofV}, we see that \eqref{threeBall} holds with $f = w \, \exp\pr{-h}$.
That is,
\begin{align}
\norm{w \, \exp\pr{-h}}_{L^\iny\pr{Q_{s_1}}} 
&\le C \pr{r^{-2} \norm{w \, \exp\pr{-h}}_{L^1\pr{B_{r/2}}}}^\te 
\pr{\norm{w \, \exp\pr{-h}}_{L^1\pr{B_{r_3}}}}^{1 - \te},
\label{3ballw}
\end{align}
where $r/4 = \rho\pr{s/4}$, $r_2 = \rho\pr{s_2}$ and $r_3 \in \pr{r_2, \rho\pr{7/5}}$ is such that $r_3-r_2 \sim 1$ and $\rho\pr{7/5} - r_3 \sim 1$. 
With $s_1 = 1$, we have
\begin{align}
\norm{u}_{L^2\pr{B_{b}}} 
&\le C \norm{w}_{L^2\pr{Q_{s_1}}} 
\le C \norm{w \, \exp\pr{-h}}_{L^\iny\pr{Q_{s_1}}} \norm{\exp\pr{\abs{h}}}_{L^2\pr{Q_{s_1}}} \nonumber \\
&\le \exp(CK^2) \norm{w \, \exp\pr{-h}}_{L^\iny\pr{Q_{s_1}}},
\label{1ballEst}
\end{align}
where the second inequality follows from an application of H\"older's inequality, and the third follows from Lemma \ref{lemma0701}.
To bound the righthand side of \eqref{3ballw}, an application of Lemma \ref{lemma0701} shows that
\begin{align}
r^{-2} \norm{w \, \exp\pr{-h}}_{L^1\pr{B_{r/2}}}
&\le r^{-2} \int_{B_{r/2}} \abs{w} \exp\pr{\abs{h}} \nonumber \\
&\le C r^{-1} \pr{\int_{B_{r/2}} \abs{w}^2}^{1/2} \pr{\fint_{B_{r/2}}\exp\pr{2\abs{h}}}^{1/2} \nonumber \\
&\le C r^{-1} r^{-CK}\exp(CK+CK^2) \pr{ \norm{u}_{L^2\pr{B_r}} +  \norm{v}_{L^2\pr{B_{r/2}}}} .
\label{RHSEst}
\end{align}
Next we need to estimate $\norm{v}_{L^2\pr{B_r}}$.
Since $v\pr{0} = 0$, then
\begin{align}
\int_{B_{r/2}} \abs{v\pr{z}}^2
&= \int_{B_{r/2}} \abs{v\pr{z} - v\pr{0}}^2
= \int_{B_{r/2}} \abs{\int_0^1 \gr v\pr{t z} \cdot z \, dt}^2 dz \nonumber \\
&\le \pr{r/2}^2 \int_{B_{r/2}} \int_0^1 \abs{\gr v\pr{t z}}^2  dt dz \nonumber \\
&= Cr^3 \int_0^{r/2} \pr{\frac{1}{|B_s|}\int_{B_s}|\gr v\pr{z} |^2 dz } ds.
\label{L2vEst0}
\end{align}
As 
\begin{align*}
{v}_y &=  a_{11} u_x + a_{12}  u_y + W_1 u \\
-{v}_x &= a_{21} u_x + a_{22} u_y + W_2 u ,
\end{align*}
then \eqref{ABd} implies that
\begin{align}
\int_{B_{r/2}} \abs{v\pr{z}}^2
&\le Cr^3 \int_0^{r/2} \pr{\frac{1}{|B_s|}\int_{B_s}|\gr u |^2 + \abs{W u}^2  } ds 
\le Cr^3 \int_0^{r/2} \pr{\frac{K^4 \norm{u}_{L^\iny\pr{B_{\al s}}}^2}{s^2|B_s|}   } ds,
\label{L2vEst}
\end{align}
where we have used Caccioppoli's type inequality (see Lemma \ref{Cacc+}).
Since $u$ vanishes to infinite order at $0$, there exists $R_1 < d$ and $C_1 > 0$ so that
$$\abs{u\pr{z}} \le C_1 \abs{z}^2 \quad \text{ for all } \abs{z} < R_1.$$
Choosing $\al$ so that $\al r_3 = \rho\pr{7/5}$, it follows from the computations above that
\begin{align*}
\int_{B_{r_3}} \abs{v\pr{z}}^2
\le CK^4 \brac{\int_0^{R_1/\al} \pr{\frac{ \norm{u}_{L^\iny\pr{B_{\al s}}}^2}{s^2|B_s|}   } ds + \int_{R_1/\al}^{r_3} \pr{\frac{\norm{u}_{L^\iny\pr{B_{\al s}}}^2}{s^2|B_s|}   } ds}
\le \exp\pr{C K}
\end{align*}
and therefore, for any $\te \in \pr{0,1}$,
\begin{align}
\norm{w \, \exp\pr{-h}}_{L^1\pr{B_{r_3}}}^{1 - \te} \le \exp\pr{C K^2}.
\label{r3Est}
\end{align}
Now we assume that $\norm{u}_{L^2\pr{B_b}} \ge \exp\pr{-k}$ for some $k > 0$, then combine \eqref{1ballEst}, \eqref{3ballw}, and \eqref{r3Est} with the definition of $\te$ to conclude that
\begin{align*}
C r^{\tilde C \pr{k + C K^2}}
&\le  r^{-2} \norm{w \, \exp\pr{-h}}_{L^1\pr{B_{r/2}}}.
\end{align*}
Using that $u$ vanishes to infinite order at $0$, estimates \eqref{RHSEst}, \eqref{L2vEst0}, and \eqref{L2vEst} may be combined to conclude that there exists $N_0 > \tilde C \pr{k + C K^2}$ and $r_{N_0} > 0$ so that for any $r \le r_{N_0}$,
$$r^{-2} \norm{w \, \exp\pr{-h}}_{L^1\pr{B_{r/2}}} \le C_{N_0} r^{N_0}.$$
As this leads to a contradiction, we must have $\norm{u}_{L^2\pr{B_b}} < \exp\pr{-k}$ for every $k > 0$.
This means that $u \equiv 0$ in $B_b$.

In the case where $z_0 \ne 0$, a translation and scaling allows us to reduce to the case of $z_0 = 0$.

\section{The proof of Theorem \ref{OofV2}}
\label{S7}

For the proof of Theorem \ref{OofV2}, we follow the ideas used previously in \cite{DKW17} to rewrite the equation as a product of Beltrami operators, then we invoke a number of the ideas that were used above in the proof of Theorem \ref{OofV}.

Let $u$ be a solution to \eqref{epde3} in $B_d \su \R^2$.
That is,
$$- \di\pr{A \gr u} + W \cdot \gr u = 0 \quad \text{ in } \; B_d.$$
Dividing this equation through by $\sqrt{\mbox{det} A}$ gives
\begin{equation*}
\di\pr{\frac{A}{\sqrt{\mbox{det}A}}\gr u}-\widetilde{W}\cdot\gr u=0,
\end{equation*}
where
\begin{equation}\label{WW}
\widetilde W=A\gr\left(\frac{1}{\sqrt{\mbox{det}A}}\right)+\frac{W}{\sqrt{\mbox{det}A}}.
\end{equation}
Since $\norm{W}_{L^q\pr{B_d}} \le K$, then conditions \eqref{ellip}, \eqref{ABd}, and \eqref{gradDec} imply that $\|\widetilde W\|_{L^q\pr{B_d}} \le C K$, where $C$ depends on $\la$, $\La$, and $\mu$.
Moreover, the ellipticity constant of $A/\sqrt{\mbox{det}A}$ is $\lambda^2$.
Thus, there is no loss of generality in assuming that $u$ is a solution to \eqref{epde3} where $A$ is symmetric with determinant equal to $1$.

When $A$ is symmetric and has determinant equal to $1$, the term $\nu$ defined in \eqref{nuDef} is equal to zero and the definition of $D$ is greatly simplified.
Further, we have the following decomposition result from \cite{DKW17}.

\begin{lem}[Lemma 4.4 in \cite{DKW17}]
Assume that $A$ is uniformly elliptic, bounded, symmetric, Lipschitz continuous and has determinant equal to $1$.
Then the operator $\di\pr{A \gr}$ may be decomposed as
$$\di\pr{A \gr} = \pr{D + \Ga} \widetilde D,$$
where
\begin{align*}
D &= \frac{\pr{a_{11}  + 1} + i a_{12}}{\det\pr{A + I}} \del_x + \frac{  a_{12} + i \pr{a_{22} + 1}}{\det\pr{A + I}}  \del_y \\
\widetilde D &= \brac{1 + a_{11} - i a_{12}} \del_x + \brac{a_{12} - i\pr{1+a_{22}}} \del_y 
= \det\pr{A+I} \overline D  \\
\Ga 
&= \frac{\pr{\al \del_x a_{11} - \be \del_x a_{12} + \ga \del_y a_{11} + \de \del_y a_{12}} + i \pr{\ga \del_x a_{11} + \de \del_x a_{12} - \al \del_y a_{11} + \be \del_y a_{12}}}{ a_{11} \det\pr{A+I}^2},
\end{align*}
with 
\begin{align*}
&\al = a_{11} + a_{22} + 2 a_{11}a_{22} \qquad
\be = 2 a_{12}\pr{1 + a_{11}} \\
&\ga = a_{12}\pr{a_{22} - a_{11}} \qquad\qquad
\de = \pr{1 + a_{11}}^2 - a_{12}^2.
\end{align*}
Moreover, $\norm{\Ga}_{L^\iny} \le C\pr{\la, \La, \mu}$.
\label{decompLem}
\end{lem}

\begin{lem}[cf. Lemma 7.1 in \cite{DKW17}]
There exists $\widetilde\Upsilon \in L^q\pr{B_d}$ so that 
\begin{equation}
W \cdot \gr u =\widetilde\Upsilon \widetilde D u.
\label{UpsEqn}
\end{equation}
Moreover,
\begin{equation}
\|\widetilde\Upsilon\|_{L^q\pr{B_{d}}} \le C K.
\label{UpsBd}
\end{equation}
\end{lem}

\begin{proof}
Set $\Upsilon = e + i f$, where $e, f$ are real-valued functions to be determined.
Then
\begin{align*}
\Upsilon \widetilde D u 
&= \pr{e + i f} \set{\brac{1 + a_{11} - i a_{12}}\del_x u + \brac{a_{12} - i \pr{1 + a_{22}}}\del_y u} \\
&= \brac{e \pr{1 + a_{11}} + f a_{12}}\del_x u + \brac{e a_{12} + f \pr{1 + a_{22}}} \del_y u \\
&+ i \set{\brac{f\pr{1 + a_{11}} - e a_{12}} \del_x u + \brac{f a_{12} - e\pr{1 + a_{22}}} \del_y u}
\end{align*}
so that
\begin{align*}
\frac{1}{2}\brac{\Upsilon \widetilde D u + \overline{\Upsilon  \widetilde D u} }
&= \brac{e \pr{1 + a_{11}} + f a_{12}}\del_x u + \brac{e a_{12} + f \pr{1 + a_{22}}} \del_y u.
\end{align*}
If we define 
$$\widetilde \Upsilon = \left\{\begin{array}{ll} 
\frac{1}{2}\brac{\Upsilon + \bar {\Upsilon}\frac{ \overline{ \widetilde D u}}{\widetilde D u}} & \text{ whenever }  \widetilde D u \ne 0 \\ 
0 & \text{ otherwise } \end{array}\right.,$$
then \eqref{UpsEqn} will be satisfied if we choose $e$, $f$ so that
\begin{align*}
e \pr{1 + a_{11}} + f a_{12} &= W_1 \\
e a_{12} + f \pr{1 + a_{22}} &= W_2.
\end{align*}
Solving this system, we see that
\begin{align*}
\brac{\begin{array}{l} e \\ f \end{array}} 
&= \frac{1}{\det\pr{A+I}} \brac{\begin{array}{ll} 1 + a_{22} & - a_{12} \\ - a_{12} & 1 + a_{11}  \end{array}}  \brac{\begin{array}{l} W_1\\ W_2  \end{array}} 
=  \frac{1}{\det\pr{A+I}} \brac{\begin{array}{l} \pr{1 + a_{22}} W_1 - a_{12} W_2 \\
- a_{12}W_1 + \pr{1 + a_{11}} W_2   \end{array}}.
\end{align*}
The bounds on $A$ and $W$ imply that \eqref{UpsBd} holds and the proof is complete.
\end{proof}

An application of the previous two lemmas shows that
\begin{equation}
D \widetilde D u = \pr{ \widetilde \Upsilon - \Ga} \widetilde D u.
\label{diffEq2}
\end{equation}
Now we apply the technique from the proof of Theorem \ref{OofV} to the equation above, where $\widetilde D u$ now plays the role of $w$.
The similarity principle given in Lemma \ref{simPrinc} implies that any solution to \eqref{diffEq2} in $B_d$ takes the form
$$\widetilde D u = f\pr{z} g\pr{z},$$
with
$$D_{\widetilde D u} f = 0 \;\; \text{ in } \; B_d,$$
and 
$$g\pr{z} = \exp\pr{T \om\pr{z}},$$
where 
$\om \in L^t$ for some $t \in \brac{2, q}$ with 
$$\norm{\om}_{L^t\pr{B_d}} \le C\brac{\|\widetilde \Upsilon\|_{L^q\pr{B_d}} +  \norm{\Ga}_{L^{\iny}\pr{B_d}}}.$$
Now we have to consider the cases of $q > 2$ and $q = 2$ separately.

\subsection{The case $q > 2$}
Assuming that $q > 2$, Corollary \ref{simCor} implies that for a.e. $z \in B_{\rho\pr{7/5}}$,
\begin{align*}
\exp\brac{-C \pr{ \|\widetilde \Upsilon\|_{L^q\pr{B_{\rho\pr{7/5}}}} +  \norm{\Ga}_{L^{\iny}\pr{B_{\rho\pr{7/5}}}}}} 
&\le \abs{g\pr{z}} \le \exp\brac{C \pr{ \|\widetilde \Upsilon\|_{L^q\pr{B_{\rho\pr{7/5}}}} +  \norm{\Ga}_{L^{\iny}\pr{B_{\rho\pr{7/5}}}}}}.
\end{align*}
That is,
\begin{align}
\exp\pr{-C K } \le & \abs{g\pr{z}} \le \exp\pr{C K} \;\; \text{ in } B_{\rho\pr{7/5}}.
\label{gBnd2}
\end{align}
Proceeding as in the proof of Theorem \ref{OofV}, we see that
\begin{align}
\| \widetilde D u\|_{L^\iny\pr{Q_{s_1}}} \le \exp\pr{C K} \pr{r^{-1}\|\widetilde D u\|_{L^2\pr{B_{r/2}}}}^\te \| \widetilde D u\|_{L^2\pr{B_{r_3}}}^{1 - \te}.
\label{3balls22}
\end{align}
Since $|\widetilde D u| \sim \abs{\gr u}$, then an application of Lemma \ref{Cacc+} shows that
\begin{align*}
\|\widetilde D u\|_{L^2\pr{B_{r/2}}} &\le \pr{1 + r^{2 - \frac 4 q} K^2 } \norm{u}_{L^\iny\pr{B_{r}}} \\
\|\widetilde D u\|_{L^2\pr{B_{r_3}}} &\le \pr{1 + K^2 } \norm{u}_{L^\iny\pr{B_{d}}} \le  e^{C K}.
\end{align*}
where we have applied \eqref{localBd} in the second inequality.
Upon setting $s_1 = 6/5$ in \eqref{3balls22} and using the bounds established above, we have
\begin{align}
\norm{\gr u}_{L^\iny\pr{Q_{6/5}}} 
&\le \exp\pr{C K} \pr{ r^{-1}\norm{ u}_{L^\iny\pr{B_{r}}}}^\te .
\label{3bgrLHS}
\end{align}
To complete the proof, we need to bound the left-hand side from below using the assumption that $\norm{u}_{L^\iny\pr{B_b}} \ge 1$.
We repeat the argument from \cite{KSW15} here.
Since $B_b \su Q_{1}$, this assumption implies that there exists $z_0 \in Q_1$ such that $\abs{u\pr{z_0}} \ge 1$. 
Without loss of generality, we assume that $u\pr{z_0} \ge 1$.
Since $u$ is real-valued, then for any $a > 0$, we have that either $u\pr{z} \ge a$ for all $z \in Q_{6/5}$, or there exists $z_1 \in Q_{6/5}$ such that $u\pr{z_1} < a$.
If the second case holds, then we see that $u\pr{z_1} \le a$, while $u\pr{z_0} \ge 1$.
If we set $a = \exp\pr{- K}$ then it follows that
\begin{align*}
C \norm{\gr u}_{L^\iny\pr{Q_{6/5}}} 
\ge \abs{u\pr{z_0} - u\pr{z_1}} 
\ge 1 - \exp\pr{- K} 
\ge \frac{1}{2}.
\end{align*}
Combining this bound with \eqref{3bgrLHS} and rearranging leads to the proof of the theorem.
If we are in the former case, then $u\pr{z} \ge a$ for all $z \in Q_{6/5}$ and the conclusion of the theorem is obviously satisfied.
The proof of the Theorem \ref{OofV2}(a) is now complete.

\subsection{The case of $q = 2$}
When $q = 2$, $\om \in L^2\pr{B_d}$ with $\norm{\om}_{L^2\pr{B_d}} \le C K$.
As
\begin{equation*}
h\pr{z} :=T \om (z)=\frac{1}{\pi}\int_{B_d}\frac{\om(\xi)}{\xi-z}d\xi,
\end{equation*}
we have that
$$\norm{h}_{W^{1,2}\pr{B_d}}=\norm{h}_{L^2\pr{B_d}}+ \norm{\gr h}_{L^2\pr{B_d}}\le CK.$$
Since we do not have $T \om \in L^\iny$ for $\om \in L^2$, we again use the result from \cite{KW15} described in Lemma \ref{lemma0701}.
Following the arguments in the proof of Theorem \ref{OofV}, we see that \eqref{threeBall} holds with $f = \widetilde Du \, \exp\pr{-h}$.
That is,
\begin{align*}
\norm{\widetilde Du \, \exp\pr{-h}}_{L^\iny\pr{Q_{s_1}}} 
&\le C \pr{r^{-2} \norm{\widetilde Du \, \exp\pr{-h}}_{L^1\pr{B_{r/2}}}}^\te 
\pr{\norm{\widetilde Du \, \exp\pr{-h}}_{L^1\pr{B_{r_3}}}}^{1 - \te},
\end{align*}
where $r/4 = \rho\pr{s/4}$, $r_2 = \rho\pr{s_2}$ and $r_3 \in \pr{r_2, \rho\pr{7/5}}$ is such that $r_3-r_2 \sim 1$ and $\rho\pr{7/5} - r_3 \sim 1$. 
Then an application of Lemmas \ref{Cacc+} and \ref{lemma0701} shows that
\begin{align*}
r^{-2} \norm{\widetilde Du \, \exp\pr{-h}}_{L^1\pr{B_{r/2}}}
&\le r^{-2} \int_{B_{r/2}} \abs{\widetilde Du} \exp\pr{\abs{h}} \\
&\le C r^{-1} \pr{\int_{B_{r/2}} \abs{\gr u}^2}^{1/2} \pr{\fint_{B_{r/2}}\exp\pr{2\abs{h}}}^{1/2} \\
&\le C r^{-1} r^{-CK}\exp(CK+CK^2) \pr{1 + K^2} \norm{u}_{L^\iny\pr{B_r}} \\
&\le r^{-CK} \exp\pr{C K^2}\norm{u}_{L^\iny\pr{B_r}}
\end{align*}
and
\begin{align*}
\norm{\widetilde Du \, \exp\pr{-h}}_{L^1\pr{B_{r_3}}}
&\le \exp\pr{C K^2}\norm{u}_{L^\iny\pr{B_{\rho\pr{7/5}}}}
\le M \exp\pr{C K^2},
\end{align*}
where we have used the upper bound on $u$.
Similarly, we see that with $s_1=6/5$ and
\begin{equation}
\tilde b = \si\pr{6/5},
\label{btsigma}
\end{equation}
we have
\begin{align*}
1 \le \norm{\gr u}_{L^2\pr{B_{\tilde b}}} 
&\le \norm{\gr u}_{L^2\pr{Q_{s_1}}} 
\le C \norm{\widetilde Du}_{L^2\pr{Q_{s_1}}} 
= C \norm{\widetilde Du \, \exp\pr{-h}}_{L^\iny\pr{Q_{s_1}}} \norm{\exp\pr{\abs{h}}}_{L^2\pr{Q_{s_1}}} \\
&\le \exp(CK^2) \norm{\widetilde Du \, \exp\pr{-h}}_{L^\iny\pr{Q_{s_1}}}.
\end{align*}
Combining the inequalities above and simplifying shows that
\begin{align*}
\norm{u}_{L^\iny\pr{B_r}}
&\ge r^{CK}\exp\pr{- \frac{C K^2 + \log M}{\te}}
\ge r^{C\pr{\log M + K^2}}.
\end{align*}

\section{The proof of Theorem \ref{OofV3}}
\label{S8}

\subsection{The case of $A= I$, $q = \iny$}

We first consider case (a) of Theorem \ref{OofV3}.
Let $u$ be a solution to \eqref{epde} with $A = I$.
The first step is to prove that there exists a positive solution to the equation
\begin{equation}
-\LP \phi + \pr{W_1 + W_2} \cdot \gr \phi + \pr{V - W_1 \cdot W_2} \phi = 0
\label{adjPDE}
\end{equation}
in $B_{9/5}$.
Let $\eta$ be some constant to be determined and set 
$$\phi_1\pr{x,y} = \exp\pr{\eta x}.$$ 
Then by the bounds on $W_1$, $W_2$ and $V$, we see that
\begin{align*}
-\LP \phi_1 + \pr{W_1 + W_2} \cdot \gr \phi_1 +  \pr{V - W_1 \cdot W_2}  \phi_1
&\le\pr{ -\eta^2 + 2K \eta + 2 K^2} \phi_1 .
\end{align*}
If $\eta = 3 K$, then $\phi_1$ is a subsolution.  
Now define $\phi_2 = \exp\pr{6 K}$ so that $\phi_2 \ge \phi_1$ on $B_{9/5}$.
Since $V - W_1 \cdot W_2 \ge 0$, then $-\LP \phi_2 + \pr{W_1 + W_2} \cdot \gr \phi_2 +  \pr{V - W_1 \cdot W_2}  \phi_2 \ge 0$, so $\phi_2$ is a supersolution.  
It follows that there exists a positive solution $\phi$ to \eqref{adjPDE} such that
\begin{equation}
\exp\pr{-C_1 K} \le \phi\pr{z} \le \exp\pr{C_1K} \;\;\;\ \text{ for all } z \in B_{9/5}.
\label{phiBd1}
\end{equation}
Define $\disp v = \frac{u}{\phi}$ and observe that
\begin{align*}
\di \pr{\gr v + W_1 v} 
&= \di \brac{\frac{\gr u + W_1 u}{\phi} - \frac{u \gr \phi}{\phi^2}} \\
&=  \frac{\di \pr{\gr u + W_1 u}}{\phi} - \frac{u \LP \phi}{\phi^2} - \frac{ u W_1 \cdot \gr \phi }{\phi^2} - 2\frac{\gr u \cdot \gr \phi}{\phi^2}  + 2 \frac{u \abs{\gr \phi}^2}{\phi^3} \\
&=  \pr{W_2 - 2 \frac{\gr \phi}{\phi}} \cdot \gr v 
+ \pr{W_1 \cdot W_2 - 2 W_1 \cdot \frac{ \gr \phi }{\phi} } v.
\end{align*}
Therefore,
\begin{align*}
\di \pr{\gr v + W_1 v} + W_1 \cdot \gr v + \abs{W_1}^2 v 
&=  \pr{W_1 + W_2 - 2 \frac{\gr \phi}{\phi}} \cdot \gr v 
+ \pr{\abs{W_1}^2 + W_1 \cdot W_2- 2W_1 \cdot \frac{ \gr \phi }{\phi} } v \\
&= W_3 \cdot \pr{\gr v + W_1 v},
\end{align*}
where $W_3 = W_1 + W_2 - 2 \frac{\gr \phi}{\phi}$.
Define the operators
\begin{align*}
\bar D &= \del_x + i \del_y + W_{11} + i W_{12} \\
D &= \del_x - i \del_y + W_{11} - i W_{12},
\end{align*}
where $W_1 = \pr{W_{11}, W_{12}}$.
Then
\begin{align*}
Dv &= \del_x v - i \del_y v + W_{11} v - i W_{12} v \\
\bar D D v 
&= \pr{\del_x + i \del_y + W_{11} + i W_{12} } \pr{\del_x v - i \del_y v + W_{11} v - i W_{12} v} \\
&= \del_{xx} v - i \del_{xy} v + \del_x\pr{W_{11} v} - i \del_x\pr{W_{12} v} 
+i \del_{yx} v + \del_{yy} v + i\del_y\pr{ W_{11} v} + \del_y\pr{ W_{12} v} \\
&+ W_{11} \del_x v - i W_{11} \del_y v + \pr{W_{11}}^2 v - i W_{11} W_{12} v 
+ i W_{12}\del_x v + W_{12} \del_y v +i  W_{11}W_{12} v + \pr{W_{12}}^2 v \\
&= \di\pr{\gr v + W_1 v} + W_1 \cdot \gr v + \abs{W_1}^2 v 
+ i\pr{\del_y W_{11} - \del_x W_{12}} v.
\end{align*}
Since we have assumed that the scalar curl of $W_1$ vanishes weakly, then 
$$\bar D D v = \di\pr{\gr v + W_1 v} + W_1 \cdot \gr v + \abs{W_1}^2 v .$$
If we let $Z = \frac 1 2 \brac{W_{31} + i W_{32} + \pr{W_{31} - i W_{32}}\frac{\overline{Dv}}{Dv}}$ then
$$Z Dv = W_3 \cdot \pr{\gr v + W_1 v}.$$
Therefore, with $w = Dv$, we have $\bar D w = Z w$, or $\bar \del w + \al w = 0$ where $\al = W_{11} + i W_{12} - Z$, so that
$$\bar \del \pr{e^{T\pr{\al}}w} = 0.$$
As shown in Lemma 3.1 of \cite{KSW15}, for example, $\norm{\gr \log \phi}_{L^\iny\pr{B_{7/5}}} \le C K$ so that  $\norm{\al}_{L^\iny\pr{B_{7/5}}} \le C K$ and $\norm{T\pr{\al}}_{L^\iny\pr{B_{7/5}}} \le C K$.
Now we apply the Hadamard three-circle theorem in combination with an interior estimate to $e^{T\pr{\al}}w$ to conclude that
\begin{align*}
\norm{e^{T\pr{\al}}w}_{L^\iny\pr{B_{6/5}}} 
&\le C \norm{e^{T\pr{\al}}w}_{L^2\pr{B_{7/5}}}^{1 - \te} \pr{ r^{-1}\norm{e^{T\pr{\al}}w}_{L^2\pr{B_{r/2}}}}^\te ,
\end{align*}
where $\te = \log\pr{\frac 7 6}/{\log\pr{\frac{14}{5r}}}$.
Rearranging and using the bound on $T\pr{\al}$, we see that
\begin{align*}
\norm{w}_{L^\iny\pr{B_{6/5}}} 
\le \exp\pr{C K} \norm{w}_{L^2\pr{B_{7/5}}}^{1 - \te} \pr{r^{-1}\norm{w}_{L^2\pr{B_{r/2}}}}^\te .
\end{align*}
Since $\abs{w} = \abs{Dv} \le \abs{\gr v} + K \abs{v}$, the Caccioppoli estimate given in Lemma \ref{Cacc+} and the bound in \eqref{localBd} imply that
\begin{align*}
\norm{w}_{L^2\pr{B_{r/2}}} &\le \exp\pr{C K} \norm{u}_{L^\iny\pr{B_{r}}} \\
\norm{w}_{L^2\pr{B_{7/5}}} &\le \exp\pr{C K}. 
\end{align*}
Therefore,
\begin{align}
\norm{w}_{L^\iny\pr{B_{6/5}}} 
\le \exp\pr{C K} \pr{r^{-1}\norm{u}_{L^\iny\pr{B_{r}}}}^\te .
\label{wUpper}
\end{align}
Now we need to use \eqref{localNorm} to bound the left-hand side.
Recall that $w = \pr{\del_x v + W_{11} v} - i \pr{\del_y v + W_{12} v}$.
Since $W_1$ is assumed to be weakly curl-free, the function $\Phi$ given by
\begin{align*}
\Phi\pr{x,y} &= \int_0^x W_{11}\pr{t, y} dt + \int_0^y W_{12}\pr{0, s} ds
\end{align*}
satisfies $\gr \Phi = W_1$.
Moreover, for any $\Om \su B_{9/5}$, $\norm{\Phi}_{L^\iny\pr{\Om}} \le C \norm{W_1}_{L^\iny\pr{B_{9/5}}} \le C K$.
Using this representation of $W_1$, we see that
\begin{align*}
w 
&= \pr{\del_x v + W_{11} v} - i \pr{\del_y v + W_{12} v} 
= \pr{\del_x v + \del_x \Phi v} - i \pr{\del_y v + \del_y \Phi v} \\
&= e^{-\Phi} \brac{\del_x\pr{e^\Phi v} - i \del_y\pr{e^\Phi v} }.
\end{align*}
Since $v$ and $\Phi$ are real-valued, $\abs{w} \sim e^{-\Phi} \abs{\gr\pr{e^\Phi v}}$ so that $\norm{w}_{L^\iny\pr{B_{6/5}}} \ge e^{- CK} \norm{\gr\pr{e^\Phi v}}_{L^\iny\pr{B_{6/5}}}$.
To bound $\norm{\gr\pr{e^\Phi v}}_{L^\iny\pr{B_{6/5}}}$ from below, we again repeat the argument from \cite{KSW15}.
The lower bound on $u$ given in \eqref{localNorm} implies that there exists $z_0 \in B_1$ such that $\abs{u\pr{z_0}} \ge 1$. 
Without loss of generality, we assume that $u\pr{z_0} \ge 1$.
Since $u$ is real-valued, then for any $a > 0$, we have that either $u\pr{z} \ge a$ for all $z \in B_{6/5}$, or there exists $z_1 \in B_{6/5}$ such that $u\pr{z_1} < a$.
If the second case holds, then we see that $u\pr{z_1} \le a$ so that by \eqref{phiBd1} and the bound on $\Phi$, $e^\Phi v\pr{z_1} \le a \exp\pr{CK}$, while $u\pr{z_0} \ge 1$ so that $e^\Phi v\pr{z_0} \ge \exp\pr{-CK}$.
If we set $a = \frac 1 2\exp\pr{- 2CK}$, then
\begin{align*}
C \norm{\gr (e^\Phi v)}_{L^\iny\pr{B_{6/5}}} 
&\ge \abs{e^\Phi v\pr{z_0} - e^\Phi v\pr{z_1}} 
\ge \exp\pr{-CK} - \frac 1 2\exp\pr{- 2CK}\exp\pr{CK} \\
&\ge \frac 1 2\exp\pr{- CK}.
\end{align*}
It follows that $\norm{w}_{L^\iny\pr{B_{6/5}}} \ge e^{- CK}$.
Substituting this bound into \eqref{wUpper} and rearranging leads to the proof of the theorem.
If we are in the former case, then $u\pr{z} \ge a$ for all $z \in Q_{6/5}$ and the conclusion of the theorem is obviously satisfied.

\subsection{The case of $W_2, V \equiv 0$}
\label{curlCor}

If $W_2, V \equiv 0$, then equation \eqref{epde} reduces to
$$-\di\pr{A \gr u + W_1 u} = 0.$$
Define $\tilde u\pr{x, y} = u\pr{y, -x}$, $\widetilde W_1\pr{x, y} = W_1\pr{y, -x}$ and $\tilde A\pr{x,y} = \brac{\begin{array}{ll} a_{22}\pr{y, -x} & -a_{12}\pr{y, -x} \\ - a_{12}\pr{y, -x} & a_{11}\pr{y, -x} \end{array}}$.
Since \eqref{ellip} implies that $\xi^T A\pr{x, y} \xi \ge \la \abs{\xi}^2$ for every $\pr{x,y} \in B_d$ and $\xi \in \R^2$, then upon replacing $\xi = \pr{\xi_1, \xi_2}^T$ with $\pr{\xi_2, -\xi_1}^T$, it is clear that $\tilde A$ is also uniformly elliptic.
The other conditions on $A$ are clearly inherited for $\tilde A$.
Since $\di\pr{\tilde A \gr \tilde u} = \di\pr{A \gr u}\pr{y, -x}$ and, by the curl-free condition, $\widetilde W_1 \cdot \gr \tilde u = \di\pr{W_1 u}\pr{y,-x}$ then the PDE further reduces to
$$-\di\pr{\tilde A \gr \tilde u} - \widetilde W_1 \cdot \gr \tilde u = 0.$$
An application of Theorem \ref{OofV2} leads to the proof of Theorem \ref{OofV3} for case (b).

\section{Unique continuation at infinity estimates}
\label{S9}

We follow the scaling approach of Bourgain and Kenig from \cite{BK05} to show how Theorems \ref{LandisThm}, \ref{LandisThm2} and \ref{LandisThm3} follow from Theorems \ref{OofV1}, \ref{OofV2}, and \ref{OofV3}, respectively.

\subsection{The proof of Theorem \ref{LandisThm}}
Let $u$ be a solution to \eqref{epde2} in $\R^2$.
Choose $z_0 \in \R^2$ so that $\abs{z_0} = b R$.
Define $u_R(z) = u(z_0 + Rz)$, $A_R\pr{z} = A\pr{z_0 + R z}$, $W_{R}\pr{z} = R \, W\pr{z_0 + R z}$.
Notice that for any $r > 0$,
\begin{align*}
\norm{W_{R}}_{L^{q}\pr{B_r\pr{0}}}
&= \pr{\int_{B_r\pr{0}} \abs{W_{R}\pr{z}}^{q} dz}^{\frac 1 {q}}
= \pr{\int_{B_r\pr{0}} \abs{R \, W\pr{z_0 + R z}}^{q} dz}^{\frac 1 {q}} \\
&= R^{1 - \frac 2 {q} } \pr{  \int_{B_r\pr{0}} \abs{W\pr{z_0 + R z}}^{q} d\pr{Rz} }^{\frac 1 {q_1}}
= R^{1 - \frac 2 {q} } \norm{W}_{L^{q}\pr{B_{r R}\pr{z_0}}}.
\end{align*}
Therefore, $\disp \norm{W_{R}}_{L^{q}\pr{B_d\pr{0}}}  \le \al R^{1 - \frac 2 {q}}$.
As
\begin{align*}
& - \di\brac{ A_R\pr{z} \gr u_R\pr{z} + W_{R}\pr{z} u_R\pr{z}}  \\
&= R^2 \set{- \di\brac{ A\pr{z_0 + Rz} \gr u\pr{z_0 + Rz} + W\pr{z_0 + Rz} u\pr{z_0 + Rz}}} 
= 0,
\end{align*}
then $u_R$ satisfies a scaled version of \eqref{epde2} in $B_{d}$.
By assumption \eqref{uBd},
\begin{align*}
\norm{u_R}_{L^\iny\pr{B_{d}}}
&= \norm{u}_{L^\iny\pr{B_{dR}\pr{z_0}}} 
\le \exp\brac{C_1 \pr{b + d}^{1 - \frac 2 {q}} R^{1 - \frac 2 {q}}}.
\end{align*}
Thus, if we choose $\hat C \ge C_1 \pr{b + d}^{1 - \frac 2 {q}}/\al$, then
\begin{align}
\norm{u_R}_{L^\iny\pr{B_{d}}}
\le \exp\pr{\hat C \al R^{1 - \frac{2}{q}} }.
\label{uRUpperBd}
\end{align}
Note that for $\disp\widetilde{z_0} := - \frac {z_0}{R} $, we have $\disp\abs{\widetilde{z_0}} = b$ and then \eqref{normed} implies that $\abs{u_R(\widetilde{z_0})} = \abs{u(0)} \ge 1$, so that 
\begin{align}
\norm{u_R}_{L^\iny(B_b)} \ge 1.
\label{uRLowerBd}
\end{align}

Assuming that $R$ is sufficiently large so that $1/R$ is sufficiently small, we apply Theorem \ref{OofV1}(a) to $u_R$ with $K = \al R^{1 - \frac {2}{q}}$ and $C_0 = \hat C$ to get
\begin{align*}
\norm{u}_{L^\iny\pr{{B_{1}(z_0)}}} 
&= \norm{u_R}_{L^\iny\pr{B_{1/R}(0)}}  
\ge (1/R)^{ C {\al R^{1 - \frac {2}{q}}}}
=  \exp\pr{- C {\al R^{1 - \frac {2}{q}}} \log R },
\end{align*}
where $C$ depends on $\la$, $\La$, $q$, and $\hat C$.

\subsection{The proof of Theorem \ref{LandisThm2}}

Let $u$ be a solution to \eqref{epde3} in $\R^2$.
Choose $z_0 \in \R^2$ so that $\abs{z_0} = b R$ and define $u_R$, $A_R$ and $W_R$ as above.
As before, $\disp \norm{W_{R}}_{L^{q}\pr{B_d\pr{0}}}  \le \al R^{1 - \frac 2 {q}}$.
Since
\begin{align*}
& - \di\brac{ A_R\pr{z} \gr u_R\pr{z}} + W_{R}\pr{z} \cdot \gr u_R\pr{z}   \\
&= R^2 \set{- \di\brac{ A\pr{z_0 + Rz} \gr u\pr{z_0 + Rz}} + W\pr{z_0 + Rz} \cdot \gr u\pr{z_0 + Rz}} 
= 0,
\end{align*}
then $u_R$ satisfies a scaled version of \eqref{epde3} in $B_{d}$.

If $q > 2$, then as above, \eqref{uBd} implies that \eqref{uRUpperBd} holds with $\hat C \ge C_1 \pr{b + d}^{1 - \frac 2 {q}}/\al$ and \eqref{normed} implies that \eqref{uRLowerBd} also holds.
Assuming that $R$ is sufficiently large, Theorem \ref{OofV2}(a) applied to $u_R$ with $K = \al R^{1 - \frac {2}{q}}$ and $C_0 = \hat C$ shows that
\begin{align*}
\norm{u}_{L^\iny\pr{{B_{1}(z_0)}}} 
&\ge \exp\pr{- C {\al R^{1 - \frac {2}{q}}} \log R },
\end{align*}
where $C$ depends on $\la$, $\La$, $\mu$, $q$, and $\hat C$.

If $q = 2$, then $\disp \norm{u_R}_{L^\iny\pr{B_{d}}} = \norm{u}_{L^\iny\pr{B_{dR}\pr{z_0}}} \le \brac{\pr{b + d}R}^m$ so we take $M = \brac{\pr{b + d}R}^m$.
With $\disp\widetilde{z_0} := - \frac {z_0}{R} $, $\disp\abs{\widetilde{z_0}} = b$ and then for $R$ sufficiently large so that $\pr{\tilde b -b}R \ge b$, we see that
$$\norm{\gr u_R}_{L^2\pr{B_{\tilde b}}} = \norm{\gr u}_{L^2\pr{B_{\tilde bR}\pr{z_0}}} \ge \norm{\gr u}_{L^2\pr{B_{b}}}  \ge 1.$$
Now we apply Theorem \ref{OofV2}(b) with $K = \al$ and $M = \brac{\pr{b + d}R}^m$ to get
\begin{align*}
\norm{u}_{L^\iny\pr{B_1\pr{z_0}}}
&= \norm{u_R}_{L^\iny\pr{B_{1/R}\pr{0}}}
\ge \pr{1/R}^{C\pr{m\log \brac{\pr{b + d}R} + \al^2}}
\ge \exp\brac{- C \pr{\log R}^2},
\end{align*}
where $C$ depends on $\la$, $\La$, $\mu$, $\al$, and $m$.

\subsection{The proof of Theorem \ref{LandisThm3}}
We first consider the case where $A = I$ and $q = \iny$.
Let $u$ be a solution to $\disp -\di\pr{\gr u + W_1 u} + W_2 \cdot \gr u + V u = 0$ in $\R^2$.
Choose $z_0 \in \R^2$ so that $\abs{z_0} = R$, define $u_R$ as above, and set $W_{i, R}\pr{z} =R W_i\pr{z_0 + Rz}$ for $i = 1,2$ and $V_{R}\pr{z} =R^2 V\pr{z_0 + Rz}$.
We have $\disp \norm{W_{i, R}}_{L^{\iny}\pr{B_{9/5}\pr{0}}}  \le \al_i R$ for $i = 1,2$ and $\disp \norm{V_{R}}_{L^{\iny}\pr{B_{9/5}\pr{0}}}  \le \al_0 R^2$.
Since
\begin{align*}
 - \di\brac{ \gr u_R\pr{z} + W_{1,R}\pr{z} u_R\pr{z}} + W_{2,R}\pr{z} \cdot \gr u_R\pr{z} + V_R\pr{z} u_R\pr{z}
&= 0,
\end{align*}
then $u_R$ satisfies a scaled version of the equation in $B_{9/5}$.
Let $\al = \max\set{\al_1, \al_2, \sqrt{\al_0}}$.
As above, \eqref{uBd} implies that \eqref{uRUpperBd} holds with $\hat C \ge 3 C_1/\al$, $d = \frac 9 5$ and $q = \iny$, while \eqref{normed} implies that \eqref{uRLowerBd} also holds with $b = 1$.
Assuming that $R$ is sufficiently large, Theorem \ref{OofV3}(a) applied to $u_R$ with $K = \al R$ and $C_0 = \hat C$ shows that
\begin{align*}
\norm{u}_{L^\iny\pr{{B_{1}(z_0)}}} 
&\ge \exp\pr{- C \al R \log R },
\end{align*}
where $C$ depends on $\hat C$.

For the cases where $q < \iny$ and $W_2, V \equiv 0$, the reduction described in Section \ref{curlCor} allows us to reduce to the setting described by Theorem \ref{LandisThm2} and the result is immediate.

\begin{appendix}

\section{Maximum principle}
\label{AppA}
\numberwithin{equation}{section}
\setcounter{equation}{0}  

Here we prove the maximum principle that is used in the proof of Lemma \ref{phiLem}.
In particular, we generalize some of the standard theorems from Chapter 8 of \cite{GT01} to extend to elliptic operators with singular lower order terms.
To this end, we employ roughly the same techniques, but with applications of H\"older and Sobolev inequalities to accommodate for the unbounded potential functions.

We assume that $\Om \su \R^2$ is open, connected and bounded with a $C^1$ boundary.
Assume that $A$ satisfies \eqref{ellip} and \eqref{ABd}, while $W_1 \in L^{q_1}\pr{B_d}$, $W_2 \in L^{q_2}\pr{B_d}$, and $V \in L^{p}\pr{B_d}$ for some $q_1, q_2 \in (2, \iny]$, $p \in (1, \iny]$.
Define
$$\mathcal{L}^* := -\di\pr{A^T \gr + W_2} + W_1 \cdot \gr + V.$$

\begin{thm}[cf. Theorem 8.1 in \cite{GT01}]
Let $u \in W^{1,2}\pr{\Om}$ weakly satisfy $\mathcal{L}^* u \le 0$ in $\Om$.
Assume that \eqref{pos2} holds.
Then
$$\sup_{\Om} u \le \sup_{\del \Om} u^+.$$
\label{maxPrinc}
\end{thm}

\begin{proof}
Let $v \in W^{1,2}_0\pr{\Om}$ be a non-negative function for which $u v \ge 0$ in $\Om$.
Since $\frac{2p}{p-1} > 2$, then the Sobolev inequality implies that $u, v \in L^{\frac{2p}{p-1}}\pr{\Om}$.
The H\"older inequality gives that $uv \in L^{p'}\pr{\Om}$. 
As $\frac{2q_i}{q_i-1} > 2$ for each $i$, the Sobolev inequality implies that $u, v \in L^{\frac{2q_i}{q_i-2}}\pr{\Om}$ and we may similarly conclude that $uv \in L^{q_i'}\pr{\Om}$. 
Note that $D\pr{u v} = D u \, v + u \, D v$.
Another application of the H\"older inequality in combination with the boundary information implies $uv \in W^{1, q_1'}_0\pr{\Om} \cap W^{1, q_2'}_0\pr{\Om} \cap L^{p'}\pr{\Om}$, so we may use it as a test function.
Since $\mathcal{L}^* u \le 0$, it follows from the definition that
\begin{align*}
\int_{\Om} A^T \gr u \cdot \gr v + W_2 \, u \cdot \gr v + W_1 \cdot \gr u \, v + V \, u \, v \le 0.
\end{align*}
Rearranging and using \eqref{pos2}, we see that
\begin{align*}
\int_{\Om} A^T \gr u \cdot \gr v + \pr{W_1 - W_2} \cdot \gr u \, v  \le - \int_{\Om} V \, u \, v + W_2 \cdot \gr\pr{u v} \le 0.
\end{align*}
Therefore,
\begin{align*}
\int_{\Om} A^T \gr u \cdot \gr v 
&\le \int_{\Om} \pr{W_2 - W_1} \cdot \gr u \, v
\le \sum_{i = 1}^2 \norm{W_i}_{L^{q_i}\pr{\Om}} \pr{\int_\Om \abs{\gr u \, v}^{q_i^\prime}}^{\frac 1 {q_i^\prime}}.
\end{align*}
In the case where $\norm{W_1}_{L^{q_1}\pr{\Om}} = 0$ and $\norm{W_2}_{L^{q_2}\pr{\Om}} = 0$, set $\disp l = \sup_{\del \Om} u^+$ and define $v = \max\set{u - l, 0}=(u-l)^+$.
The conclusion is then immediate.
Otherwise, choose $k$ so that $\disp l \le k \le \sup_{\Om} u$ and set $v = \pr{u - k}^+$.
(If no such $k$ exists, then we are finished.)
We have that $v \in W^{1,2}_0\pr{\Om}$ and 
$$D v = \left\{\begin{array}{ll} D u & u > k \\ 0 & u \le k.  \end{array}\right.$$
It follows from the last line of inequalities that
\begin{align*}
\int_{\Om} A^T \gr v \cdot \gr v 
\le \sum_{i = 1}^2 \norm{W_i}_{L^{q_i}\pr{\Om}} \pr{\int_\Ga \abs{D v \, v}^{q_i^\prime}}^{\frac 1 {q_i^\prime}}.
\end{align*}
where $\Ga = \supp Dv \su \supp v$.
The ellipticity condition in combination with a H\"older inequality then gives that
\begin{align*}
\norm{Dv}_{L^2\pr{\Om}}^2
\le \la^{-1} \norm{Dv}_{L^2\pr{\Om}} \sum_{i = 1}^2 \norm{W_i}_{L^{q_i}\pr{\Om}} \norm{v}_{L^{\frac{2q_i}{q_i-2}}\pr{\Ga}}
\end{align*}
or
\begin{align*}
\norm{Dv}_{L^2\pr{\Om}}
\le  \la^{-1} \sum_{i = 1}^2 \norm{W_i}_{L^{q_i}\pr{\Om}} \norm{v}_{L^{\frac{2q_i}{q_i-2}}\pr{\Ga}}.
\end{align*}
Now we apply the Sobolev inequality with some $2^* >\max\set{ \frac{2q_1}{q_1-2},  \frac{2q_2}{q_2-2}} > 2$ and the H\"older inequality to see that
\begin{align*}
C\norm{v}_{L^{2^*}\pr{\Om}} \le  \norm{Dv}_{L^2\pr{\Om}}
\le \la^{-1}  \norm{v}_{L^{2^*}\pr{\Om}} \sum_{i = 1}^2 \norm{W_i}_{L^{q_i}\pr{\Om}} \abs{\supp Dv}^{\frac{q_i-2}{2q_i} - \frac 1 {2^*}}.
\end{align*}
In particular, with $Q > 0$ chosen so that $\max\set{\abs{\supp Dv}^{\frac{q_i-2}{2q_i} - \frac 1 {2^*}}}_{i=1}^2 = \abs{\supp Dv}^Q$,
\begin{align*}
\abs{\supp Dv}
\ge \pr{ \frac{C\la}{ \norm{W_1}_{L^{q_1}\pr{\Om}} + \norm{W_2}_{L^{q_2}\pr{\Om}}} }^{\frac{1}{Q}}.
\end{align*}
Since this inequality is independent of $k$, it also holds as $k$ tends to $\disp \sup_{\Om} u$.
This means that the function $u$ must attain its supremum in $\Om$ on a set of positive measure, where at the same time $D u = 0$.
This contradiction implies that $\disp \sup_{\Om} u \le l$, as required.
\end{proof}

\section{Green's functions}
\label{AppB}
\numberwithin{equation}{section}

The purpose of this appendix is to establish a representation formula for solutions to non-homogeneous uniformly elliptic equations with vanishing Dirichlet boundary data.
To this end, we mimic the main technique presented in \cite{DHM16}, which is based on the ideas in \cite{HK07} and \cite{GW82}.
Since we only require such results for reasonably nice, bounded domains (balls), we assume throughout that $\Om \su \R^2$ is open, bounded, and connected.
Finally, we point out that the bounds given for the Green's functions are not sharp.
As is well known, pointwise logarithmic bounds for Green's functions in the plane are the best possible.
However, since we work using the methods of \cite{GW82}, \cite{HK07}, and \cite{DHM16}, and seek integrability properties for the Green's function instead of sharp pointwise bounds, our estimates will have power bounds.

We consider second-order, uniformly elliptic, bounded operators of divergence form with one first order term.
We use coercivity, the Caccioppoli inequality, and De Giorgi-Nash-Moser theory to establish existence, uniqueness, and a priori estimates for the Dirichlet Green's functions.

\subsection*{Notation and properties of solutions}

Let $\Om \su \R^2$ be open, bounded, and connected. 
In contrast to the main body of the article where we use the notation $z = \pr{x, y}$ to denote a point in $\R^2$, here we let $x, y$, etc. denote points in $\R^2$.

For any $x \in \Om$, $r > 0$, we define $\Om_r(x):=\Om \cap B_r(x)$ and $\Si_r\pr{x} := \del \Om \cap B_r\pr{x}$.
Let $C^\iny_c\pr{\Om}$ denote the set of all infinitely differentiable functions with compact support in $\Om$. 

For future reference, we mention that for $\Omega, U\subset \R^2$ open and connected, the assumption 
\begin{equation*}
u \in W^{1,2}(\Omega), \quad u={ 0} \mbox{ on } U\cap \partial\Om, 
\end{equation*}
is always meant in the weak sense of 
\begin{equation}
\label{eqB.1}
u \in W^{1,2}(\Omega) \mbox{ and } u \xi \in W^{1,2}_0(\Omega) \mbox{ for any } \xi \in C_c^\infty(U).
\end{equation}
This definition of (weakly) vanishing on the boundary is independent of the choice of $U$.  
Indeed, suppose $V$ is another open and connected subset of $\R^2$ such that $V \cap \partial \Om = U \cap \partial \Om$ and let $\xi \in C_c^\infty\pr{V}$.
Choose $\psi \in C_c^\infty\pr{U \cap V}$ such that $0\le \psi \le 1$ and $\psi \equiv 1$ on the support of $\xi$ in some neighborhood of the boundary.
Then $\xi \pr{1-\psi}|_{\Om}\in C_c^\infty\pr{\Om}$, so that $u \xi \pr{1-\psi} \in W^{1,2}_0\pr{\Om}$.
Additionally, $\xi \psi \in C_c^\infty\pr{U}$, so by \eqref{eqB.1}, $u \xi \psi \in W^{1,2}_0 \pr{\Om}$.
Therefore, $u \xi = u \xi \psi + u \xi \pr{1-\psi} \in W^{1,2}_0\pr{\Om}$, as desired.

Let $A = \pr{a_{ij}}_{i, j= 1}^{2}$ be bounded, measurable coefficients defined on $\Om$.
We assume that $A$ satisfies an ellipticity condition described by \eqref{ellip} and the boundedness assumption given in \eqref{ABd}.
Choose $W \in L^q\pr{\Om}$ for some $q \in (2, \iny]$ so that \eqref{pos2} holds with $q_i = q$.
The non-homogeneous second-order operator is
\begin{align}
\widetilde{L} 
&= -\di \pr{A \gr} + W \cdot \gr 
\label{eqB.2}
\end{align}
and the adjoint operator to $\widetilde{L}$ is given by
\begin{align}
\widetilde{L}^*
&= -\di \pr{A^T \gr + W} .
\label{eqB.3}
\end{align}
All operators are understood in the sense of distributions on $\Omega$. 
Specifically, for every $u \in W^{1,2}\pr{\Om}$ and $v \in C_c^\infty\pr{\Om}$, we use the naturally associated bilinear form and write the action of the functional ${\widetilde L}u$ on $v$ as
\begin{align}
({\widetilde L}u, v)={B}\brac{u, v} 
&= \int_\Om A \gr u \cdot \gr v + W \cdot \gr u \, v .
\label{eqB.4}
\end{align}
It is not hard to check that for such $u, v$ and for the coefficients as described above, the bilinear form above is well-defined and finite.  
Similarly, ${B}^*\brac{\cdot, \cdot}$ denotes the bilinear operator associated to $\widetilde{L}^*$, given by
\begin{align}
({\widetilde L}^*u, v)={B}^*\brac{u, v} 
&= \int  A^T \gr u \cdot \gr v + W \, u \cdot \gr v .
\label{eqB.5}
\end{align}
Clearly,
\begin{equation}
{B}\brac{v,u}=\, {B}^*\brac{u,v}.
\label{eqB.6}
\end{equation}
For any distribution ${\bf F} = \pr{f, G}$ on $\Omega$ and $u$ as above, we always understand $\widetilde{ L}u= \bf F $ on $\Omega$ in the weak sense, that is, as ${B}\brac{u,v}= {\bf F}(v)$ for all $v\in C_c^\infty\pr{\Om}$. 
Typically $f$ and $G$ will be elements of some $L^p(\Omega)$ spaces, so the action of $\bf F$ on $v$ is then simply $\disp \int f v + G \cdot Dv.$ 
The identity $\widetilde{ L}^*u= \bf F $ is interpreted similarly.

\begin{rem}
Assumption \eqref{ABd} and that $W \in L^q\pr{\Om}$, $q > 2$, in combination with H\"older and Sobolev inequalities imply that there exists $\La > 0$ so that for every $u, v \in W^{1,2}_0\pr{\Om}$,
\begin{align}
|{B}\brac{u, v}| &\le \La \norm{u}_{W^{1,2}\pr{\Om}} \norm{v}_{W^{1,2}\pr{\Om}}.
\label{eqB.7}
\end{align}
Since \eqref{pos2} holds, then \eqref{ellip} and the Poincar\'e inequality implies that there exists $\ga > 0$ so that for every $u \in W^{1,2}_0\pr{\Om}$
\begin{equation}
{B}\brac{u, u} \ge \ga \norm{u}_{W^{1,2}\pr{\Om}}^2.
\label{eqB.8}
\end{equation}
\end{rem}

Now we describe the important properties of solutions to either $\widetilde{L} u = 0$ or $\widetilde{L}u = \bf F$ that will be employed in the constructions below. We will use the following version of Moser (boundary) boundedness.

\begin{lem}{\rm\cite[Lemma~5.1]{DHM16}}
\label{l5.1}
Let $\Om \su \R^2$ be open and connected.
Let $u \in W^{1,2}\pr{\Om_{2R}}$ satisfy $u = 0$ along $\Si_{2R}$.
Let $f \in L^{\ell}\pr{\Om_{R}}$ for some $\ell \in \pb{ 1, \iny}$, $G \in L^{m}\pr{\Om_{R}}$ for some $m \in \pb{ 2, \iny}$ and assume that $\widetilde{L} u \le - \di G + f$ in $\Om_{R}$ weakly in the sense that for any $\vp \in W^{1,2}_{0}\pr{\Om_{R}}$ such that $\vp \ge 0$ in $\Om_{R}$, we have 
\begin{align*}
{B}\brac{u, \vp} \le \int G \cdot \gr \vp + f \vp .
\end{align*}
Then $u^+ \in L_{loc}^\iny\pr{\Om_R}$ and for any $r < R$, $s > 0$,
\begin{align}
\sup_{\Om_{r}} u^+ 
&\le \frac{C}{\pr{R- r}^{\frac 2 {s}}} \norm{u^+}_{L^{s}\pr{\Om_R}} 
+ c_{s} \brac{R^{2 - \frac 2 {\ell}}\norm{f}_{L^{\ell}\pr{\Om_R}} + R^{1 - \frac 2 {m}}\norm{G}_{L^{m}\pr{\Om_R}}},
\label{eqB.9}
\end{align}
where $C = C\pr{q, s, \ell, m, \ga,  \La, \norm{W}_{L^{q}\pr{\Om_R}}}$ and $c_{s}$ depends only on $s$.
Note that all of the constants are independent of $R$.
\end{lem}

\begin{rem}
Because of assumption \eqref{pos2}, the conclusion of Lemma \ref{l5.1} also holds for the operator $\widetilde{L}^*$.
\end{rem}

The following Caccioppoli inequality will be used in our constructions.

\begin{lem}{\rm\cite[Lemma~4.1]{DHM16}}
If $u \in W_0^{1,2}\pr{\Om}$ is a weak solution to $\widetilde{L} u = 0$ in $\Omega$ and $\zeta\in C^\infty(\R^2)$, then
\begin{align}
\int \abs{D u}^2 \zeta^2 \le C \int \abs{u}^2 \abs{D \zeta}^2,
\label{eqB.10}
\end{align}
where $C$ is a constant that depends on $\ga$, $\La$, $q$, $\norm{W}_{L^{q}\pr{\Om}}$, but $C$ is independent of the sets on which $\zeta$ and $D\zeta$ are supported. 
\label{Cacc}
\end{lem}

We also rely on a lemma regarding the H\"older continuity of solutions.

\begin{lem}{\rm\cite[Lemma~6.6]{DHM16}}
Let $u \in W^{1,2}\pr{B_{2R_0}}$ be a solution in the sense that $\disp {B}\brac{u, \vp} = 0$ for any $\vp \in W^{1,2}_0\pr{B_{R_0}}$.
Then there exists $\eta \in \pr{0,1}$, such that for any $R \le R_0$, if $x, y \in B_{R/2}$
\begin{align}
\abs{u\pr{x} - u\pr{y}}
&\le C_{R_0} \pr{\frac{\abs{x-y}}{R}}^\eta  \pr{\fint_{B_{R}} \abs{u}^{2^*} }^{\frac 1 {2^*}}.
\label{eqB.11}
\end{align}
\end{lem}

\subsection*{Green's functions}

This subsection resembles the work done in \cite{HK07} and \cite{DHM16}.
We use the properties of our operator as well as the properties of solutions to $\widetilde{L}u = \bf F$ or $\widetilde{L}^* u = \bf F$ described above to establish existence, uniqueness, and a collection of a priori estimates for the Dirichlet Green's function associated to $\Om \su \R^2$.
We follow closely the arguments in \cite{HK07} and \cite{DHM16}, adapting to $n = 2$.
As previously mentioned, our estimates are not sharp since we do not obtain logarithmic bounds for the Green's functions.

First, we clarify the meaning of the Green's function.

\begin{defn} 
Let $\Omega$ be an open, connected, bounded subset of $\R^2$.
We say that the function $\Ga\pr{x,y}$ defined on the set $\set{(x,y)\in \Omega \times \Omega: x\ne y}$ is the \textbf{Green's function} of $\widetilde{L}$ if it satisfies the following properties:
\begin{itemize}
\item[1)] $\Ga\pr{\cdot, y}$ is locally integrable and $\widetilde{L}\Ga\pr{\cdot, y} = \de_y I$ for all $y \in \Omega$ in the sense that for every $\phi \in C^\iny_c\pr{\Omega}$,
\begin{equation}
B\brac{\Ga\pr{\cdot, y}, \phi} = \phi\pr{y}.
\label{eqB.12}
\end{equation}
\item[2)] For all $y\in \Om$ and $r > 0$, $\Ga\pr{\cdot, y}\in W^{1,2}\pr{\Om \setminus  \Omega_r\pr{y}}$.  
In addition, $\Ga(\cdot,  y)$ vanishes on $\partial \Omega$ in the sense that for every $\zeta \in C_c^\infty\pr{\Om}$ satisfying $\zeta \equiv 1$ on $B_r(y)$ for some $r>0$, we have
\begin{equation}
(1-\zeta)\Ga(\cdot, y) \in W_0^{1,2}\pr{\Om \setminus \Omega_r(y)}.
\label{eqB.13}
\end{equation}
\item[3)] For some $\ell_0 \in \pb{1, \iny}$ and $m_0 \in \pb{2, \iny}$, and any $f \in L^{\ell_0}\pr{\Om}$, $G \in L^{m_0}\pr{\Om}$, the function $u$ given by
\begin{equation}
u\pr{y} = \int_{\Om} \brac{ \Ga\pr{x,y} f\pr{x} + D_x \Ga\pr{x, y} \cdot G\pr{x} } dx
\label{eqB.14}
\end{equation}
belongs to $W^{1,2}_0\pr{\Om}$ and satisfies $\widetilde{L}u = f - \di G$ in the sense that for every $\phi  \in C^\iny_c\pr{\Om}$,
\begin{align}
&B\brac{u, \phi}
= \int_{\Om} f \phi + G \cdot D \phi.
\label{eqB.15}
\end{align}
\end{itemize}
We say that the function $\Ga\pr{x,y}$ is the \textbf{continuous Green's function} if it satisfies the conditions above and is also continuous.
\label{d3.1}
\end{defn}

We show here that there is at most one Green's function.  
In general, we mean uniqueness in the sense of Lebesgue, i.e. almost everywhere uniqueness.
However, when we refer to the continuous Green's function, we mean true pointwise equivalence.

Assume that $\Ga$ and $\widetilde \Ga$ are Green's functions satisfying Definition \ref{d3.1}. 
Then, for all $f \in L^\infty\pr{\Om}$, the functions $u$ and $\widetilde{u}$ given by
\begin{equation*}
u\pr{y}
=\int_{\Om} \Ga\pr{x,y} f\pr{x} dx, \quad 
\widetilde{u}\pr{y}
=\int_{\Om} \widetilde{\Ga}\pr{x,y} f\pr{x} dx
\end{equation*}
satisfy
\begin{equation*}
\widetilde{L}^* \pr{u - \widetilde{u}}=0 \quad \text{in } \Om
\end{equation*}
and $u-\widetilde{u}\in W^{1,2}_0(\Om)$.  
By uniqueness of solutions ensured by the Lax-Milgram lemma, $u-\widetilde{u}\equiv 0$.  
Thus, for a.e. $x\in \Om$,
\begin{equation*}
\int_{\Om} \brac{\Ga(x,y)-\widetilde{\Ga}(x,y)} f(x) \, dx=0, \quad 
\forall f\in L^\infty(\Om).
\end{equation*}
Therefore, $\Ga = \widetilde \Ga$ a.e. in $\set{x \ne y}$.
If we further assume that $\Ga$ and $\widetilde \Ga$ are continuous Green's functions, then we conclude that $\Ga \equiv \widetilde \Ga$ in $\set{x \ne y}$.

\begin{thm}
\label{t3.2} 
Let $\Om$ be an open, connected, bounded subset of $\R^2$.  
Then there exists a unique continuous Green's function $\Ga(x,y)$, defined in $\set{x,y\in \Om, x\neq y}$, that satisfies Definition \ref{d3.1}.  
We have $\Ga(x,y)=\Ga^*(y,x)$, where $\Ga^*$ is the unique continuous Green's function associated to $\widetilde{L}^*$.  
Furthermore, $\Ga(x,y)$ satisfies the following estimates:
\begin{align}
& \norm{\Ga\pr{\cdot, y}}_{W^{1,2}\pr{\Om \setminus \Om_r\pr{y}}}
+\norm{\Ga\pr{x,\cdot}}_{W^{1,2}\pr{\Om \setminus \Om_r\pr{x}}} 
\le C r^{-\eps}, \qquad \forall r>0,
\label{eqB.16} \\
& \norm{\Ga\pr{\cdot, y}}_{L^{s}\pr{\Om_r\pr{y}}}
+\norm{\Ga\pr{x,\cdot}}_{L^{s}\pr{\Om_r\pr{x}}} 
\le C_s r^{- \eps+ \frac{2}{s}}, \qquad \forall r>0, \quad \forall s \in [1,\iny),
\label{eqB.17} \\
& \norm{D \Ga\pr{\cdot, y}}_{L^{s}\pr{\Om_r\pr{y}}}
+\norm{D \Ga\pr{x,\cdot}}_{L^{s}\pr{\Om_r\pr{x}}} 
\le C_s r^{-1-\eps +\frac{2}{s}}, \qquad \forall r>0, \quad \forall s \in [ 1, 2), 
\label{eqB.18} \\
& \abs{\set{x \in \Om : \abs{\Ga\pr{x,y}} > \tau }}
+\abs{\set{y \in \Om : \abs{\Ga\pr{x,y}} > \tau }} 
\le C \tau^{- \frac{2}{\eps}}, \qquad \forall \tau >0,
\label{eqB.19} \\
& \abs{\set{x \in \Om : \abs{D_x \Ga\pr{x,y}} > \tau }}
+\abs{\set{y \in \Om : \abs{D_y \Ga\pr{x,y}} > \tau }} 
\le C \tau^{- \frac{2}{1+\eps}}, \qquad \forall \tau >0,
\label{eqB.20} \\
& \abs{\Ga\pr{x,y}} \le C |x-y|^{-\eps} \qquad \forall x \ne y,
\label{eqB.21} 
\end{align}
where in each case, $\eps > 0$ is some arbitrarily small number that may vary from line to line.
Moreover, each constant depends on $\ga$, $\La$, $\eps$, and the constants from \eqref{eqB.9} and \eqref{eqB.10}, and each $C_s$ depends additionally on $s$.
Moreover, for any $0<R\le R_0$,
\begin{align}
& \abs{\Ga\pr{x,y} - \Ga\pr{z,y}} 
\le C_{R_0} C \pr{\frac{|x-z|}{R}}^\eta R^{-\eps},
\label{eqB.22}
\end{align}
whenever $|x-z|<\frac{R}{2}$ and
\begin{align}
& \abs{\Ga\pr{x,y} - \Ga\pr{x,z}} 
\le C_{R_0} C \pr{\frac{|y-z|}{R}}^\eta R^{-\eps},
\label{eqB.23}
\end{align}
whenever $|y-z|<\frac{R}{2}$, where $C_{R_0}$ and $\eta=\eta(R_0)$ are the same as in \eqref{eqB.11}.
\end{thm}

\begin{pf}[Proof of Theorem~\ref{t3.2}] 
Let $u \in W^{1,2}_0\pr{\Om}$. 
Fix $y\in \Om$, $\rho > 0$, and consider the linear functional
$$u \mapsto \fint_{B_\rho(y)} u.$$
By  the H\"older inequality and Sobolev embedding with $2^* \in \pr{2, \iny}$,
\begin{align}
\abs{ \fint_{B_\rho(y)} u}
&\le \frac{1}{\abs{B_\rho\pr{y}}} \int_{B_\rho\pr{y}} \abs{u}
\le \abs{B_\rho\pr{y}}^{-\frac{1}{2^*}} \pr{\int_{\Om} \abs{u}^{2^*} }^{\frac{1}{2^*}} 
\le c \abs{B_\rho\pr{y}}^{-\frac{1}{2^*}} \pr{\int_{\Om} \abs{Du}^{2} }^{\frac{1}{2}}  \nonumber \\
&\le C \rho^{-\frac{2}{2^*}} \norm{u}_{W^{1,2}\pr{\Om}}.
\label{eqB.24}
\end{align}
Therefore, the functional is bounded on $W^{1,2}_0\pr{\Om}$, and by the Lax-Milgram theorem there exists a unique $\Ga^\rho = \Ga^\rho_y = \Ga^\rho\pr{\cdot, y} \in W^{1,2}_0\pr{\Om}$ satisfying 
\begin{equation}
\label{eqB.25}
{B}[\Ga^\rho, u] = \fint_{B_\rho(y)} u = \frac{1}{\abs{B_\rho(y)}} \int_{B_\rho(y)} u, \quad \forall \, u \in W^{1,2}_0\pr{\Om}.
\end{equation}
By the coercivity of $A$ given by \eqref{eqB.8} along with \eqref{eqB.24}, we obtain,
$$\ga \norm{\Ga^\rho}^2_{W^{1,2}\pr{\Om}}
\le {B}\brac{\Ga^\rho,\Ga^\rho}
= \abs{\fint_{B_\rho(y)} \Ga^\rho}
\le C \rho^{-\frac{2}{2^*}} \norm{\Ga^\rho}_{W^{1,2}\pr{\Om}}$$
so that for any $\eps \in \pr{0, 1}$,
\begin{equation}
\label{eqB.26}
\norm{D\Ga^\rho}_{L^2\pr{\Om}} 
\le C \rho^{-\eps}.
\end{equation}

For some $\ell_0 \in \pb{1, \iny}$ and $m_0 \in \pb{2, \iny}$, let $f \in L^{\ell_0}\pr{\Om}$ and $G \in L^{m_0}\pr{\Om}$.
For any $\ell \in \pb{1, \ell_0}$ and any $m \in \pb{1, m_0}$, it is clear that that $f \in L^\ell\pr{\Om}$ and $G \in L^m\pr{\Om}$ since $\Om$ is bounded. 
Consider the linear functionals
\begin{align*}
& W^{1,2}_0\pr{\Om} \ni w \mapsto \int_\Om f w \\
& W^{1,2}_0\pr{\Om} \ni w \mapsto \int_\Om G \cdot D w.
\end{align*}
The first functional is bounded on $W^{1,2}_0\pr{\Om}$ since for every $w \in W^{1,2}_0\pr{\Om}$ and every $\ell \in \pb{1, \ell_0}$,
\begin{align}
\abs{\int_{\Om} f \, w }
&\le \norm{f}_{L^{\ell}\pr{\Om}} \norm{w}_{L^{2^*}\pr{\Om}}\abs{\supp f}^{1 - \frac {1}{2^*} - \frac 1{\ell}} 
\le c \norm{f}_{L^{\ell}\pr{\Om}}\abs{\supp f}^{1 - \frac {1}{2^*} - \frac 1{\ell}} \norm{Dw}_{L^{2}\pr{\Om}},
\label{eqB.27}
\end{align}
where we have again used Sobolev embedding with some $2^* \in \pr{2, \iny}$.
Similarly, we see that the second functional is also bounded on $W^{1,2}_0\pr{\Om}$ since for every $m \in \pb{2, m_0}$
\begin{align}
\abs{\int_\Om G \cdot D w }
&\le \norm{G}_{L^{m}\pr{\Om}}\abs{\supp G}^{ \frac {1}{2} - \frac 1{m}}  \norm{Dw}_{L^{2}\pr{\Om}}.
\label{eqB.28}
\end{align}
Once again, by Lax-Milgram, we obtain $u_1, u_2 \in W^{1,2}_0\pr{\Om}$ such that
\begin{equation}
\label{eqB.29}
{B}^*\brac{u_1, w}=\int_\Om f \, w, \quad \forall \, w \in W^{1,2}_0\pr{\Om}
\end{equation}
and
\begin{equation}
\label{eqB.30}
{B}^*\brac{u_2, w}=\int_\Om G \cdot D w, \quad \forall \, w \in W^{1,2}_0\pr{\Om}.
\end{equation}

Set $w=u_1$ in \eqref{eqB.29} and use the coercivity assumption, \eqref{eqB.8}, for ${B}^*$ along with \eqref{eqB.27} to get
\begin{equation}
\label{eqB.31}
\norm{Du_1}_{L^2\pr{\Om}} \le C \norm{f}_{L^{\ell}\pr{\Om}}\abs{\supp f}^{ 1 - \frac {1}{2^*} - \frac 1{\ell}}.
\end{equation}
With $w=u_2$ in \eqref{eqB.30}, we similarly obtain from \eqref{eqB.28} that
\begin{equation}
\label{eqB.32}
\norm{Du_2}_{L^2\pr{\Om}} \le C \norm{G}_{L^{m}\pr{\Om}}\abs{\supp G}^{ \frac {1}{2} - \frac 1{m}}.
\end{equation}
Also, if we take $w = \Ga^\rho$ in \eqref{eqB.29} and \eqref{eqB.30}, we get
\begin{equation}
\label{eqB.33}
\int_\Om f \, \Ga^\rho
= {B}^*[u_1, \Ga^\rho]
= {B}[\Ga^\rho, u_1] 
 = \fint_{B_\rho(y)} u_1,
\end{equation}
and
\begin{equation}
\label{eqB.34}
\int_\Om G \cdot D \Ga^\rho
= {B}^*[u_2, \Ga^\rho]
= {B}[\Ga^\rho, u_2] 
 = \fint_{B_\rho(y)} u_2.
\end{equation}
In particular, with $u:= u_1 + u_2$, we see that
\begin{equation}
\label{eqB.35}
\int_\Om f \, \Ga^\rho + G \cdot D \Ga^\rho
 = \fint_{B_\rho(y)} u.
\end{equation}

Now assume that $f$ and $G$ are supported in $\Om_r(y)$, for some $r > 0$.
Let $u_1, u_2$ be as in \eqref{eqB.29}, \eqref{eqB.30}, respectively.  
Since $u_1, u_2 \in W^{1,2}_0\pr{\Om}$, then $u_1, u_2 \in W^{1,2}\pr{\Om_{2r}}$ and $u_1, u_2 = 0$ on $\Si_{2r}$ so that Lemma \ref{l5.1} is applicable. 
Then, by \eqref{eqB.9} with some $s = 2^* \in \pr{2, \iny}$
$$\norm{u_1}^2_{L^\infty\pr{\Om_{r/2}\pr{y}}} 
\le C \pr{ r^{-\frac{4}{2^*}} \norm{u_1}^2_{L^{2^*}\pr{\Om_r\pr{y}}} + r^{4 - \frac{4}{\ell}} \norm{f}^2_{L^{\ell}\pr{\Om_r\pr{y}}}}$$
and
$$\norm{u_2}^2_{L^\infty\pr{\Om_{r/2}\pr{y}}} 
\le C \pr{ r^{-\frac{4}{2^*}} \norm{u_2}^2_{L^{2^*}\pr{\Om_r\pr{y}}} + r^{2 - \frac{4}{m}} \norm{G}^2_{L^{m}\pr{\Om_r\pr{y}}}}.$$
By Sobolev embedding and \eqref{eqB.31} with $\supp f \subset \Om_r\pr{y}$,
\begin{align*}
\norm{u_1}^2_{L^{2^*}\pr{\Om_r\pr{y}}} 
&\le \norm{u_1}^2_{L^{2^*}\pr{\Om}} 
\le C \norm{Du_1}^2_{L^2\pr{\Om}}
\le C \norm{f}^2_{L^{\ell}\pr{\Om}}\abs{ \Om_r\pr{y}}^{2 - \frac 2 {2^*} - \frac 2{\ell}} 
\le C r^{4 - \frac {4}{2^*} - \frac 4 \ell}  \norm{f}^2_{L^{\ell}\pr{\Om}} .
\end{align*}
Combining the previous two inequalities for $u_1$, we see that
$$\norm{u_1}^2_{L^\infty\pr{\Om_{r/2}\pr{y}}} 
\le C  \pr{1 + r^{-\frac{8}{2^*}}} r^{4 - \frac{4}{\ell}} \norm{f}^2_{L^{\ell}\pr{\Om_r\pr{y}}}.$$
For any $\ell \in \pb{1, \ell_0}$, choose $2^* \in \pr{2, \iny}$ so that $\frac 2 {2^*} < 1 - \frac 1 \ell$.
Since $\Om$ is bounded, then so too is $r$, and we have
\begin{align}
\norm{u_1}_{L^\infty \pr{\Om_{r/2}\pr{y}}} 
&\le C r^{2-\frac{2}{\ell} - \frac 4 {2^*}} \norm{f}_{L^{\ell}\pr{\Om}} 
= C r^{2-\frac{2}{\ell} - \frac 4 {2^*}} \norm{f}_{L^{\ell}\pr{\Om_r\pr{y}}}.
\label{eqB.36}
\end{align}
Mimicking the argument with $u_2$, $G$ and \eqref{eqB.32}, we see that
$$\norm{u_2}^2_{L^\infty\pr{\Om_{r/2}\pr{y}}} 
\le C  \pr{1 + r^{-\frac{4}{2^*}}} r^{2 - \frac{4}{m}} \norm{G}^2_{L^{m}\pr{\Om_r\pr{y}}}.$$
Now for any $m \in \pb{2, m_0}$, choose $2^* \in \pr{2, \iny}$ so that $\frac 2 {2^*} < 1 - \frac 2 m$ and we conclude that
\begin{align}
\norm{u_2}_{L^\infty \pr{\Om_{r/2}\pr{y}}} 
&\le C r^{1-\frac{2}{m} - \frac 2 {2^*}} \norm{G}_{L^{m}\pr{\Om}} 
= C r^{1-\frac{2}{m} - \frac 2 {2^*}} \norm{G}_{L^{m}\pr{\Om_r\pr{y}}}.
\label{eqB.37}
\end{align}

By \eqref{eqB.33} and \eqref{eqB.36}, if $\rho \le r/2$, we have that for every $\ell \in \pb{1, \ell_0}$,
$$\abs{\int_{\Om_r\pr{y}} f \, \Ga^\rho}
= \abs{\int_\Om f \, \Ga^\rho}
\le \fint_{B_\rho(y)} \abs{u_1}
\le \norm{u_1}_{L^\infty(B_\rho(y))}
\le \norm{u_1}_{L^\infty(\Om_{r/2}(y))} 
\leq Cr^{2-\frac{2}{\ell} - \frac 4 {2^*}} \norm{f}_{L^{\ell}\pr{\Om_r\pr{y}}}.$$
By duality, since we can take $\ell_0 = \iny$, this implies that for $r > 0$,
\begin{equation}
\label{eqB.38}
\norm{\Ga^\rho}_{L^s\pr{\Om_r\pr{y}}}
\le C r^{\frac{2}{s} - \eps}, \quad \mbox{ for all } \rho \le \frac{r}{2}, \quad \forall s \in \brp{1,\iny}.
\end{equation}
We similarly conclude that
\begin{equation}
\label{eqB.39}
\norm{D\Ga^\rho}_{L^s\pr{\Om_r\pr{y}}}
\le C r^{\frac{2}{s} - 1 - \eps}, \quad \mbox{ for all } \rho \le \frac{r}{2}, \quad \forall s \in \brp{1,2}.
\end{equation}
Note that in both cases, $\eps \in \pr{0, 1}$ is chosen so that the power on $r$ is positive.

Fix $x\ne y$ and set $r:= \frac{4}{3} |x-y|$.  
For $\rho \le r/2$, $\Ga^\rho$ is a weak solution to $\widetilde{L} \Ga^\rho=0$ in $\Om_{r/4}(x)$.  
Moreover, since $\Ga^\rho \in W^{1,2}_0\pr{\Om}$, then $\Ga^\rho \in W^{1,2}\pr{\Om_{r/2}\pr{x}}$ and  $\Ga^\rho = 0$ on $\Si_{r/2}\pr{x}$, so we may use Lemma \ref{l5.1}.
Thus, applying \eqref{eqB.9} and \eqref{eqB.38} with $s=1$, we get for a.e. $x \in \Om$ as above, 
\begin{equation}
\label{eqB.40}
\abs{\Ga^\rho(x)}
\le C r^{-2} \norm{\Ga^\rho}_{L^1\pr{\Om_{r/4}\pr{x}}} 
\le C r^{-2} \norm{\Ga^\rho}_{L^1\pr{\Om_{r}\pr{y}}} 
\le Cr^{-\eps}
\approx  |x-y|^{-\eps}.
\end{equation}
 
Now, for any  $r > 0$ and $\rho\le r/{2}$, let $\zeta$ be a cut-off function such that
\begin{equation}
\label{eqB.41}
\zeta \in C^{\iny}(\R^n), \quad
0\le \zeta \le 1, \quad 
\zeta \equiv 1 \text{ outside $B_{r}(y)$}, \quad 
\zeta \equiv 0 \text{ in $B_{r/2}(y)$}, \quad 
\text{and} \;  \abs{D \zeta} \le C/r.
\end{equation}
Then the Caccioppoli inequality of Lemma \ref{Cacc} implies that
\begin{equation}
\label{eqB.42}
\int_{\Om} \zeta^2 \abs{D\Ga^\rho}^2 
\le C \int_{\Om} \abs{D\zeta}^2 \abs{\Ga^\rho}^2 
\le C r^{-2} \int_{\Om_{r}(y)\setminus \Om_{r/2}(y)} \abs{\Ga^\rho}^2, \quad \forall \rho \le \frac{r}{2}.
\end{equation}
Combining \eqref{eqB.42} and \eqref{eqB.40}, we have for all $r > 0$ and $\zeta$ as above,
\begin{align}
\label{eqB.43}
\begin{split}
\int_{\Om} \abs{D(\zeta \Ga^\rho)}^2 
&\le 2 \int_{\Om} \zeta^2 \abs{D\Ga^\rho}^2 + 2 \int_{\Om} \abs{D\zeta}^2 \abs{\Ga^\rho}^2 \\
&\le C r^{-2} \int_{\Om_{r}(y)\setminus \Om_{r/2}(y)} \abs{\Ga^\rho}^2 
\le C r^{-2\eps}
, \quad \forall \rho\le \frac{r}{2}.
\end{split}
\end{align}
It follows from Sobolev embedding with arbitrary $2^* \in \pr{2, \iny}$ and \eqref{eqB.43} that for $r > 0$,
\begin{equation*}
\int_{\Omega \setminus \Om_r(y)} \abs{\Ga^\rho}^{2^*} 
\le \int_{\Omega} \abs{\zeta \Ga^\rho}^{2^*} 
\le c \pr{\int_{\Omega} \abs{D\pr{\zeta \Ga^\rho}}^2}^{\frac{2^*}{2}} 
\le C r^{-2^*\eps}
, \quad \forall \rho \le \frac{r}{2}.
\end{equation*}
On the other hand, if $\rho > \frac r 2$, then \eqref{eqB.26} implies that
\begin{equation*}
\int_{\Om \setminus \Om_r(y)} \abs{\Ga^\rho}^{2^*} 
\le \int_{\Om} \abs{\Ga^\rho}^{2^*} 
\le c \pr{ \int_{\Om} \abs{D\Ga^\rho}^2 }^{\frac{2^*}{2}} 
\le C r^{-2^*\eps}.
\end{equation*}
Therefore, combining the previous two results, we have that for any $\eps \in \pr{0, 1}$
\begin{equation}
\label{eqB.44}
\int_{\Om\setminus \Om_r(y)} |\Ga^\rho|^{2^*} 
\le C r^{-2^*\eps}, \quad \forall \, r, \rho > 0.
\end{equation}

For any $\eps \in \pr{0 ,1}$, $2^* \in \pr{2, \iny}$, fix $\tau > 0$.  
Let $A_\tau=\set{x\in \Om: \abs{\Ga^\rho}>\tau}$ and set $r=\tau^{-\frac{2^*}{2 + 2^* \eps}}$.  
Then, using \eqref{eqB.44}, we see that if $\rho > 0$,
\begin{equation*}
\abs{A_\tau\setminus \Om_r(y)}
\le \tau^{-2^*}\int_{A_\tau\setminus \Om_r(y)} \abs{\Ga^\rho}^{2^*} 
\le C \tau^{-2^*} r^{-2^*\eps}
= C \tau^{-\frac{2^* 2}{2 + 2^* \eps}}.
\end{equation*}
Since $\abs{A_\tau \cap \Om_r(y)} \le \abs{\Om_r(y)} \le Cr^2 = C \tau^{-\frac{2^* 2}{2 + 2^* \eps}}$, we have
\begin{equation}
\label{eqB.45}
\abs{\set{ x\in \Om: \abs{\Ga^\rho(x)} > \tau }}
\le C  \tau^{-\frac{2^* 2}{2 + 2^* \eps}} \quad \forall \, \rho > 0.
\end{equation}

Fix $r > 0$ and let $\zeta$ be as in \eqref{eqB.41}.  
Then \eqref{eqB.43} gives
\begin{equation*}
 \int_{\Om \setminus \Om_{r}(y)} \abs{D\Ga^\rho}^2 
 \le C r^{-2\eps}
 , \quad \forall r > 0, \quad \forall \rho\le \frac{r}{2}.
\end{equation*}
Now, if $\rho > \frac r 2$, we have from \eqref{eqB.26} that
\begin{equation*}
 \int_{\Om\setminus \Om_r(y)} \abs{D\Ga^\rho}^2 
 \le \int_{\Om} \abs{D\Ga^\rho}^2 
 \le C \rho^{-2\eps}
 \le C r^{-2\eps}.
\end{equation*}
Combining the previous two results yields
\begin{equation}
\label{eqB.46}
 \int_{\Om\setminus \Om_r(y)} \abs{D\Ga^\rho}^2 
 \le C r^{-2\eps}, \quad \forall \, r, \rho > 0.
\end{equation}

Fix $\tau > 0$.  
Let $A_\tau = \set{x \in \Om : \abs{D\Ga^\rho} > \tau}$ and set $r=\tau^{-\frac{1}{1+\eps}}$.  
Then, using \eqref{eqB.46}, we see that if $\rho > 0$,
$$\abs{A_\tau \setminus \Om_r(y)} 
\le \tau^{-2} \int_{A_\tau\setminus \Om_r(y)} \abs{D\Ga^\rho}^2 
\le C \tau^{-2} r^{-2\eps}
= C \tau^{-\frac{2}{1+\eps}}.$$
Since $\abs{A_\tau \cap \Om_r(y)} \le C r^2 = C\tau^{-\frac{2}{1+\eps}}$, then
\begin{equation}
\label{eqB.47}
 \abs{\set{x\in \Om: \abs{D\Ga^\rho(x)} > \tau }} \le C \tau^{-\frac{2}{1+\eps}}
 \quad \forall \, \rho > 0.
\end{equation}

For any $\si> 0$ and $s > 0$, we have
\begin{equation*}
\int_{\Om_r(y)} \abs{D\Ga^\rho}^s 
\le \si^s \abs{\Om_r(y)} + \int_{\set{\abs{D\Ga^\rho}>\si}} \abs{D \Ga^\rho}^s.
\end{equation*}
By \eqref{eqB.47}, for $s \in \pr{0,\frac{2}{1+\eps}}$ and $\rho > 0$, 
\begin{align*}
\int_{\set{|D\Ga^\rho|>\si}} \abs{D\Ga^\rho}^s 
&= \int_0^\infty s \tau^{s-1} \abs{\set{\abs{D\Ga^\rho}>\max\set{\tau,\si}}} d\tau \\
&\le C \si^{-\frac{2}{1+\eps}} \int_0^\si s \tau^{s-1} \, d \tau
+ C\int_\si^\infty s \tau^{s-1 -\frac{2}{1+\eps}} \, d\tau
= C \pr{1-\frac{s}{s-\frac{2}{1+\eps}}} \si^{s-\frac{2}{1+\eps}}.
\end{align*}
Therefore, taking $\si=r^{-\pr{1+\eps}}$, we conclude that
\begin{equation}
\label{eqB.48}
 \int_{\Om_r(y)} \abs{D\Ga^\rho}^s 
 \le C_s r^{-s\pr{1+\eps}+2}, \quad \forall \, r, \rho > 0, \quad \forall s\in \pr{0, \tfrac{2}{1+\eps}}.
\end{equation}

Now we repeat the process for $\Ga^\rho$, using \eqref{eqB.45} in place of \eqref{eqB.47}.
For any $\si> 0$ and $s > 0$, we have
\begin{equation*}
\int_{\Om_r(y)} \abs{\Ga^\rho}^s 
\le \si^s \abs{\Om_r(y)} + \int_{\set{\abs{\Ga^\rho}>\si}} \abs{\Ga^\rho}^s.
\end{equation*}
By \eqref{eqB.45}, for $s \in \pr{0, \frac{2^* 2}{2 + 2^* \eps}}$ and $\rho > 0$, 
\begin{align*}
\int_{\set{|\Ga^\rho|>\si}} \abs{\Ga^\rho}^s 
&= \int_0^\infty s \tau^{s-1} \abs{\set{\abs{\Ga^\rho}>\max\set{\tau,\si}}} d\tau \\
&\le C \si^{-\frac{2^* 2}{2 + 2^* \eps}} \int_0^\si s \tau^{s-1} \, d \tau
+ C\int_\si^\infty s \tau^{s-1 - \frac{2^* 2}{2 + 2^* \eps}} \, d\tau \\
&= C \pr{1-\frac{s}{s-\frac{2^* 2}{2 + 2^* \eps}}} \si^{s - \frac{2^* 2}{2 + 2^* \eps}}.
\end{align*}
Taking $\si=r^{-\frac{2 + 2^* \eps}{2^*}}$, we conclude that
\begin{equation}
\label{eqB.49}
 \int_{\Om_r(y)} \abs{\Ga^\rho}^s
 \le C_s r^{-s\frac{2 + 2^* \eps}{2^*}+2}, \quad \forall \, r, \rho > 0, \quad \forall \, s \in \pr{0, \tfrac{2^* 2}{2 + 2^* \eps}}.
\end{equation}

Fix $s\in \brp{1, 2}$ and $\tilde s\in \brp{1, \iny}$.  
There exists $\eps \in \pr{0, 1}$ and $2^* \in \pr{2, \iny}$ so that $s < \frac{2}{1+\eps}$ and $\tilde s < \frac{2^* 2}{2 + 2^* \eps}$.
It follows from \eqref{eqB.48} and \eqref{eqB.49} that for any $r > 0$
\begin{equation}
\label{eqB.50}
\norm{\Ga^\rho}_{W^{1,s}\pr{\Omega_{r}\pr{y}}} \le C\pr{r} \mbox{ and } \norm{\Ga^\rho}_{L^{\tilde s}\pr{\Omega_{r}\pr{y}}} \le C\pr{r}\quad \text{uniformly in } 
\rho.
\end{equation}
Therefore, (using diagonalization) we can show that there exists a sequence $\set{\rho_\mu}_{\mu=1}^\infty$ tending to $0$ and a function $\Ga =\Ga_{y} = \Ga\pr{\cdot, y}$ such that
\begin{equation}
\label{eqB.51}
\Ga^{\rho_\mu} \rightharpoonup \Ga \quad \text{in } 
W^{1,s}\pr{\Om_{r}\pr{y}} \mbox{ and in } L^{\tilde s} \pr{\Om_{r}\pr{y}}, \mbox{ for all } r > 0.
\end{equation}

Furthermore, for fixed $r_0 < r$, \eqref{eqB.44} and \eqref{eqB.46} and that $\Om$ is bounded imply uniform bounds on $\Ga^{\rho_\mu}$ in $W^{1,2}\pr{\Om \setminus \Om_{r_0}\pr{y}}$ for small $\rho_\mu$.
Thus, there exists a subsequence of $\set{\rho_\mu}$ (which we will not rename)  and a function $\widetilde{\Ga}=\widetilde{\Ga}_{y} = \widetilde {\Ga}\pr{\cdot, y}$ such that
\begin{equation*}
\Ga^{\rho_\mu} \rightharpoonup \widetilde{\Ga} \quad \text{in } W^{1,2}\pr{\Om \setminus \Om_{r_0}\pr{y}}.
\end{equation*}
Since $\Ga \equiv \widetilde{\Ga}$ on $\Om_{r}\pr{y}\setminus \Om_{r_0}\pr{y}$, we can extend $\Ga$ to the entire $\Om$ by setting $\Ga=\widetilde{\Ga}$ on $\Om\setminus \Om_{r}\pr{y}$.  
For ease of notation, we call the extended function $\Ga$.  
Applying the diagonalization process again, we conclude that there exists a sequence $\rho_\mu \to 0$ and a function $\Ga$ on $\Omega$ such that for every $s\in \brp{1, 2}$ and $\tilde s\in \brp{1, \iny}$,
\begin{equation}
\label{eqB.52}
\Ga^{\rho_\mu} \rightharpoonup \Ga \quad \text{in } 
W^{1,s}\pr{\Om_{r}\pr{y}} \mbox{ and in }L^{\tilde s}\pr{\Om_{r}\pr{y}},
\end{equation}
and
\begin{equation}
\label{eqB.53}
\Ga^{\rho_{\mu}} \rightharpoonup \Ga \quad \text{in } W^{1,2}\pr{\Om \setminus 
\Om_{r_0}\pr{y}},
\end{equation}
for all $0 < r_0 < r$.

Let $\phi \in C_c^\infty\pr{\Om}$ and $r > 0$.
Choose $\eta \in C^\iny_c\pr{B_r\pr{y}}$ to be a cutoff function so that $\eta \equiv 1$ in $B_{r/2}\pr{y}$.
We write $\phi = \eta \phi + \pr{1 - \eta} \phi$.
By \eqref{eqB.25} and the definition of ${B}$,
\begin{align*}
\lim_{\mu \to \iny} \fint_{B_{\rho_\mu}\pr{y}} &\eta \phi
= \lim_{\mu\to \infty} {B}[\Ga^{\rho_\mu}_y,\eta \phi]
= \lim_{\mu\to \infty} \int_\Om A \gr {\Ga^{\rho_\mu}_y} \cdot \gr \pr{\eta \phi} + W \cdot \gr {\Ga^{\rho_\mu}_y} \, {\eta \phi}.
\end{align*}
Note that $\eta \phi$ and $D\pr{\eta \phi}$ belong to $C^\iny_c\pr{\Om_r\pr{y}}$.  
From this and the boundedness of $A$ given by \eqref{ABd}, it follows that there exists a $s^\prime > 2$ such that each $a_{ij} D_i\pr{\eta \phi}$ belongs to $L^{s^\prime}\pr{\Om_r\pr{y}}$.
Since $W \in L^q\pr{\Om}$ for some $q \in \pb{2, \iny}$, then $W \, \eta \phi \in L^{q}\pr{\Om_r\pr{y}}$.
Therefore, by \eqref{eqB.52},
\begin{align}
\lim_{\mu \to \iny} \fint_{B_{\rho_\mu}\pr{y}} \eta \phi
&= \int_\Om A \gr {\Ga_y} \cdot \gr \pr{\eta \phi} + W \cdot \gr {\Ga_y} \, {\eta \phi}
= {B}[\Ga_{y},\eta \phi] .
\label{eqB.54}
\end{align}
Another application of \eqref{eqB.25} shows that
\begin{align*}
\lim_{\mu \to \iny} \fint_{B_{\rho_\mu}\pr{y}} \pr{1 - \eta} \phi
&= \lim_{\mu\to \infty} \int_\Om A \gr {\Ga^{\rho_\mu}_y} \cdot \gr \brac{\pr{1 -\eta} \phi} + W \cdot \gr {\Ga^{\rho_\mu}_y} \, \pr{1 -\eta} \phi.
\end{align*}
Since $\phi \in C_c^\infty\pr{\Om}$ and $\eta \in C^\iny_c\pr{B_r\pr{y}}$, then $\pr{1 - \eta}\phi$ and $D\brac{\pr{1 - \eta}\phi}$ belong to $C_c^\infty(\Om \setminus B_{r/2}\pr{y})$.
In combination with \eqref{ABd}, this implies that each $a_{ij} D_i\brac{\pr{1 - \eta} \phi}$ belongs to $L^{2}\pr{\Om \setminus B_{r/2}\pr{y}}$.
H\"older's inequality and that $\Om$ is bounded implies that $W \pr{1 - \eta} \phi$ belongs to $L^{2}\pr{\Om \setminus B_{r/2}\pr{y}}$ as well.
Therefore, it follows from \eqref{eqB.53} that
\begin{align}
\lim_{\mu \to \iny} \fint_{B_{\rho_\mu}\pr{y}} \pr{1 - \eta} \phi
&= \int_\Om A \gr {\Ga_y} \cdot \gr \brac{\pr{1 -\eta} \phi} + W \cdot \gr {\Ga_y} \, \pr{1 -\eta} \phi
= {B}[\Ga_{y},\pr{1-\eta} \phi] .
\label{eqB.55}
\end{align}
Upon combining \eqref{eqB.54} and \eqref{eqB.55}, we see that for any $\phi \in C^\iny_c\pr{\Om}$,
\begin{align*}
\phi\pr{y}
&= \lim_{\mu \to \iny} \fint_{B_{\rho_\mu}\pr{y}} \phi
= \lim_{\mu \to \iny} \fint_{B_{\rho_\mu}\pr{y}} \eta \phi
+ \lim_{\mu \to \iny} \fint_{B_{\rho_\mu}\pr{y}} \pr{1 - \eta} \phi \nonumber \\
&= {B}[\Ga_{y},\eta \phi] 
+  {B}[\Ga_{y},\pr{1-\eta} \phi] 
=  {B}[\Ga_{y}, \phi].
\end{align*}
That is, for any $\phi \in C_c^\infty\pr{\Om}$, 
$${B}\brac{\Ga_y, \phi}= \phi\pr{y}$$
and $\Ga$ satisfies property \eqref{eqB.12} in the definition of the Green's function.

As before, for $\ell_0 \in \pb{1, \iny}$ and $m_0 \in \pb{2, \iny}$, we take $f\in L^{\ell_0}\pr{\Om}$, $G \in L^{m_0}\pr{\Om}$ and let $u_1, u_2 \in W^{1,2}_0\pr{\Om}$ be the unique weak solutions to ${\widetilde L}^* u_1 = f$ and ${\widetilde L}^* u_2 = - \di G$.
That is, $u_1$ and $u_2 \in W^{1,2}_0\pr{\Om}$ satisfy \eqref{eqB.29} and \eqref{eqB.30}, respectively, so that with $u := u_1 + u_2$,
$$B^*\brac{u, w} = \int_{\Om} w \, f + Dw \cdot G, \quad \forall w \in W^{1,2}_0\pr{\Om}.$$
Then for a.e. $y\in \Om$, 
\begin{align}
u(y)
=\lim_{\mu\to \infty} \fint_{B_{\rho_\mu}\pr{y}} u
= \lim_{\mu\to \infty} {B}\brac{\Ga^{\rho_\mu}_y, u}
&= \lim_{\mu\to \infty} {B}^*\brac{u, \Ga^{\rho_\mu}_y}
= \lim_{\mu\to \infty} \int_{\Om} \Ga^{\rho_\mu} f + D \Ga^{\rho_\mu} \cdot G
\label{eqB.56}
\end{align}
where we have used \eqref{eqB.33} - \eqref{eqB.35}.

Let $\eta \in C^\iny_c\pr{B_r\pr{y}}$ be as defined in the previous paragraph.
Then $\eta f \in L^{\ell_0}\pr{B_r\pr{y}}$.
Since $\Om$ is bounded, then $f \in L^\ell\pr{\Om}$ for some $\ell \in \pr{1, 2}$ and it follows that $\pr{1 - \eta}f  \in L^\ell\pr{\Om \setminus B_{r/2}\pr{y} }$.
Equation \eqref{eqB.53} in combination with a Sobolev inequality implies that for all $0 < r_0 < r$,
\begin{equation*}
\Ga^{\rho_{\mu}} \rightharpoonup \Ga \quad \text{in } L^{\ell'}\pr{\Om \setminus \Om_{r_0}\pr{y}},
\end{equation*}
where $\ell' \in \pr{2, \iny}$ denotes the H\"older conjugate to $\ell$.
Consequently, using the property above and \eqref{eqB.52} with $\tilde s = \ell_0'$ shows that
\begin{align*}
\lim_{\mu\to \infty} \int_{\Om} \Ga^{\rho_\mu} f
&= \lim_{\mu\to \infty}  \int_{B_r\pr{y}} \Ga^{\rho_\mu} \eta f
+ \lim_{\mu\to \infty}  \int_{\Om \setminus B_{r/2}\pr{y}} \Ga^{\rho_\mu} \pr{1 - \eta} f \\
&= \int_{B_r\pr{y}} \Ga \eta f
+  \int_{\Om \setminus B_{r/2}\pr{y}} \Ga \pr{1 - \eta} f
= \int_{\Om} \Ga \, f.
\end{align*}
Since $m_0 > 2$, then $m_0' \in \pr{1, 2}$ and then according to \eqref{eqB.52} we can pair $\eta G$ with $D \Ga^{\rho_\mu}$ in $B_r\pr{y}$ and take the limit.
As $\Om$ is bounded, then $G \in L^2\pr{\Om}$ so that $\pr{1 - \eta} G \in L^2\pr{\Om \setminus B_{r/2}\pr{y} }$.
With the aid of \eqref{eqB.52} with $s = m_0'$ and \eqref{eqB.53}, we see that
\begin{align*}
\lim_{\mu\to \infty} \int_{\Om} D\Ga^{\rho_\mu} \cdot G
&= \lim_{\mu\to \infty}  \int_{B_r\pr{y}} D\Ga^{\rho_\mu} \cdot \eta G 
+ \lim_{\mu\to \infty}  \int_{\Om \setminus B_{r/2}\pr{y}} D\Ga^{\rho_\mu} \cdot \pr{1 - \eta} G \\
&= \int_{B_r\pr{y}} D\Ga \cdot \eta G 
+  \int_{\Om \setminus B_{r/2}\pr{y}} D\Ga \cdot \pr{1 - \eta} G
= \int_{\Om} D\Ga \cdot G.
\end{align*}
Combining the last two equations with \eqref{eqB.56} gives \eqref{eqB.14}.
Property \eqref{eqB.15} follows as well.

The first part of each of the estimates \eqref{eqB.12}--\eqref{eqB.16} follow almost directly by passage to the limit and recalling that we use the notation $\Ga = \Ga_y = \Ga\pr{\cdot, y}$. 
Indeed, for any $r > 0$ and any $g \in L^\infty\pr{\Om_r\pr{y}}$, \eqref{eqB.49} implies that for any $s \in \brp{1, \iny}$
$$\abs{\int_{\Om} \Ga \, g}
= \lim_{\mu\to \infty} \abs{\int_{\Om} \Ga^{\rho_\mu} g}
\le C_s r^{-\eps+\frac{2}{s}} \norm{g}_{L^{s'}\pr{\Om_r\pr{y}}},$$
where $\eps > 0$ is arbitrarily small and $s'$ is the H\"older conjugate exponent of $s$.  
By duality, we obtain that for every $s \in \brp{1, \iny}$ and $r > 0$,
$$\norm{\Ga\pr{\cdot, y}}_{L^s\pr{\Om_r\pr{y}}} 
\le C_s r^{-\eps+\frac{2}{s}},$$
that is, the first part of \eqref{eqB.17} holds. 
A similar argument using \eqref{eqB.48}, \eqref{eqB.44} and \eqref{eqB.46} yields the first parts of \eqref{eqB.18} and \eqref{eqB.16}, respectively. 
Now, as in the proofs of  \eqref{eqB.45} and \eqref{eqB.47}, the first part of \eqref{eqB.16} gives the first parts of \eqref{eqB.19} and \eqref{eqB.20}.

Passing to the proof of \eqref{eqB.21}, fix $x \ne y$. For a.e. $x \in \Om$, the Lebesgue differentiation theorem implies that
\begin{align*}
\Ga\pr{x} 
&= \lim_{\de \to 0^+} \fint_{\Om_\de\pr{x}} \Ga
= \lim_{\de \to 0^+} \frac{1}{\abs{\Om_\de}} \int \Ga \, \chi_{\Om_\de\pr{x}},
\end{align*}
where $\chi$ denotes an indicator function.
Assuming as we may that $2\de \le \min\set{d_x, \abs{x - y}}$, it follows that $\chi_{\Om_\de\pr{x}} = \chi_{B_\de\pr{x}} \in L^{2^*}\pr{\Om\setminus \Om_{\de}\pr{y}}$ for any $2^* \in \pr{2, \iny}$, where $d_x=\mbox{dist}(x,\partial\Omega)$.  
Therefore, \eqref{eqB.52} implies that
\begin{align*}
\frac{1}{\abs{B_\de}}  \int \Ga \, \chi_{B_\de\pr{x}}
&= \lim_{\mu \to \iny} \frac{1}{\abs{B_\de}} \int \Ga^{\rho_\mu} \, \chi_{B_\de\pr{x}}
= \lim_{\mu \to \iny} \fint_{B_\de\pr{x}} \Ga^{\rho_\mu}.
\end{align*}
If $\rho_\mu \le \frac 1 3 \abs{x - y}$, $\rho_\mu<d_y$,  then \eqref{eqB.40} implies that for a.e. $z \in B_\de\pr{x}$
\begin{align*}
\abs{\Ga^{\rho_\mu}\pr{z}}
\le C \abs{z-y}^{-\eps},
\end{align*}
where $C$ is independent of $\rho_\mu$.
Since $\abs{z - y} >\frac{1}{2} \abs{x - y}$ for every $z \in B_\de\pr{x} \su B_{\abs{x-y}/2}\pr{x}$, then
\begin{align*}
\norm{\Ga^{\rho_\mu}}_{L^\iny\pr{B_\de\pr{x}}} \le C \abs{x - y}^{-\eps}.
\end{align*}
By combining with the observations above, we see that for a.e. $x \in \Om$,
\begin{align*}
\Ga\pr{x, y} 
&= \lim_{\de \to 0^+} \frac{1}{\abs{\Om_\de}} \int \Ga \, \chi_{\Om_\de\pr{x}}
= \lim_{\de \to 0^+}  \lim_{\mu \to \iny} \fint_{B_\de\pr{x}} \Ga^{\rho_\mu}
\le \lim_{\de \to 0^+}  \lim_{\mu \to \iny} C \abs{x - y}^{-\eps}
= C \abs{x - y}^{-\eps},
\end{align*}
which is \eqref{eqB.21}.

Now we have to prove that $\Ga(\cdot, y)= 0$ on $\del \Om$ in the sense that for all $\zeta \in C_c^\iny\pr{\Om}$ satisfying $\zeta\equiv 1$ on $B_r(y)$ for some $r>0$, equation \eqref{eqB.13} holds.
By Mazur's lemma, $W^{1,2}_0\pr{\Om}$ is weakly closed in $W^{1,2}\pr{\Om}$.  
Therefore, since $(1-\zeta)\Ga^{\rho_{\mu}}= \Ga^{\rho_{\mu}}- \zeta \Ga^{\rho_{\mu}}  \in W^{1,2}_0\pr{\Om}$ for all $\rho_\mu > 0$, it suffices for \eqref{eqB.13} to show that
\begin{equation}
\label{eqB.57}
(1-\zeta)\Ga^{\rho_\mu} \rightharpoonup (1-\zeta) \Ga \quad \text{in } W^{1,2}\pr{\Om}.
\end{equation}
Since $(1-\zeta)\equiv 0$ on $B_r(y)$, the result \eqref{eqB.57} follows from 
\eqref{eqB.53}.  
Indeed,
\begin{align*}
\int_{\Om} (1-\zeta) \Ga(\cdot, y)\phi 
&=\int_{\Om} \Ga (\cdot, y)(1-\zeta) \phi 
= \lim_{\mu\to \iny} \int_{\Om} \Ga^{\rho_\mu}(\cdot, y) (1-\zeta) \phi \\
&=\lim_{\mu\to \iny} \int_{\Om} (1-\zeta) \Ga^{\rho_\mu} (\cdot, y) \phi, \quad 
\forall \phi \in L^{2}\pr{\Om}, \quad \text{and} \\
\int_{\Om} D\brac{\pr{1-\zeta}\Ga(\cdot, y)}\cdot \psi 
&= -\int_{\Om} \Ga(\cdot, y) D\zeta \cdot \psi + \int_{\Om} D\Ga(\cdot, y) \cdot (1-\zeta)\psi \\
&=-\lim_{\mu\to \iny} \int_{\Om} \Ga^{\rho_\mu} (\cdot, y) D\zeta \cdot \psi + \lim_{\mu \to \iny} \int_{\Om} D\Ga^{\rho_\mu} (\cdot, y) \cdot (1-\zeta) \psi \\
&=\lim_{\mu \to \iny} \int_{\Om} D\brac{\pr{1-\zeta}\Ga^{\rho_\mu} (\cdot, y)}\cdot \psi, \quad 
\forall \psi \in L^2\pr{\Om}^2,
\end{align*}
so that \eqref{eqB.13} holds.
Since $\Ga(x,y)$ satisfies \eqref{eqB.12} -- \eqref{eqB.15}, it is the unique Green's function associated to $\widetilde{L}$.

Fix $x, \, y \in \Om$ and $0<R\le R_0< |x-y|$.  
Then $\widetilde{L}\Ga\pr{\cdot, y}=0$ on $B_{R_0}\pr{x}$.  
Therefore, by H\"older continuity of solutions described by \eqref{eqB.11} and the pointwise bound \eqref{eqB.21}, whenever $\abs{x-z} < \frac{R}{2}$ we have
\begin{align*}
\abs{\Ga\pr{x,y} - \Ga\pr{z,y}} 
&\le C_{R_0} \pr{\frac{|x-z|}{R}}^\eta C \norm{\Ga(\cdot, y)}_{L^\infty\pr{B_{R}(x)}}
\le C_{R_0} C \pr{\frac{|x-z|}{R}}^\eta R^{-\eps}.
\end{align*}
This is the H\"older continuity of $\Ga(\cdot, y)$ described by \eqref{eqB.22}.

Using the pointwise bound on $\Ga^\rho$ in place of those for $\Ga$, a similar statement holds for $\Ga^\rho$ with $\rho\leq \frac 38 |x-y|$, and it follows that for any compact set $K \Subset \Om\setminus\set{y}$, the sequence $\set{\Ga^{\rho_\mu}\pr{\cdot, y}}_{\mu = 1}^\iny$ is equicontinuous on $K$.
Furthermore, for any such $K \Subset \Om\setminus\set{y}$, there are constants 
$C_K < \iny$ and $\rho_K > 0$ such that for all $\rho < \rho_K$,
\begin{align*}
\norm{\Ga^{\rho}\pr{\cdot, y}}_{L^\infty\pr{K}} \le C_K.
\end{align*}
Passing to a subsequence if necessary, we have that for any such compact $K 
\Subset \Om\setminus\set{y}$,
\begin{equation}
\Ga^{\rho_\mu}(\cdot, y) \to \Ga\pr{\cdot,  y}
\label{eqB.58}
\end{equation}
uniformly on $K$.

We now aim to show 
\begin{equation*}
\Ga\pr{x,y} = \Ga^*\pr{y,x},
\end{equation*}
where $\Ga^*$ is the Green's function associated to $\widetilde{L}^*$.
Let $\widehat \Ga^{\si}=\widehat \Ga^{\si}_{x} = \widehat \Ga^{\si}\pr{x, \cdot}$ denote the averaged function associated to $\widetilde{L}^*$ at the point $x \in \Om$.
That is, we follow the procedure from above that was used to construct $\Ga^\rho_y$, except that we work with the adjoint operator $\widetilde{L}^*$ and consider the function in terms of the variable $y$ centred at the point $x \in \Om$.
The resulting function is $\widehat \Ga^{\si}$.
  
By the same arguments used for $\Ga^\rho$, we obtain a sequence $\set{\si_{\nu}}_{\nu =1}^\iny$, $\si_\nu \to 0$, such that $\widehat \Ga^{\si_\nu}\pr{\cdot, x}=\widehat \Ga^{\si_\nu}_x$ converges to $\Ga^*\pr{\cdot, x}$ uniformly on compact subsets of $\Om \setminus \set{x}$, where $\Ga^*\pr{\cdot, x}$  is a Green's function for $\widetilde{L}^*$ that satisfies the properties analogous to those for $\Ga\pr{\cdot, y}$.  
In particular, $\Ga^*\pr{\cdot, x}$ is H\"older continuous.

By \eqref{eqB.25}, for $\rho_\mu$ and $\sigma_\nu$ sufficiently small,
\begin{equation}
 \fint_{B_{\rho_\mu}\pr{y}} \widehat \Ga^{\sigma_\nu} \pr{\cdot, x} 
 = {B}\brac{ \Ga^{\rho_\mu}_y, \widehat \Ga^{\si_\nu}_x}
 = {B}^*\brac{\widehat \Ga^{\si_\nu}_x, \Ga^{\rho_\mu}_y}
 = \fint_{B_{\si_\nu}\pr{x}} \Ga^{\rho_\mu}(\cdot, y).
 \label{eqB.59}
\end{equation}
Define
$$g_{\mu \nu}
:= \fint_{B_{\rho_\mu}\pr{y}} \widehat \Ga^{\si_\nu}\pr{\cdot, x} 
=  \fint_{B_{\si_\nu}\pr{x}} \Ga^{\rho_\mu}\pr{\cdot, y} .$$
By continuity of $\Ga^{\rho_\mu}\pr{\cdot, y}$, it follows that for any $x \ne y \in \Om$,
$$\lim_{\nu \to \iny} g_{\mu \nu}
= \lim_{\nu \to \iny}  \fint_{B_{\rho_\mu}\pr{y}} \widehat \Ga^{\si_\nu}\pr{\cdot, x} 
=  \Ga^{\rho_\mu}\pr{x, y},$$
so that by \eqref{eqB.58},
$$\lim_{\mu \to \iny}\lim_{\nu \to \iny} g_{\mu \nu}
= \lim_{\mu \to \iny} \Ga^{\rho_\mu}\pr{x, y} 
= \Ga\pr{x, y}.$$
But by weak convergence in $W^{1,s}\pr{B_r\pr{y}}$, i.e., \eqref{eqB.52},
$$\lim_{\nu \to \iny} g_{\mu \nu}
= \lim_{\nu \to \iny} \fint_{B_{\rho_\mu}\pr{y}} \widehat \Ga^{\si_{\nu}}\pr{\cdot, x} 
= \fint_{B_{\rho_\mu}\pr{y}} \Ga^* \pr{\cdot, x},$$
and it follows then by continuity of $\Ga^*\pr{\cdot, x}$ that
$$\lim_{\mu \to \iny}\lim_{\nu \to \iny} g_{\mu \nu}
= \lim_{\mu \to \iny} \fint_{B_{\rho_\mu}\pr{y}} \Ga^* \pr{\cdot, x} 
= \Ga^* \pr{y, x}.$$
Therefore, for all ${x \ne y}$,
\begin{equation}
\Ga\pr{x,y} = \Ga^*\pr{y,x}.
\label{eqB.60}
\end{equation}
Consequently, all the estimates which hold for $\Ga\pr{\cdot, y}$ hold analogously for $\Ga\pr{x, \cdot}$ and the proof is complete.
\end{pf}

\begin{rem}
We have seen that there is a subsequence $\set{\rho_{\mu}}_{\mu = 1}^ \iny$, $\rho_\mu \to 0$, such that $\Ga^{\rho_\mu}\pr{x, y} \to \Ga\pr{x, y}$ for all $x\in \Om \setminus \set{y}$. 
In fact, a stronger result can be proved.  
By \eqref{eqB.59},
\begin{align*}
\Ga^{\rho}\pr{x,y}
= \lim_{\nu \to \iny} \fint_{B_{\si_\nu}\pr{x}} \Ga^\rho\pr{\cdot, y} 
= \lim_{\nu \to \iny} \fint_{B_\rho\pr{y}} \widehat \Ga^{\si_{\nu}}\pr{\cdot, x} 
= \fint_{B_\rho\pr{y}} \Ga^{ \, *}\pr{\cdot, x}. 
\end{align*}
By \eqref{eqB.60}, this gives
$$\Ga^{\rho}\pr{x,y}  = \fint_{B_\rho\pr{y}} \Ga\pr{x, z}dz.$$
By continuity, for all $x \ne y$,
\begin{equation*}
\lim_{\rho \to 0} \Ga^\rho\pr{x,y} = \Ga\pr{x,y}.
\end{equation*}
\end{rem}

\end{appendix}

\def\cprime{$'$}

%

\end{document}